\documentclass[10pt, reqno]{amsart}
\usepackage{fullpage}
\usepackage{array}
\usepackage{textcmds}  % amsrefs needs this, but it has to be loaded early to avoid re-defs.
\usepackage{amsmath, amssymb, amsfonts, amstext, verbatim, amsthm, mathrsfs}
\usepackage[mathcal]{eucal}
\usepackage{microtype}
\usepackage{graphicx}
\usepackage[all]{xy}
\usepackage[modulo]{lineno}
\usepackage[usenames]{color}
\usepackage{url}
\usepackage[colorlinks=true,linkcolor=blue,citecolor=blue,urlcolor=blue,citebordercolor={0 0 1},urlbordercolor={0 0 1},linkbordercolor={0 0 1}]{hyperref} %needs to be loaded after most things
\usepackage[nameinlink]{cleveref}
\usepackage[shortalphabetic]{amsrefs} %needs to be loaded after hyperref
\usepackage{enumitem}
\usepackage{xspace}

\theoremstyle{plain}

\newtheorem{thm}{Theorem}[subsection]
\newtheorem*{thm*}{Theorem}

\newtheorem{lem}[thm]{Lemma}

\newtheorem{cor}[thm]{Corollary}

\newtheorem{prop}[thm]{Proposition}
\newtheorem{quest}[thm]{Question}

\theoremstyle{definition}

\newtheorem{rem}[thm]{Remark}
\newtheorem{defn}[thm]{Definition}

\newtheorem{ex}[thm]{Example}

\newcommand{\res}[2]{\left. #1 \right|_{#2}}
\newcommand{\CAlg}{\mathrm{CAlg}}
\newcommand{\Sp}{\mathcal{S}}

\newcommand{\oh}[1]{\widehat{#1}}
\newcommand{\sq}[1]{\widetilde{#1}}

\newcommand{\C}{\mathcal{C}}

\renewcommand{\O}{\mathcal{O}}

\renewcommand{\H}{H}
\newcommand{\D}{\mathcal{D}}

\newcommand{\SCR}{SCR}
\newcommand{\kk}{\mathbf{k}}

\let\vinograd\gg% Just in case we redefine it
\newcommand{\bA}{\mathbf{A}}
\newcommand{\cA}{\mathcal{A}}

\newcommand{\sA}{\mathscr{A}}

\newcommand{\cB}{\mathcal{B}}

\newcommand{\cC}{\mathcal{C}}

\newcommand{\bF}{\mathbf{F}}
\newcommand{\cF}{\mathcal{F}}

\newcommand{\sF}{\mathscr{F}}
\newcommand{\bG}{\mathbf{G}}
\newcommand{\cG}{\mathcal{G}}

\newcommand{\bL}{\mathbf{L}}

\newcommand{\cM}{\mathcal{M}}
\newcommand{\fM}{\mathfrak{M}}

\newcommand{\sM}{\mathscr{M}}

\newcommand{\cO}{\mathcal{O}}

\newcommand{\sO}{\mathscr{O}}
\newcommand{\bP}{\mathbf{P}}

\newcommand{\sU}{\mathscr{U}}

\newcommand{\sV}{\mathscr{V}}

\newcommand{\cW}{\mathcal{W}}

\newcommand{\cX}{\mathcal{X}}

\newcommand{\bZ}{\mathbf{Z}}

\newcommand{\fm}{\mathfrak{m}}
\newcommand{\id}{\mathop{{\rm id}}\nolimits}
\newcommand{\pt}{\ast}

\newcommand{\heart}{\heartsuit}
\newcommand{\isom}{\simeq}
\DeclareMathOperator{\Cech}{Cech}

\DeclareMathOperator{\Tot}{Tot}
\DeclareMathOperator{\Fun}{Fun}
\DeclareMathOperator{\fib}{fib}
\DeclareMathOperator{\cofib}{cone}

\DeclareMathOperator{\im}{im}
\newcommand{\ilim}{\mathop{\varprojlim}\limits}
\newcommand{\dlim}{\mathop{\varinjlim}\limits}

\renewcommand{\mod}{\text{\rm{-mod}}}  %This obviously conflicts with a certain other notation

\renewcommand{\bZ}{{\mathbb Z}}

\newcommand{\DD}{\mathbb{D}}
\newcommand{\ZZ}{\mathbb{Z}}

\newcommand{\RGamma}{R\Gamma}
\newcommand{\op}[1]{\!\!\mathop{\rm ~#1}\nolimits}

\newcommand{\Ind}{\op{Ind}}
\newcommand{\Spec}{\op{Spec}}
\newcommand{\Spf}{\op{Spf}}
\newcommand{\RHom}{\op{RHom}}
\newcommand{\Hom}{\op{Hom}}
\newcommand{\Map}{\op{Map}}
\newcommand{\Ext}{\op{Ext}}

\newcommand{\Perf}{\op{Perf}}
\newcommand{\APerf}{\op{APerf}}
\newcommand{\DDAPerf}{\DD\APerf}
\newcommand{\DCoh}{\op{DCoh}}

\newcommand{\inner}[1]{\underline{#1}}
\newcommand{\Gm}{{\mathbb{G}_m}}

\newcommand{\dual}{\vee}
\newcommand{\QC}{QC}

\renewcommand{\L}{L}
\newcommand{\I}{\mathscr{I}}

\newcommand{\X}{\mathscr{X}}
\newcommand{\Y}{\mathscr{Y}}
\newcommand{\Z}{\mathscr{Z}}
\renewcommand{\S}{\mathscr{S}}

\newcommand{\colim}{\varinjlim}

\newcommand{\Coh}[1][]{\op{Coh}^{#1}}

\DeclareMathOperator{\Bl}{Bl}

\newcommand{\CA}{{\hyperref[property:CA]{(CA)}}\xspace}
\newcommand{\CD}[1][]{{\hyperref[property:CD]{(CD)${}_{#1}$}}\xspace}
\newcommand{\CP}[1][]{{\hyperref[property:CP]{(CP)${}_{#1}$}}\xspace}
\newcommand{\GEheart}[1][]{{\hyperref[property:GE-heart]{($\mathrm{GE}^{\heart}$)${}_{#1}$}}\xspace}
\newcommand{\GE}[1][]{{\hyperref[property:GE]{(GE)${}_{#1}$}}\xspace}
\newcommand{\pGE}[1][]{{\hyperref[property:pGE]{(pGE)${}_{#1}$}}\xspace}
\newcommand{\LL}[1][]{{\hyperref[property:L]{(L)${}_{#1}$}}\xspace}

\newcommand{\val}[1]{\sharp(#1)}

\makeatletter
\newcommand{\customlabel}[2]{
\protected@write \@auxout {}{\string \newlabel {#1}{{#2}{}}}}
\makeatother

\newcommand{\ps}[1]{[\![#1]\!]}

\begin{document}

\title{Mapping stacks and categorical notions of properness}
\author{Daniel Halpern-Leistner \and Anatoly Preygel}

\begin{abstract}
One fundamental consequence of a scheme $X$ being proper is that the functor classifying maps from $X$ to any other suitably nice scheme or algebraic stack is representable by an algebraic stack. This result has been generalized by replacing $X$ with a proper algebraic stack. We show, however, that it also holds when $X$ is replaced by many examples of algebraic stacks which are not proper, including many global quotient stacks. This leads us to revisit the definition of properness for stacks.

We introduce the notion of a formally proper morphism of stacks and study its properties. We develop methods for establishing formal properness in a large class of examples. Along the way, we prove strong h-descent results which hold in the setting of derived algebraic geometry but not in classical algebraic geometry. Our main applications are algebraicity results for mapping stacks and the stack of coherent sheaves on a flat and formally proper stack.
\end{abstract}

\maketitle

\tableofcontents

%\linenumbers

\section{Introduction}

If $X$ is a projective variety and $Y$ is a quasi-projective variety over a field $k$, then it is well-known that there is a quasi-projective variety $\Map(X,Y)$ parameterizing morphisms from $X$ to $Y$. The properness hypothesis on $X$ is necessary:

\begin{ex} \label{E:maps_A1}
The functor $\Map(\bA^1_k,\bA^1_k)$ parameterizing families of maps from $\bA^1_k$ to itself is not representable by an algebraic space: If it were, then any compatible family of $k\ps{x} / x^n$-points for $n \geq 1$ would extend to a unique $k\ps{x}$-point. But letting $t$ denote the coordinate on $\bA^1$, the compatible family of morphisms $e^{tx} \colon \bA^1_{k[x]/x^n} \to \bA^1_k$ does not extend to a morphism $\bA^1_{k\ps{x}} \to \bA^1_{k}$, because $e^{tx}$ is not a polynomial in $t$ unless $x$ is nilpotent.
%. The formula $e^{tx}$ defines a family of morphisms which does not extend, because $e^{tx}$ is not a polynomial in $t$ unless $x$ is nilpotent.

\emph{On the other hand}, if $Y$ is equipped with a linearizable $\bG_m$-action, then the functor parameterizing families of $\bG_m$-equivariant maps is representable by a scheme, $\Map(\bA^1_k,Y)^{\bG_m}$. The scheme $\Map(\bA^1_k,Y)^{\bG_m}$ is the disjoint union of the Bia\l{}ynicki-Birula strata \cite{Bialynicki-Birula} for the action of $\bG_m$ on $Y$.
\end{ex}

This illustrates a common phenomenon in equivariant algebraic geometry (see, for instance, \cites{HaimanSturmfels, AlexeevBrion}): certain non-proper schemes behave as if they are proper if one takes into account equivariance. This paper is dedicated to the systematic development of this concept, and its consequences.

In the example above, we will deduce the representability of $\inner{\Map}(\bA^1_k,Y)^{\bG_m}$ from the fact that the functor $\inner{\Map}(\bA^1_k / \bG_{m,k}, Y / \bG_{m,k})$ parameterizing flat families of morphisms of algebraic stacks is itself representable by an algebraic stack (see \Cref{ex:BB_strata} below). However, the stack $\bA^1_k / \bG_{m,k}$ is still not proper in the sense of \cite{LMB00} -- i.e., separated, finite type, and universally closed. The problem is that points in a separated stack must have proper automorphism groups, so the quotient of a scheme $X$ by an algebraic group $G$ can not be proper unless $G$ acts with finite stabilizers on $X$.

Our goal is thus to propose an alternative to the notion of properness for an algebraic stack $\X$ which applies to many of the quotient stacks which arise in practice, such as the stack $\bA^1/\bG_m$. The resulting notion has convenient formal properties, and we establish two of the useful consequences which are familiar from the case of schemes: $\Map(\X,\Y)$ is an algebraic stack, as is the stack of coherent sheaves on $\X$, when $\X$ is ``formally proper'' and flat over the base.

\subsection{Definition of a formally proper morphism}
\label{ssec:results_mapping}

Our notion of properness makes use of the derived category of ``almost perfect complexes'' on a stack, $\APerf(\X)$. When $\X$ is noetherian and classical, $\APerf(\X)$ agrees with the derived category $D^-\Coh(\X)$ of right-bounded complexes of quasi-coherent sheaves with coherent homology, and more generally $\APerf(-)$ extends the definition of pseudo-coherent complexes \cite{stacks-project}*{\href{https://stacks.math.columbia.edu/tag/08FS}{Tag 08FS}} to derived schemes and stacks.

We will also use the theory of formal completion, which we review in \Cref{section:completions}. If $\X$ is an algebraic stack and $Z \subset |\X|$ a co-compact closed subset, we will denote the formal completion of $\X$ along $Z$ by $\oh{\X}$.
\begin{defn}\label{defn:complete-closed-immersion}
  We say that $\X$ is \emph{complete along $Z$} if the restriction functor $\APerf(\X) \to \APerf(\oh{\X})$ is an equivalence of $\infty$-categories.  In this case we say that $\Z \to \X$ is a \emph{complete closed immersion} for any closed substack $\Z \subset \X$ having underlying closed subset $Z$.
\end{defn}

\begin{ex} Let $\X = \Spec(R)$ for some noetherian ring $R$, and let $Z \subset \Spec(R)$ be the closed subset corresponding to an ideal $I \subset R$. Then by \Cref{lem:classical_spf} we have
\[
\APerf(\X) = D^-\Coh(R), \qquad \APerf(\oh{\X}) = \ilim_n D^-\Coh(R/I^n),
\]
and $\X$ is complete along $Z$ if and only if $R$ is $I$-adically complete. 
\end{ex}

Although this notion of completeness is formulated for derived stacks, we will see in \Cref{lem:completeness_nil} that it only depends on the underlying classical stack, and in fact only on its underlying reduced stack.

\begin{defn} \label{def:cohomologically_proper}
We say that a morphism of algebraic (derived) stacks $\pi \colon \X \to \S$ is \emph{formally proper} if it is (almost) finitely presented and for any noetherian algebraic (derived) stack $\Y$ which is complete along a closed subset $Z \subset |\Y|$ and any morphism $\Y \to \S$, the fiber product $\X^\prime := \X \times_\S \Y$ is complete along the preimage of $Z$.
\end{defn}

Say $\Y = \Spec(R)$ for a complete noetherian ring $R$ and $Z$ is the closed subset defined by an ideal $I \subset R$. Then $\X' = \X \times_\S \Y$ being complete along the preimage of $Z$ is equivalent to the restriction map giving an equivalence
\[
D^-\Coh (\X') \xrightarrow{\simeq} \ilim_n D^-\Coh(\X \times_\S \Spec(R/I^n)).
\]
When $\X \to \S$ is representable by proper algebraic spaces, then this equivalence is the Grothendieck existence theorem, or formal GAGA, for derived categories \cite{Knutson}*{Chap.~V}. Therefore \Cref{def:cohomologically_proper} asserts that $\pi$ universally satisfies a strong form of the Grothendieck existence theorem.

At first blush, it is strange to elevate a major theorem into a definition. On the other hand, the definition of a proper scheme is also somewhat complicated: checking that a morphism is universally closed, or even checking the valuative criterion, is a priori an infinite and intractable task. We argue that the value of \Cref{def:cohomologically_proper} (as with any definition!) results from the availability of interesting examples, the usefulness of its formal properties, and the theorems one can prove with it.

As a first source of examples, we generalize the notion of a projective morphism of schemes to that of a \emph{cohomologically projective} morphism of derived algebraic stacks in \Cref{defn:coh_projective}, and we prove:
\begin{thm*}[\Cref{thm:coh_projective_is_proper}] A cohomologically projective morphism of (derived) algebraic stacks is formally proper.
\end{thm*}
The proof is formally similar to the classical proof of the Grothendieck existence theorem for projective morphisms, but the framework of derived algebraic geometry offers some significant conceptual simplifications.

\begin{ex}[\Cref{prop:projective_over_affine_quotient}]
Let $S$ be a noetherian scheme over a field $k$, $G$ a linearly reductive algebraic $k$-group acting on a $k$-scheme $X$, and $\pi : X \to S$ a $G$-invariant morphism which factors as a projective morphism $X \to X'$ followed by an affine morphism $X' \to S$ of finite type. If $\pi$ admits a relatively ample $G$-equivariant invertible sheaf and $H_0 \pi_\ast(\cO_X)^G$ is a coherent $\cO_S$-module, then $X/G \to S$ is cohomologically projective.
\end{ex}

\begin{ex}[\Cref{prop:good-moduli-projective}]
If $S$ is a noetherian scheme over a field $k$, $G$ is an algebraic $k$-group, and $X$ is a $G$-scheme which admits a good quotient \cite{seshadri} which is projective over $S$, then $X/G$ is cohomologically projective over $S$. This generalizes the formal GAGA result for good moduli spaces established in \cite{GZB}.
\end{ex}

Formally proper morphisms also enjoy some convenient formal properties:

\begin{prop} \label{prop:coverings}
Say $\X' \to \X$ is a surjective and formally proper morphism of (derived) $\S$-stacks, then if $\X' \to \S$ is formally proper, so is $\X \to \S$. In addition, formally proper morphisms are closed under base change, composition, and fiber products. 
\end{prop}
\begin{proof}
The first claim follows from \Cref{thm:descent-apGE-new} (which holds under weaker hypotheses on the map $\X' \to \X$) and \Cref{T:cp_fully_faithful}, and the remaining claims are immediate from \Cref{def:cohomologically_proper}.
\end{proof}

As a result, any stack $\X$ which admits a surjective proper morphism $\X' \to \X$ from a cohomologically projective stack $\X'$ is formally proper. This allows one to construct many examples of formally proper morphisms which are not cohomologically projective. For instance, if $\cG$ is a reductive group scheme over a scheme $S$, then $B_S \cG \to S$ is formally proper under suitable hypotheses (\Cref{prop:BG_proper}). See \Cref{ex:many_coh_proper} below for more examples.

One key technical result underlying \Cref{prop:coverings} and many other results is a new $h$-descent theorem in the context of derived algebraic geometry. Recall an $h$-cover is an (almost) finitely presented morphism which is universally a topological submersion. Examples include fppf\footnote{In the derived context, we say that a morphism is $fppf$ if it is surjective, flat, and almost finitely presented. A better abbriviation might be fpppf, for ``fid\`element \`a plat et presque de pr\'esentation finie.'' But there is no risk of confusion: whereas a finitely presented classical ring homomorphism need not be almost finitely presented when regarded as a map of simplicial commutative rings, a flat and finitely presented ring homomorphism arises as the derived base change of a flat and finitely presented homomorphism of noetherian rings and is therefore almost finitely presented.} morphisms, as well as proper and formally proper surjective morphisms.
\begin{thm*}[\Cref{thm:descent-aperf}] The presheaf $\APerf(-)$ satisfies derived descent for $h$-covers of locally noetherian derived stacks.
\end{thm*}
Although many of the results in this paper are classical in nature, this result truly requires the setting of derived algebraic geometry. For instance the inclusion of the reduced subscheme $X^{red} \to X$ is a derived $h$-cover, but it has a trivial Cech nerve in the category of classical schemes and therefore this is not an effective descent morphism in this category. On the other hand, there are classical stacks for which this descent theorem provides our only method of establishing formal properness.

Finally, we note that \Cref{def:cohomologically_proper} is closely related to other plausible notions of properness for algebraic stacks. By \Cref{T:cp_fully_faithful} and \Cref{P:universally_closed}, we have the following implications for an (almost) finitely presented morphism $\pi : \X \to \S$ of noetherian algebraic (derived) stacks:
\[
\xymatrix{ *+[F]\txt{$\pi$ is formally\\proper} \ar@{=>}[r] & *+[F]\txt{Property \CP:\\universally, $R^i \pi_\ast(-)$ preserves\\coherent sheaves (\Cref{def:define_CP})} \ar@{=>}[r] & *+[F]\txt{$\pi$ is universally\\closed} } 
\]
Furthermore, all of these properties hold if $\S$ is a noetherian scheme and $\pi$ is proper, by \Cref{ex:many_coh_proper} (1). However, we will see in \Cref{E:bad_example} that \CP alone is not sufficient for our applications to mapping stacks and coherent sheaves.

\subsection{Applications}

Recall that given two stacks $\X$ and $\Y$ over a base stack $\S$, the mapping stack $\inner{\Map}_\S(\X,\Y)$ is the stack whose space of $T$-points (for any scheme $T$ over $\S$) is the space of maps $\inner{\Map}_\S (T \times_\S \X, \Y)$.

\begin{thm*} [\Cref{thm:derived_mapping_stacks}]
Let $R$ be a (simplicial commutative) $G$-ring, and let $\X$ be an algebraic (derived) $R$-stack which is formally proper and of finite Tor amplitude $n$ over $\S:=\Spec(R)$. Then for any algebraic (derived) stack $\Y$ which is locally almost finitely presented with quasi-affine diagonal over $\S$, $\inner{\Map}_{\S}(\X,\Y)$ is an algebraic $(n+1)$-stack, locally almost finitely presented over $\S$. If $\pi$ is flat, then $\inner{\Map}_\S(\X,\Y)$ has quasi-affine diagonal, and the diagonal is affine if $\Y \to \S$ has affine diagonal.
\end{thm*}
\begin{rem}
This theorem holds under slightly weaker hypotheses. For an elaboration, see the discussion following \Cref{thm:derived_mapping_stacks}.
\end{rem}

\begin{ex} \label{ex:theta_stability}
Let $\Theta := \bA^1 / \bG_m$, then for any locally (almost) finitely presented morphism of (derived) algebraic stacks $\Y \to \S$ with quasi-affine (resp. affine) diagonal, $\inner{\Map}_\S(\Theta_\S,\Y)$ is a locally almost finitely presented algebraic $\S$-stack with quasi-affine (resp. affine) diagonal (the lack of noetherian hypotheses is explained in \Cref{rem:noetherian}). This is the starting point for the theory of $\Theta$-stability developed in \cite{HLInstability}, where $\inner{\Map}_\S(\Theta_\S,\Y)$ is regarded as the stack of ``filtered points'' in $\Y$. The fibers of the map $\inner{\Map}_\S(\Theta_\S,\Y) \to \Y$ defined by restriction along the inclusion $\{1\} \to \Theta$ are regarded as generalized flag spaces for points of $\Y$.
\end{ex}

\begin{ex}[BB-strata] \label{ex:BB_strata}
As a special case of the previous example, consider a quasi-separated algebraic space $X$ locally of finite presentation over a scheme $S$, and let $(\bG_m)_S$ act on $X$. Then we have a cartesian square:
\[
\xymatrix{\inner{\Map}_S(\bA^1_S,X)^{\bG_m} \ar[r] \ar[d] & \inner{\Map}_S(\Theta_S, X/(\bG_m)_S) \ar[d]^{X/\bG_m \to B\bG_m} \\ S \ar[r]^-s & \inner{\Map}_S(\Theta_S,B(\bG_m)_S) },
\]
where $s$ classifies the projection $(\bA^1/\bG_m)_S \to (B\bG_m)_S$. So the fact that the mapping stacks on the right side are algebraic implies that the functor $\Map_S(\bA^1_S,X)^{\bG_m}$ parameterizing flat families of equivariant maps from $\bA^1$ to $X$ is algebraic. This generalizes the classical theorem of Bia\l{}ynicki-Birula, and more recent work by Drinfeld \cite{drinfeld2013algebraic}.
\end{ex}

\begin{ex} If $S$ is a scheme and $G$ and $H$ are smooth affine $S$-group-schemes, the fiber of the map $\inner{\Map}_S(BG, BH) \to \inner{\Map}_S(\pt,BH) \simeq BH$ over the tautological $S$-point of $BH$ is the functor $\Hom_{gp/S}(G, H)$ parameterizing families of group homomorphisms. If $G$ is a reductive group scheme, then $\inner{\Map}_S(BG, BH)$ is algebraic with affine diagonal by \Cref{prop:BG_proper} and \Cref{thm:derived_mapping_stacks} (working \'etale locally over $S$ and using relative noetherian approximation). This recovers the deep result in SGA3 Exp. XXIV (Cor. 7.2.3) that $\Hom_{gp/S}(G,H)$ is representable by algebraic spaces with affine diagonal.
\end{ex}

\begin{ex} \label{E:bad_example}
To illustrate why formal properness is the correct hypothesis in \Cref{thm:derived_mapping_stacks}, say $S=\Spec(k)$ is a field of characteristic $0$. Then $B\bG_a$ satisfies the ``coherent pushforward'' property \CP, but $\Hom_{gp/S}(\bG_a, \bG_m)$, and thus $\inner{\Map}_k(B\bG_a,B\bG_m)$ is \emph{not} algebraic. If it were, then every compatible family of $k[x]/x^n$-points would necessarily come from a $k\ps{x}$-point, and the counterexample of \Cref{E:maps_A1} provides a counterexample here as well. Specifically, letting $t$ denote the coordinate on $\bG_a$, the compatible family of group homomorphisms $e^{tx} \colon (\bG_a)_{k[x]/x^n} \to (\bG_m)_{k[x]/x^n}$ does not extend to a morphism of group schemes $(\bG_a)_{k\ps{x}} \to (\bG_m)_{k\ps{x}}$.%since $\exp(t x)$ evidently does not lie in $k[t] \otimes_k k\ps{x} \subset \ilim_n k[t] \otimes_k k[x]/x^n$.
\end{ex}

As another application of the notion of formal properness, we generalize the well-known fact that the stack of coherent sheaves on a proper scheme is algebraic.
\begin{thm*}[\Cref{thm:moduli}]
Let $R$ be a simplicial commutative $G$-ring, and let $\X$ be an algebraic derived stack which is flat and formally proper over $\Spec(R)$. Then the stack $\underline{\Coh}_{\X/R}$ parameterizing flat families of coherent sheaves (see \Cref{def:stack_sheaves}) is an algebraic derived stack locally almost finitely presented and with affine diagonal over $\Spec(R)$.
\end{thm*}

\subsubsection{Comparison with previous results}\label{ssec:comparison-olsson}

Artin's criteria have been used in several papers to establish the algebraicity of the mapping stack $\inner{\Map}_S(\X,\Y)$, where $S$ is an algebraic space and $\X$ and $\Y$ are algebraic stacks, but always under the hypothesis that $\X \to S$ is proper -- see \cite{Ol06}*{Theorem 1.1}, \cite{Lieblich}*{Section 2.3}, and \cite{hall2014coherent}*{Theorem 1.2} for the strongest result in this direction.\footnote{The results \cite{Lieblich}*{Section 2.3} seem to be missing the hypothesis that $Y$ is flat over $S$. Similarly Aoki has claimed that if $\X / S$ is proper, and $\Y$ is locally finite presentation and $\Y$ is separated or $\Y = B \Gm$ \cites{Ao06a,Ao06b} -- however there appear to be some serious errors in that paper which are not addressed by the erratum. We thank David Rydh for pointing this out to us.} Our approach to Artin's criteria, via the Tannakian formalism, is highly indebted to the algebraicity result \cite{DAG-XIV}*{Thm.~3.2.1} in the context of derived (and spectral) algebraic geometry, which implies algebraicity under the hypothesis that $\X \to S$ is a flat and proper spectral algebraic space, and $\Y$ is a quasi-compact spectral Deligne-Mumford stack with affine diagonal. Our main conceptual contribution is relaxing the hypothesis that the map $\X \to S$ is proper.

In \cite{alper2015luna}, algebraicity is established when $\X \to S$ is a flat good moduli space morphism, using a similar approach, but working in the classical context. A good moduli space morphism is, locally over $S$, cohomologically projective, so this follows from \Cref{thm:derived_mapping_stacks} as well when $\Y$ has quasi-affine diagonal, but \cite{alper2015luna} makes use of a stronger version of Tannaka duality in the classical context \cite{hall2014coherent} which allows one to only assume that $\Y$ has affine stabilizer groups (see \Cref{rem:classical_vs_derived}). The paper \cite{alper2015luna} was developed independently of our work, and was released around the same time as the original version of this paper.

\subsection{Context, conventions, and author's note}
\label{ssec:notation}

Unless we explicitly state otherwise, all of our categories will be $\infty$-categories. The model for $\infty$-categories we have in mind is that of quasi-categories \cite{HigherTopos}, and our model for the $\infty$-category of $\infty$-groupoids or ``spaces'' will be Kan simplicial sets, which we denote $\Sp$. The reader may freely substitute their favorite models for each.

\subsubsection{Context} We will mostly work in the context of derived algebraic geometry, i.e. algebraic geometry over the $\infty$-site of simplicial commutative rings, which we denote $\SCR$. Many of our results hold in the context of spectral algebraic geometry, i.e. algebraic geometry over the $\infty$-site of connective $E_\infty$-algebras $\CAlg^{cn}$. For the most part, we indicate the intended context in the statements of our propositions, and indicate any modifications of the proof required in the different contexts. However, we only establish the algebraicity of mapping stacks, \Cref{thm:derived_mapping_stacks} in the derived context, because at the time of this writing the spectral version of Artin's criteria for stacks does not yet appear in \cite{SAG}.

\subsubsection{Notation} We briefly summarize some of the terminology for the objects we use throughout the paper:
\begin{center}
\begin{tabular}{>{\raggedright\arraybackslash}p{.28\textwidth}p{.62\textwidth}}
\emph{derived prestack}: & any functor $F : \SCR \to \Sp$\\ \hline
\emph{derived $n$-stack}: & a prestack which is local for the \'{e}tale topology on $\SCR$ and such that the $\infty$-groupoid $F(R)$ is $n$-truncated when $R$ is discrete ($n=1$ by default)\\ \hline
\emph{algebraic derived stack}, $\X$: & a derived $1$-stack which admits a surjective morphism $U \to \X$ such that $U$ is a disjoint union of affine derived schemes and the morphism is relatively representable by smooth derived algebraic spaces\\ \hline
\emph{locally noetherian algebraic derived stack}: & $U$ is a disjoint union of noetherian simplicial commutative rings\\ \hline
\emph{qc.qs.}: & quasi-compact and quasi-separated\footnote{This is a special case of the notion of $\infty$-quasi-compact which is an inductive and relative notion: Every map of affine schemes is $\infty$-quasi-compact; a map of functors is $\infty$-quasi-compact if and only if its base-change to every affine scheme is so; and a higher stack $\X/\Spec(R)$ is $\infty$-quasi-compact if it admits an affine atlas $U = \Spec(A) \to \X$ such that $U \times_\X U/\Spec(R)$ is $\infty$-quasi-compact.  If $\X$ is an $n$-stack for some finite $n$, then this is really a finitary condition since high enough diagonals of $\X$ are isomorphisms.}\\ \hline
\emph{noetherian algebraic derived stack}: & $\X$ is qc.qs. and locally noetherian\\
\end{tabular}
\end{center}
We also have the following notions for a morphism $\X \to \Y$ of derived stacks:
\begin{center}
\begin{tabular}{>{\raggedright\arraybackslash}p{.28\textwidth}p{.62\textwidth}}
\emph{relatively algebraic}: & for any map $\Spec(A) \to \Y$, $\X \times_\Y \Spec(A)$ is algebraic \\ \hline
\emph{qc.qs.}: & $\X \times_\Y \Spec(A)$ is algebraic and qc.qs. \\ \hline
\emph{almost finitely presented}: & $\X \times_\Y \Spec(A)$ is quasi-separated and admits a smooth, representable surjection $\Spec(B) \to \X \times_\Y \Spec(A)$ such that $\Spec(B) \to \Spec(A)$ exhibits $B$ as an almost finitely presented $A$-algebra.\\
\end{tabular}
\end{center}

For any prestack $\X$, we let $\QC(\X)$ denote the stable $\infty$-category of quasi-coherent sheaves on $\X$, which is the Kan extension of the functor $\QC : \SCR \to \widehat{\op{Cat}}_{\infty}$ which takes a simplicial commutative ring to its category of (dg-) modules. The same construction allows one to define the category of almost perfect complexes $\APerf(\X)$, perfect complexes $\Perf(\X)$, $n$-finitely presented complexes $\Coh[n](\X)$, etc. We recall these constructions and their basic properties in \autoref{section:qc}.

In addition to working in a derived context, we depart from the usual algebro-geometric literature in some potentially confusing notational conventions:
\begin{enumerate}
  \item We think of our $t$-structures as \emph{homologically indexed}, and write $H_i$ for $H^{-i}$ and e.g., $\tau_{\leq i}$ for $\tau^{\geq -i}$ and $\C_{\leq i}$ for $\C^{\geq -i}$.  In particular, ``bounded above'' will mean \emph{homologically bounded above} (i.e., lying in $\C_{\leq i}$ for some $i$).
  \item We implicitly work with $\infty$-categorical enhancements of various triangulated categories of sheaves.
  \item The symbols $f_*$, $f^*$, etc. will, unless otherwise stated, denote the functors \emph{of $\infty$-categories}. We do not include extra decorations to indicate that they are ``derived,'' and will instead sometimes write e.g., $H_0 \circ   f_*$ for the functor on abelian categories.  (The one exception: We will write $\RGamma$ for global sections, as a reminder that this is global sections of sheaves of spectra and not spaces.)
  \item We use $\Map$ for a mapping \emph{simplicial set} in an $\infty$-category $\cC$. If $\cC$ is stable, we use $\RHom$ to denote any stable enrichment (e.g., in spectra, in chain complexes, or in complexes of sheaves). Finally, we use $\Hom$ to denote the maps in the homotopy category of $\C$.
\end{enumerate}

\subsubsection{Comparison with the classical context}
There are natural functors of $\infty$-categories $\op{Ring} \to \SCR \to \CAlg^{cn}$, which are neither fully faithful nor essentially surjective. Nevertheless, one can left Kan extend stacks along these functors, and left Kan extension preserves algebraic stacks and their formal completions. Also, the formation of $\infty$-categories $\QC(-)$, $\APerf(-),\Perf(-),$ and $\QC(-)^{cn}$ commutes with left Kan extension along these functors (see \Cref{section:qc}). It follows that many statements which are proved for $\QC(-)$ for spectral algebraic stacks automatically imply the corresponding statement for $\QC(-)$ applied to derived or classical stacks. For instance, one can formulate \Cref{defn:complete-closed-immersion}, of a complete closed immersion, in all three contexts, but a classical or derived stack is complete along a closed subset if and only if it is complete when regarded as a spectral stack via left Kan extension.

\subsubsection{Author's note}

We would like to thank Jacob Lurie for some helpful conversations in the early stages of this project. We would also like to thank Jarod Alper, Bhargav Bhatt, Brian Conrad, Jack Hall, Johan de Jong, Dmitry Kubrak, Jon Pridham, and David Rydh for helpful conversations and comments. Shortly after the first version of this paper was made public, Jack Hall, David Rydh, and Jarod Alper released a paper proving similar results on algebraicity of mapping stacks, and these results are strengthened in their followup work \cite{alper2019etale}. Their work was developed independently and concurrently, and we have benefited from very helpful conversations with them as well as an early draft of their followup paper \cite{alper2019etale} while revising this paper. A. Preygel was supported by an NSF Postdoctoral Fellowship, and D. Halpern-Leistner was supported by the NSF postdoctoral fellowship DMS-1303960 and the NSF grant DMS-1601976.

\subsubsection{Further questions}
We show that formally proper morphisms have several of the common and useful properties of proper maps, but we would like to highlight some natural questions to inspire future research:
\begin{quest} Is the property of being formally proper local for the \'etale topology on the target? Also, in \Cref{def:cohomologically_proper}, does it suffice to only consider $\Y = \Spec(R)$ for complete local noetherian rings $R$ over $\S$?
\end{quest}

\begin{quest}
If $\pi : \X \to \Spec(A)$ is formally proper, is there a model for $\X$ over a finitely generated subring $A' \subset A$ which is also formally proper?
\end{quest}

\begin{quest}
If $\X$ is smooth and formally proper over a field of characteristic $0$, does the Hodge-de-Rham spectral sequence for $\X$ degenerate?
\end{quest}

%----------------------------------------------------------
%----------------------------------------------------------
%----------------------------------------------------------

\section{Background and first properties of complete closed immersions}

\subsection{Recollections on formal completions}
\label{section:completions}

Here we review some of the key results we will use from \cite{DAG-XII}*{Sect.~4,5}, noting how they can be slightly strengthened in the context of simplicial commutative rings.

\begin{defn}[\cite{DAG-XII}*{5.1.1}] Suppose that $\X$ is a prestack.  Let $|\X|$ denote the underlying Zariski topological space of points of $\X$. We say that a closed subset $Z \subset |\X|$ is \emph{co-compact} if the inclusion of the complement of $Z$ is a quasi-compact open immersion. Given a co-compact $Z \subset |\X|$, define the \emph{formal completion of $\X$ along $Z$} to be the following prestack
  \[ \oh{\X}(R) := \oh{\X}_Z(R) := \left\{ \eta \in \X(R) \colon \text{such that $\eta : |\Spec(R)| \to |\X|$ factors through $Z$}\right\} \]
\end{defn}

Note that as $|\Spec(\pi_0(R))| \simeq |\Spec(R)|$, $|\X|$ only depends on the underlying classical stack of $\X$. If $\X$ is an algebraic derived stack and $\I \subset H_0(\O_\X)$ is a locally finitely generated ideal sheaf, then $\I$ defines a cocompact closed substack $\Z(\I)$ of the underlying classical stack of $\X$ and hence a cocompact closed subset $Z \subset |\X|$. By abuse of terminology we will call $\oh{\X}_Z$ the completion of $\X$ along $\I$. When $\X$ (and hence $\Z(\I)$) is affine, we have an explicit description of $\oh{X}$:

\begin{prop}\label{prop:cplt-affine} Suppose that $R \in \SCR$ and $I \subset \pi_0(R)$ is a finitely generated ideal.  Let $\X = \Spec(R)$ and $\oh{\X} = \Spf(R)$ be its completion along $I$. Then there exists a tower 
  \[ \cdots \to R_2 \to R_1 \to R_0 \]
  in $\SCR_R$  such that there is an equivalence of prestacks over $\X$
  $$\Spf(R) = \dlim_n \Spec(R_n).$$
  Furthermore, we may suppose that $\pi_0(R_n) \to \pi_0(R_{n-1})$ is surjective for each $i$, and that each $R_n$ is perfect as $R$-module with Tor-amplitude uniformly bounded by the number of generators of $I$.
\end{prop}
\begin{rem} \label{rem:cplt-affine-spectral}
\cite{DAG-XII}*{Lemma~5.1.5} give an analogous statement for $R \in \CAlg$. The only difference in the $E_\infty$ statement is that $R_n$ are almost perfect as $R$ modules. This is, essentially, the one statement in this paper that depends strongly on the fact that we are working with simplicial commutative rings rather than $E_\infty$ rings.
\end{rem}

\begin{proof} The proof of \cite{DAG-XII}*{Lemma~5.1.5} goes over essentially unchanged in simplicial commutative rings, but that now the rings $R_n$ can be assumed to be \emph{perfect} rather than merely \emph{almost perfect}: This follows from the construction of the algebras denote $A(x)_n$ in op.cit., and the fact that in the universal example $\ZZ$ is a perfect module over $\ZZ[x]$ of Tor dimension explicitly bounded by $1$, thanks to the Koszul resolution.   
\end{proof}

Suppose that $S = \Spec(R)$, $I \subset \pi_0(R)$ a finitely generated ideal, and $\oh{S}$ the associated completion.  Let $\pi \colon \X \to S$ be an $S$-stack, $Z \subset |\X|$ the preimage of $Z(I) \subset |S|$, and $\oh{\X}$ the associated completion. Then the natural map
  \[ 
    \oh{\X} \stackrel\sim\longrightarrow \X \times_S \oh{S}  
  \] 
  is an equivalence (just use the functor-of-points descriptions on both sides).  Furthermore, if $\{R_n\}$ is a tower as in \Cref{prop:cplt-affine}, then there is an equivalence
  \[ 
    \oh{\X} \stackrel\sim\longleftarrow \dlim_n \X \times_S \Spec(R_n) 
  \]
  since fiber products preserve filtered colimits in presheaves.  We will generally write write $i_n \colon \X_n = \X \times_S \Spec(R_n) \to \X$ for the base-changed closed immersions.

A natural question is how this compares to the classical notion of completion.  To this end, we have:
\begin{prop} \label{lem:classical_spf}Suppose that $R$ is a noetherian \emph{classical} commutative ring and that $I \subset \pi_0(R)$.  Let $\Spf(R)$ denote the formal completion of the derived stack $\Spec(R)$ along $I$, and let $\Spf^{cl} R$ denote Kan extension of the (classical) prestack $\dlim_n \Spec(R/I^n)$.  Then the natural morphism
  \[ F \colon \Spf^{cl} R \longrightarrow \Spf(R) \]
  is an equivalence.
\end{prop}
\begin{proof} Note first that the natural maps $\Spec(R/I^n) \to \Spec(R)$ factor uniquely through $\Spf(R)$ -- this determines the natural morphism of the proposition.  
  Let $R_n$ be the Koszul-type algebra killing the $n$-th powers of a finite set of generators for the ideal $I \subset R$ (in dg-language this would be $R[B_1,...,B_r]/dB_i = f_i^n$) -- it satisfies the conditions of \Cref{prop:cplt-affine} by the proof of \cite{DAG-XII}*{Lemma~5.1.5}.  Note that $\Spf^{cl} R \isom \dlim \Spec(\pi_0(R_n))$ since $I^{(n-1)r+1} \subset (f_1^n, \ldots,f_r^n) \subset I^n$.  We must thus show that the natural map
 \[ \dlim \Spec(\pi_0(R_n)) \longrightarrow \dlim \Spec(R_n) \]
 is an equivalence.  There is almost an argument for this in \cite{DAG-XII}*{Lemma~5.2.17} -- we must show that the map of pro-objects in (almost perfect) commutative $R$-algebras $ \{R_n\} \to \{ \pi_0(R_n) \}$  is a pro-equivalence.  By the Lemma, we have that
  \[ \{ \tau_{\leq k} R_n \} \longrightarrow \{ \pi_0(R_n) \} \]
is an equivalence of pro-objects in (almost perfect) $R$-modules for all $k$. By the bound on the Tor amplitude in \Cref{prop:cplt-affine}, taking $k>r+1$ establishes the pro-equivalence at the level of $R$-modules. Furthermore, the module level statement does in fact imply the algebra statement (i.e., potential issues with pro-algebras vs algebras in pro-objects don't get in the way) because we may restrict to $(r+1)$-truncated algebras (c.f., the proof of \cite{DAG-XII}*{Lemma~6.3.3}).

\end{proof}

Since fppf morphisms are topological quotient morphisms, we can use this to deduce:
\begin{cor} Suppose that $\X$ is a noetherian classical stack, and that $\I \subset \O_\X$ is an ideal sheaf. Then the formal completion of the Kan extension of $\X$ (from classical rings to derived rings) along $\I$ is the Kan extension of the classical completion of $\X$ along $\I$.
\end{cor}

\begin{rem}
The rings $R_n$ appearing in \Cref{prop:cplt-affine} are \emph{not} unique, in contrast to the $R/I^n$ used to from classical completions -- in particular, they need not globalize. This is related to the fact that our notion of completion $\oh{\X}_Z$ depends only on the underlying subset $Z$ and not on the choice of structure sheaf etc.
  
There \emph{is} a notion of completion along a closed immersion $\Z \to \X$ in derived algebraic geometry more analogous to the usual $\I$-adic completion, i.e., it gives rise to a canonically defined pro-algebra $\oh{\O_\X} = ``\ilim_n" \cO_{\X_n}$ such that $H_0(\O_{\X_n}) = H_0(\O_\X)/\I^n$ where $\I = \ker\{ H_0(\O_\X) \to H_0(\O_\Z)\}$.   However this construction is somewhat involved, especially outside of characteristic zero, and since we do not need it we will not discuss it further.
\end{rem}

\subsubsection{Almost perfect complexes and coherent sheaves}

In this section we will use the $\infty$-category of almost perfect complexes $\APerf(-)$, as well as the ordinary category $\Coh(-)$, for both a derived stack and its formal completion. We refer the reader to \Cref{section:qc} for a reminder on how these categories are defined for arbitrary prestacks.

\begin{lem}\label{lem:artin-rees-cat}[Derived Artin-Rees Lemma] Suppose $R$ is a noetherian simplicial commutative ring, $I \subset \pi_0(R)$ is a finitely generated ideal, and $\oh{R}$ the derived $I$-adic completion of $R$.   Then the natural restriction functor 
\[ \APerf(\oh{R}) \longrightarrow \APerf(\Spf(R))\]
is a $t$-exact equivalence of $\infty$-categories, and hence induces an equivalence of abelian categories $\Coh(\oh{R}) \simeq \Coh(\Spf(R))$.
\end{lem}
\begin{proof} See \cite{DAG-XII}*{Section 4}.
\end{proof}

This globalizes as follows:

\begin{prop}\label{prop:formal-coh-props} Suppose $\X$ is a locally noetherian algebraic derived stack, $Z \subset |\X|$ a co-compact closed subset, and $\oh{\X}$ the completion of $\X$ along $Z$. Furthermore, let $f : \X' \to \X$ be a flat morphism from another locally noetherian algebraic derived stack and $Z' := f^{-1}(Z)$. Then,  \begin{enumerate}
  \item The (ordinary) category $\Coh(\oh{\X})$ is abelian. The induced pullback functor $\Coh(\oh{\X}_{Z}) \to \Coh(\oh{\X'}_{Z'})$ is exact. The exactness of sequences in $\Coh(\oh{\X})$ may be checked on a flat cover.
   \item The $\infty$-category $\APerf(\oh{\X})$ admits a $t$-structure. The induced pullback functor $\APerf(\oh{\X}_Z) \to \APerf(\oh{\X'}_{Z'})$ is $t$-exact.  The property of being connective/co-connective in $\APerf(\oh{\X})$ may be checked on a flat cover of $\X$.  In the affine case, the $t$-structure is described in \Cref{lem:artin-rees-cat}.
    \item  The heart of the $t$-structure in (2) identifies with $\Coh(\oh{\X})$.  More generally, the $\infty$-category of connective $n$-truncated objects $\APerf(\oh{\X})_{\leq n}^{cn}$ naturally identifies with $\Coh[n](\oh{\X})$.
    \item The $t$-structure on $\APerf(\oh{\X})$ is left $t$-complete, and (if $\X$ is quasi-compact) right $t$-bounded.  There is a natural equivalence
      \[ \APerf(\oh{\X})^{cn} \stackrel\sim\longrightarrow \ilim_n \APerf(\oh{\X})^{cn}_{\leq n} \isom \ilim_n \Coh[n](\oh{\X}) \]
  \end{enumerate}
\end{prop}
\begin{proof}  For (1), see e.g., \cite{Conrad-FormalGAGA}*{Sect.~1-2} -- note that this is a statement at the level of classical stacks.
  
  For (2), note first that the claim is fppf local on $\X$: See \cite{DAG-XII}*{Prop.~5.2.4, Rem.~5.2.13} for the case where $\X$ is Deligne-Mumford and \'etale descent.  It remains to show that the $t$-structure is, in fact, fppf local in the affine case: Note that the proofs of Lemmas 5.2.5 and 5.2.7 of op.cit. applies verbatim with \'etale replaced by fppf.

  For (3) and (4), we apply the last points of the following Lemma.
\end{proof}

We record some convenient facts on left $t$-complete $t$-structures that we will use.
\begin{lem}\label{lem:random-t-stuff}\mbox{}
  \begin{enumerate}
        \item Suppose that $\C$ is a stable $\infty$-category with a $t$-structure.  Then, $\C$ is left $t$-complete if and only if the following condition holds:
          Suppose given a tower $\sF_0 \leftarrow \sF_1 \leftarrow \cdots$ in $\C$ such that for every $k \geq 0$, the tower $\tau_{\leq k}(\sF_n) \in \C_{\leq k}$ is eventually constant.  Then, the tower has an inverse limit $\sF$, and for every $k \geq 0$ the natural map $\tau_{\leq k} \sF \to \tau_{\leq k} \sF_n$ is an equivalence for $n \vinograd 0$ (depending on $k$).
        \item Suppose that $F \colon \C \to \D$ is an exact and right $t$-exact functor, and that $\{\sF_n\}$ is a tower in $\C$ satisfying the above conditions.  Then, the tower $\{F(\sF_n)\}$ in $\D$ satisfies the above conditions as well.
        \item Suppose given a limit diagram $i \mapsto \C_i$ and $\C \to \C_i$, of small stable $\infty$-categories with $t$-structures such that all the functors are exact and right $t$-exact.  If each $\C_i$ is left $t$-complete, then so is $\C$.
        \item Suppose given a diagram $i \mapsto \C_i$ of small stable $\infty$-categories with $t$-structures, such that all the functors are exact and $t$-exact.  Then, there is a unique $t$-structure on the limit $\C : = \ilim_i \C_i$ such that the natural functors $\C \to \C_i$ are $t$-exact.   In this case, for each $n \geq 0$ the natural functor
          \[ \C^{cn}_{\leq n} \longrightarrow \ilim_i (\C_i)^{cn}_{\leq n} \]
          is an equivalence.
    \end{enumerate}
\end{lem}
\begin{proof} For (1), note first that left $t$-completeness is equivalent to the assertion for the tower $\sF_n = \tau_{\leq n} \sF$. It remains to suppose that $\C$ is left $t$-complete and prove the desired condition.  Consider the double-tower $\{\tau_{\leq m} \sF_n\}_{m,n}$.  We will show that it has a limit, and evaluate this limit in two different ways ``rows-then-columns'' and the transpose.

  In one direction, we have that
  \[ \sF'_m := \ilim_n \tau_{\leq m} \sF_n  \]
  exists since the diagram is eventually constant by hypothesis.  Furthermore, $\sF'_m \isom \tau_{\leq m} \sF'_n$ for any $n \geq m$ by construction.  Thus
  \[ \sF := \ilim_m \ilim_n \tau_{\leq m} \sF_n = \ilim_m \sF'_m \]
  exists, and $\sF \to \tau_{\leq m} \sF'_m$ induces an equivalence on $\tau_{\leq m}$, since $\C$ is left $t$-complete.  So, the inverse limit over the whole double-tower exists and is also equal to $\sF$.  Computing this in the other direction, we note
  \[ \ilim_m \tau_{\leq m} \sF_n = \sF_n \]
  since $\C$ is left $t$-complete, so that
  \[ \sF \isom \ilim_n \ilim_m \tau_{\leq m} \sF_n = \ilim_n \sF_n. \]
  The assertion that $\tau_{\leq k} \sF \to \tau_{\leq k} \sF_n$ is an equivalence for $n \vinograd 0$ follows by comparing it to $\tau_{\leq k} \sF \to \tau_{\leq k} \sF'_n$ for $n \geq k$. This completes the proof of (1).

  For (2), it is enough to show that $F$ preserves $\tau_{\leq k}$-equivalences to $\tau_{\leq k}$-equivalences.  This follows from the fact that $F$ preserves extension sequences and $k$-connective objects.

  For (3), we apply the criterion in (1).  Notice that a putative limit diagram in $\C$ which gives a limit diagram in each $\C_i$ is itself a limit diagram (though the converse need not hold in general!), so that (2) completes the proof.

  For (4), note that the $t$-structure is characterized by stating that an object of $\C$ is connective (resp. co-connective) if and only if this is true of its image in each $\C_i$.  Since the transition functors are $t$-exact, the truncation functors on each $\C_i$ pass to the limit to provide truncation functors on $\C$.  It is thus straightforward to check that this is a $t$-structure, and the desired description of the connective $n$-truncated objects. Notice also that (3) and (4) are essentially contained in \cite{DAG-XII}*{Remark~5.2.9}.
 \end{proof}

\subsection{Theorem on formal functions in derived algebraic geometry}

Here we describe $\QC$ on a formal completion. This leads to a formulation of the Theorem on Formal functions as a base change result, \Cref{thm:formal_functions}, which applies to arbitrary morphisms of algebraic stacks and complexes of quasi-coherent sheaves.

\begin{lem}\label{lem:completion-bdd} Suppose $R \in \SCR$, $I \subset \pi_0(R)$ is a finitely generated ideal, $\oh{R}$ is the $I$-adic completion of $R$, and $i \colon \Spf(R) \to \Spec(R)$ is the inclusion.  Then, the composite functor
 \[ i_* i^* \colon \QC(\Spec(R)) \longrightarrow \QC(\Spec(R)) \]
 has finite left $t$-amplitude  (i.e., is left $t$-exact up to a shift).
\end{lem}
\begin{proof} Let $\{R_n\}$ be a tower of perfect $R$-algebras having Tor amplitude at most $d$ as in \Cref{prop:cplt-affine}.  Then, we may identify $\QC(\Spec(R)) \isom R\mod$ and $\QC(\Spf(R)) \isom \ilim_n R_n\mod$.  Under these identifications, the functor $i_* i^*$ is identified with
\[ M \mapsto i_* i^* (M) = \ilim_n \left(R_n \otimes_R M\right). \]
If $M \in R\mod_{<0}$ then $R_n \otimes_R M \in R\mod_{<d}$ by hypothesis on $R_n$.  %Thus $i_* i^* M \in R\mod_{<(d+1)}$ since $\ilim$ has left Tor amplitude at most $1$ (i.e., there is a $\lim^1$ but no more).
\end{proof}

We let $(R\mod)^{I-cplt}\subset R\mod$ denote the subcategory of $I$-complete modules \cite{DAG-XII}*{Defn.~4.2.1}.

\begin{lem}\label{lem:completion-equiv} Suppose $R \in \SCR$, $I \subset \pi_0(R)$ is a finitely generated ideal, and $\oh{R}$ is the derived $I$-adic completion of $R$.  Let $i \colon \Spf(\oh{R}) \to \Spec(R)$ be the natural inclusion.  Then the composite
  \[ (R\mod)^{I-cplt} \subset R\mod \isom \QC(\Spec(R)) \stackrel{i^*}\longrightarrow  \QC(\Spf(R)) \]
  is an equivalence.  In particular, the co-unit $i^* i_* \to \id$ is an equivalence and the unit $\id \to i_* i^*$ can be identified with the $I$-adic completion.
\end{lem}
\begin{proof}  Note that this is proved, under the additional hypothesis that we restrict to \emph{connective} objects on both sides, in \cite{DAG-XII}*{Lemma~5.1.10}.  We will prove this stronger result by exploiting our stronger assumptions in the form of \Cref{lem:completion-bdd}.

  We note that $i^*$ has a right adjoint $i_*$.  If we identify $\QC(\Spec(R)) = R\mod$ and $\QC(\Spf(R)) = \ilim_n R_n\mod$, then $i_*$ is given by the inverse limit in $R$-modules
  \[ i_*(\{M_n\}) = \ilim_n M_n \in R\mod. \]

    It is enough to prove two assertions:
  \begin{enumerate}
    \item If $M \in R\mod$, then the natural map 
      \[ M \to i_* i^* M \] identifies $i_* i^* M$ with the $I$-adic completion of $M$.  If $M$ is almost connective, this follows by \cite{DAG-XII}*{Remark~5.1.11}.  Since $R\mod$ is right $t$-complete we may write $M = \dlim \tau_{\geq -k} M$, so and it is enough to show that both $i_* i^*$ and the $I$-adic completion functor $M \mapsto \oh{M}$ preserve this directed colimit.  By \Cref{lem:random-t-stuff}, it is enough to show that both $i_* i^*$ and the $I$-adic completion functor have bounded left $t$-amplitude.  For $i_* i^*$ this is \Cref{lem:completion-bdd}, while for the $I$-adic completion this is \cite{DAG-XII}*{Remark~5.11}.  This shows that $M \to i_* i^* M$ is the $I$-adic completion for all $M \in R\mod$.
    \item Let $\{R_n\}$ be as in \Cref{prop:cplt-affine}.  We must show that if $\{M_n\} \in \ilim_n R_n\mod$ then the unit
      \[ i^* i_* \{M_n\} \longrightarrow \{M_n\} \] is an equivalence.  More concretely, we must show that for each $k$ the natural map
      \[ R_k \otimes_R \ilim_n M_n \longrightarrow M_k \]
      is an equivalence.  Since $R_k$ is perfect over $R$, $R_k \otimes_R$ commutes with limits, so we may identify this with the map
      \[ \ilim_n\left( R_k \otimes_R M_n\right) \isom \ilim_n \left( (R_k \otimes_R R_n) \otimes_{R_n} M_n \right) \longrightarrow M_k. \]

      Recall that $\Spf(\oh{R}) \isom \dlim_n \Spec(R_n)$ and that fiber products preserve filtered colimits of (pre-)sheaves, so that we obtain
      \[ \Spec(R_k) \times_{\Spec(R)} \Spf(\oh{R}) \isom \dlim_n \Spec(R_k \otimes_R R_n) \]
      Furthermore, $\Spec(R_k) \times_{\Spec(R)} \Spf(\oh{R}) \isom \Spec(R_k)$ since $\Spec(R_k) \to \Spec(R)$ factors through the monomorphism $\Spf(R) \to \Spec(R)$.  Thus, the natural map
      \[ M_k \longrightarrow \ilim_n \left( (R_k \otimes_R R_n) \otimes_{R_n} M_n \right)  \]
      is an equivalence by (i).  This completes the proof.\qedhere
  \end{enumerate}
\end{proof}

Finally, we deduce
\begin{thm}\label{thm:qc-spf} Suppose that $\X$ is a derived prestack and $Z \subset |\X|$ is a co-compact closed subset.  Let $i \colon \oh{\X} \to \X$ be the inclusion of the formal completion along $Z$.  Then,
  \begin{enumerate}
    \item $i^*$ admits a left adjoint $i_+$;
    \item The composite
      \[ \QC_Z(\X) \subset \QC(\X) \stackrel{i^*}\longrightarrow \QC(\oh{\X}) \]
      is an equivalence, with inverse given by $i_+$ (i.e., the essential image of $i_+$ is contained in $\QC_Z(\X)$);
    \item $i_+$ is of formation local on $\X$, and satisfies the projection formula (i.e., is a functor of $\QC(\X)$-module categories);
   \end{enumerate}
The composite $i_+ \circ i^*$ can be identified with the functor $\RGamma_Z$, which fits into a fiber sequence $\RGamma_Z \longrightarrow \id \longrightarrow j_* j^*$, where $j$ is the inclusion of the open complement of $Z$.
\end{thm}

\begin{rem} \label{rem:qc-spf-spectral}
For spectral prestacks the theorem holds as stated, but with the categories $\QC(\X)$, $\QC(\oh{\X})$ and $\QC_Z(\X)$ replaced with the corresponding categories of almost connective complexes, $\QC(\X)^{acn}$, $\QC(\oh{\X})^{acn}$, and $\QC_Z(\X)^{acn}$ -- see \cite{DAG-XII}*{Theorem~5.1.9}. The proof is essentially the same in the derived context, combined with \Cref{lem:completion-equiv}.
\end{rem}

\begin{proof} 
It is enough to prove the claims in case $\X$ is affine, for the claims are local by construction (i.e., because we asked that the formation of $i_+$ be local). 

Suppose now that $\X = \Spec(R)$ and $Z$ is cut out by a finitely generated ideal $I \subset \pi_0(R)$.  Let $(R\mod)^{I-nil} \subset R\mod$ denote the full subcategory of locally $I$-nilpotent modules.  By \cite{DAG-XII}*{Prop.~4.2.5} these are equivalent via the completion functor and the ``local cohomology'' functor $\Gamma_I$.  Notice that $i^*$ vanishes on $I$-local modules (i.e., those supported away from $Z$), so that 
\[ i^*(M) \isom i^*(\oh{M}) \isom i^*(\Gamma_I(M)) \]
It thus follows from \Cref{lem:completion-equiv} and the above mentioned equivalence that the composite
\[ (R\mod)^{I-nil} \subset R\mod \stackrel{i^*}\longrightarrow \QC(\Spf(R)) \]
is also an equivalence.  Let $i_+$ be the composite of an inverse to this functor with the inclusion $(R\mod)^{I-nil} \subset R\mod$ -- more explicitly, $i_+ = \Gamma_I \circ i_*$.  Since the inclusion is left adjoint to $\Gamma_I \colon R\mod \to (R\mod)^{I-nil}$ it follows that $i_+$ is left adjoint to $i^*$.

It is easy to check that $i_+$, so defined, is local on $R$: given a cartesian diagram
\[ \xymatrix{
\Spf(R') \ar[r]^{i'} \ar[d]^{\oh{\pi}} & \Spec(R') \ar[d]^\pi \\
\Spf(R) \ar[r]_i & \Spec(R)
} \]
we wish to check that
\[ (i')_+ \oh{\pi}^* (\sF) \longrightarrow \pi^* i_+ (\sF) \] is an equivalence for all $\sF \in \QC(\Spf(R))$.  But by the above we may suppose that $\sF = i^* \sF'$ for some $\sF \in (R\mod)^{I-nil}$, and then notice that $\pi^* \sF' \in (R'\mod)^{I'-nil}$.  The result is then immediate from the equivalence applied upstairs and downstairs.

Finally, the projection formula assertion follows form \cite{DAG-XII}*{4.1.22, 4.2.6}.
\end{proof}

As an application, we have a version of the Theorem on Formal Functions without properness or coherence hypotheses:

\begin{thm}[Theorem on Formal Functions] \label{thm:formal_functions} Suppose given a cartesian diagram of derived stacks for which $\QC$ is presentable (e.g. $\X$ and $\Y$ are algebraic)
 \[ \xymatrix{
 \oh{\X} \ar[d]^{\oh{\pi}} \ar[r]^{i'} & \X \ar[d]^{\pi}\\
 \oh{\Y} \ar[r]_i & \Y \\
 }, \]
where $\oh{\Y}$ is the completion of $\Y$ along a co-compact closed subset $Z \subset |\Y|$.  (It follows that $\oh{\X}$ is the completion of $\X$ along $\pi^{-1}(Z)$.)  Then the base-change map
 \[ i^* \pi_*(\sF) \longrightarrow \oh{\pi}_* (i')^*(\sF) \]
 is an equivalence for all $\sF \in \QC(\X)$.

 In case $\Y = \Spec(R)$ is affine, we may furthermore take global sections and obtain that the map
 \[ \oh{\RGamma(\X,\sF)} \longrightarrow \RGamma(\oh{\X}, \oh{\sF}) \]
 is an equivalence of complete $R$-modules for every $\sF \in \QC(\X)$.
\end{thm}
\begin{rem} For a map of spectral stacks $\pi : \X \to \Y$ which satisfies \CD (so that $\pi_\ast$ preserves almost connective complexes), the same statement is true for any almost connective $\sF \in \QC(\X)^{acn}$ with the same proof, using \Cref{rem:qc-spf-spectral}.\end{rem}

\begin{proof}
  By general non-sense, it is enough to prove that the base-change map of \emph{left} adjoints
  \[  (i')_+ \oh{\pi}^*  \longrightarrow \pi^* i_+\]
  is an equivalence.  By \Cref{thm:qc-spf} the assertion is local on $\Y$ so that we may suppose it is affine.  Also, by \Cref{thm:qc-spf} we know that $i^*$ is essentially surjective so that it is enough to show that map
  \[ (i')_+ (i')^* \pi^* \isom (i')_+ \oh{\pi}^* i^*   \longrightarrow  \pi^* i_+ i^*\]
  is an equivalence.  By another application of \Cref{thm:qc-spf} we see that this is true (since $i_+ i^* \isom \id$ and similarly for $i'$).

The affine statement follows by applying \Cref{lem:completion-equiv}.
\end{proof}

\begin{rem} \label{rem:classical_formal_functions}
If $\Y = \Spec(R)$ with $R$ noetherian and $\pi_\ast(\sF) \in \APerf(R)$, one can recover the usual statement of the Theorem on Formal functions, that $\oh{H_i \RGamma(\sF)} \simeq H_i \RGamma(\oh{\sF})$, by combining \Cref{thm:formal_functions} with \Cref{lem:artin-rees-cat}.
\end{rem}

\subsection{Criteria for completeness}

\begin{prop}[Fully-faithfulness criterion]\label{prop:formal-functions} Let $\pi : \X \to \S$ be a morphism of noetherian algebraic spectral (or derived) stacks, and assume that $\S$ is complete along a closed subset $Z \subset |\S|$. Suppose that $\pi_\ast(\Coh(\X)) \subset \DDAPerf(\S)$ and that $G \in \APerf(\X)$. Then the natural map
  \[ \RHom_\X(F, G) \longrightarrow \RHom_{\oh{\X}}(i^* F, i^* G) \]
  is an equivalence for any $F \in \QC(\X)$, where $i : \oh{\X} \to \X$ is the formal completion of $\X$ along $Z' = \pi^{-1} (Z)$.
\end{prop}

\begin{proof}

  {\noindent}{\it Step 1: Reducing to $F, G$ bounded coherent.}\\

It suffices to prove the assertion under the additional hypothesis that $G \in \DCoh(\X)$ is bounded, because $G = \ilim \tau_{\leq n} G$, the limit being taken in $\QC(\X)^{acn}$, and \Cref{thm:qc-spf} implies that $i^\ast$ preserves limits. Suppose for concreteness, shifting as needed, that $G \in \APerf(\X)^{cn}_{\leq k}$. Let $\C \subset \QC(\X)^{acn}$ denote the full subcategory consisting of those $F$ for which the map 
  \[ \Map_\X(F,G) \longrightarrow \Map_{\oh{\X}}(i^* F, i^* G) \]
  is an equivalence for our fixed $G$.  Note that $\C$ is closed under extensions and arbitrary colimits which exist in $\QC(\X)^{acn}$.  We wish to show that  $\C = \QC(\X)^{acn}$.
  
Note that \Cref{thm:qc-spf} allows us to identify the right hand side with
  \[ \Map_{\oh{\X}}(i^* F, i^* G) \isom \Map_{\X}(\RGamma_{Z'} F, G) \isom \Map_{\X}(F \otimes_{\O_X} \RGamma_{Z'} \O_X, G) \] 
  Furthermore $\RGamma_{Z'} \O_X$ is almost connective, so that tensoring with it is right $t$-exact up to a shift.  So $\C \subset \QC(\X)$ is the subcategory for which the map to \emph{this} is an equivalence.  We see that $\C$ is closed under arbitrary colimits and contains $\QC(\X)_{\geq d}$ for some $d$ -- since $G$ is bounded and tensoring by $\RGamma_{Z'} \O_X$ is right $t$-exact up to a shift.  By \Cref{thm:coh-gen} it is thus enough to show that $\APerf(\X)_{<d} \subset \C$, i.e., that the map
  \[ \Map_\X(F, G) \longrightarrow \Map_\X(F \otimes_{\O_X} \RGamma_{Z'} \O_X, G) \]
  is an equivalence for all $F \in \APerf(\X)_{<d}$.

In particular, we are reduced to proving this in case both $F, G \in \DCoh(\X)$.  

\bigskip

{\noindent}{\it Step 2: Reducing to a fact on inner-Homs.}\\
Taking shifts, we may reduce to the case $G \in \DCoh(\X)_{\leq 0}$ and $F \in \DCoh(\X)_{\geq 0}$.  Let
\[ H = \RHom^{\otimes_{\QC(X)}}_\X(F, G)  \in \QC(\X)\]
be the inner Hom for the monoidal structure on $\QC(\X)$.  Using the definition of inner Hom it is enough to show that
\begin{equation}\label{eqn:base-change-local-coh} \RGamma(X, H) \longrightarrow \RHom_X(\RGamma_{Z'} \O_X, H) 
\end{equation}
is an equivalence.  Note that $H$ is co-connective with coherent homology sheaves by \Cref{lem:bounded_inner_hom}. Thus it is enough to prove this assertion for \emph{any} such $H$, not necessarily arising as an inner Hom.

\bigskip

{\noindent}{\it Step 3: Reducing to a fact on $\S$:}\\
Consider the fiber sequence of quasi-coherent sheaves on $\S$, $\RGamma_Z \O_\S \longrightarrow   \O_\S \longrightarrow j_* \O_\sU$, where $j$ is the open immersion $j \colon \sU = \S \setminus Z \to \S$. Applying $\pi^\ast$ to this sequence, we can use the base change formula to identify the third term with $(j')_\ast \O_{\sU'}$, where $j' \colon \sU' = \X \setminus Z' \to \X$ is the base-changed open immersion. It follows that $\RGamma_{Z'} \O_\X \isom \pi^*(\RGamma_Z \O_\S)$.

This lets us identify the morphism of \Cref{eqn:base-change-local-coh} with the morphism
\[ \RGamma(\S, \pi_* H) \longrightarrow \RHom_\S(\RGamma_Z \O_\S, \pi_* H). \]
Since $H$ was left bounded with coherent homology, the assumption on $\pi$ ensures that has $\pi_* H$ has left bounded, coherent, homology sheaves on $\S$ (see \Cref{rem:reformulate_CP} below).

It is thus enough to show that for \emph{any} $H' \in \QC(\S)$ having coherent homology sheaves the natural map $\RGamma(\S, H') \to \RHom_\S(\RGamma_Z \O_\S, H')$ is an equivalence.

\bigskip

{\noindent}{\it Step 4: Completing the proof.}\\
Note first that, since $\RGamma_Z \O_\S$ is almost connective, we can reduce as above to the case where $H'$ is right bounded and hence almost perfect.  Next, note that $i^* \RGamma_Z \O_\S = \O_{\oh{\S}}$ so that another application of \Cref{thm:qc-spf} yields
\[ \RHom_\S(\RGamma_Z \O_\S, H') \isom \RHom_{\APerf(\oh{\S})}(i^* \O_\S, i^* H') \]
Thus the desired assertion follows by our assumption that $\APerf(\S) \to \APerf(\oh{\S})$ is an equivalence.
\end{proof}

\begin{lem} \label{lem:bounded_inner_hom}
Let $\X$ be a noetherian algebraic spectral (or derived) stack, and let $F \in \APerf(\X)_{\geq 0}$ and $G \in \DDAPerf(\X)_{\leq 0}$. Let $H = \RHom^{\otimes_{\QC(\X)}}_\X(F,G)$ be the inner Hom for the monoidal structure on $\QC(\X)$. Then $H \in \DDAPerf(\X)_{\leq 0}$, and its formation is fppf local on $\X$.
\end{lem}
\begin{proof}
Note that \cite{tsd-mf}*{A.1.1} guarantees that the formation of $H$ is fppf local since $F$ is almost perfect and $G$ is bounded above.  As explained in the proof there, this does not require condition $(\ast)$ from op.cit., since the proof reduces to the affine case.  Now the claim that $H$ is co-connective with coherent homology sheaves is flat local, so in proving it we may suppose that $\X = \Spec(A)$ is affine with $A$ noetherian.  In this case, the proof in op.cit. shows that there is a third quadrant spectral sequence converging to the homology groups of $\H$ whose starting term consists of finite sums of the homology of $G$.  This shows both that the homology of $\H$ is appropriately bounded above and that each homology sheaf is coherent.
\end{proof}

Once one knows fully-faithfulness, one can establish essential surjectivity with the following:

\begin{prop}[Essential surjectivity criterion] \label{lem:characterize_complete_closed_immersions}
Let $\X$ be a noetherian algebraic derived (or spectral) stack, and let $Z \subset |\X|$ be a cocompact closed subset. Let $\oh{\X}$ be the formal completion of $\X$ along $Z$, and assume that $\APerf(\X) \to \APerf(\oh{\X})$ is fully faithful. Then the following are equivalent:
\begin{enumerate}
\item $\X$ is complete along $Z$, i.e. $\APerf(\X) \to \APerf(\oh{\X})$ is an equivalence;
\item for each $m \geq 0$, the induced functor of $(m+1,1)$-categories, $\Coh[m](\X) \to \Coh[m](\oh{\X})$, is an equivalence;
\item $\Coh(\X) \to \Coh(\oh{\X})$ is an equivalence
\item every $F \in \Coh(\oh{\X})$ admits a surjection from the restriction of an object of $\Coh(\X)$.
\end{enumerate}
\end{prop}

\begin{proof}
It is clear that $(1) \Rightarrow (2) \Rightarrow (3) \Rightarrow (4)$. Now assume $(4)$. For any $\oh{F} \in \Coh(\oh{\X})$, one can find $G_0,G_1 \in \Coh(\X)$ along with surjections $i^\ast G_0 \to \oh{F}$ and $i^\ast G_1 \to \ker(i^\ast G_0 \to \oh{F})$. It follows that $\tau_{\leq 0} \cofib(i^\ast G_1 \to i^\ast G_0) \simeq \oh{F}$. But by fully-faithfulness the morphism $i^\ast G_1 \to i^\ast G_0$ is induced by a morphism $G_1 \to G_0$ and we have $\oh{F} \simeq i^\ast \tau_{\leq 0} \cofib(G_1 \to G_0)$. Hence we have shown (3), and (2) follows from this by a simple inductive argument using the $t$-structure and \Cref{prop:formal-coh-props}. Finally for $(2) \Rightarrow (1)$, given $\oh{F} \in \APerf(\oh{\X})$, fully-faithfulness of $i^\ast$ and essential surjectivity on $\APerf(\oh{\X})_{\geq n}$ allow one to lift $\{\tau_{\leq n} \oh{F}\}$ to a tower $\cdots \to F_1 \to F_0$ in $\APerf(\X)$ with $i^\ast(F_n) \simeq \tau_{\leq n} \oh{F}$. Then part (4) of \Cref{prop:formal-coh-props} and part (2) of \Cref{lem:random-t-stuff} guarantee that
\[
i^\ast (\ilim F_n) \isom \ilim i^\ast F_n \isom \oh{F}.
\]
\end{proof}

\subsubsection{Completeness is classical}

The following is an application of \Cref{thm:descent-aperf} below, but we include it here for expository reasons:

\begin{lem} \label{lem:completeness_nil}
Let $\phi: \X \hookrightarrow \X'$ be an almost finitely presented closed immersion of algebraic derived (or spectral) stacks, and $Z \subset |\X'|$ a closed subset. If $\X'$ is complete along $Z$ then $\X$ is complete along $|\X| \cap Z \subset |\X|$, and the converse holds if $\phi$ is surjective.
\end{lem}
\begin{proof}
Note that $\phi_\ast(\cO_\X)$ is a commutative algebra object in the symmetric monoidal $\infty$-category $\APerf(\X')$, and $\APerf(\X) \simeq \phi_\ast(\cO_\X)\mod (\APerf(\X')^{\otimes})$. The same is true for $\oh{\X'}$ and $\oh{\X} = \oh{\X'} \times_{\X'} \X$. So if $\APerf(\X') \to \APerf(\oh{\X'})$ is an equivalence, then so is $\APerf(\X) \to \APerf(\oh{\X})$.

Conversely, assume that $\APerf(\X) \to \APerf(\oh{\X})$ is an equivalence. Consider the Cech nerve $\Cech(\phi) = \X^{\times_{\X'} \bullet + 1}$. Using descent for surjective closed immersions, \Cref{prop:closed-descent-QCcn}, it suffices to show that $$\APerf(\X^{\times_{\X'} n+1}) \to \APerf(\oh{\X^{\times_{\X'} n+1}})$$is an equivalence for all $n$. This follows from the previous paragraph, because $\X^{\times_{\X'} n+1} \to \X$ is an almost finitely presented closed immersion.

\end{proof}

\begin{cor}[Formal properness is classical] \label{cor:classical_coh_properness}
An almost finitely presented morphism of algebraic derived (or spectral) stacks $\pi \colon \X \to \S$ is formally proper if and only if for every \emph{classical} noetherian algebraic $\S$-stack $\S'$ which is complete along a closed subset $Z \subset |\S'|$, the \emph{classical} fiber product $(\X \times_{\S} \S')^{\rm cl}$ is complete along the preimage of $Z$.
\end{cor}
\begin{proof}
This is an immediate consequence of \Cref{lem:completeness_nil}, which implies that any noetherian algebraic derived stack $\X$ is complete if and only if $\X^{\rm cl}$ is complete.
\end{proof}

\subsection{The coherent pushforward property, \texorpdfstring{\CP}{(CP)}}

One property enjoyed by a proper morphism of schemes is that the pushforward and all higher direct images of a coherent sheaf is coherent. In this section we study this as an abstract property of a morphism of stacks.

\begin{defn} \label{def:define_CP}
Let $\X$ be a noetherian algebraic derived (or spectral) stack over a noetherian affine derived (resp. spectral) scheme $\Spec(R)$. We introduce the property
\smallskip
\begin{enumerate} [label=(CP), leftmargin=2cm]
\item[$(CP)_R$:] For any $F \in \DCoh(\X)$ and $i \in \ZZ$ the sheaf $H_i \circ \RGamma(F) \in \QC(R)^{\heart}$ is coherent. \label{property:CP}
\end{enumerate}
\smallskip
We say that a morphism $\pi : \X \to \S$ between derived prestacks satisfies the \emph{coherent pushforward} property, or \CP, if $\pi$ is relatively algebraic and almost of finite presentation, and for any morphism $\Spec(R) \to \S$ with $R$ noetherian, the base change $\X_R$ satisfies \CP[R].
\end{defn}

\begin{rem} \label{rem:reformulate_CP}
Because any $F \in \DCoh(\X)$ is a finite sequence of extensions of shifts of its homology sheaves, \CP[R] is equivalent to the condition $\pi_\ast(\Coh(\X)) \subset \DDAPerf(R)$, and using the right $t$-completeness of $\APerf(\X)$ and the left $t$-exactness of $\pi_\ast$, one sees that it is also equivalent to $\pi_\ast(\DDAPerf(\X))\subset \DDAPerf(R)$.
\end{rem}

For a morphism $\pi : \X \to \S$ of noetherian stacks, we will sometimes abuse notation and write \CP[\S] to denote the property that $\pi_\ast (\DCoh(\X)) \subset \DDAPerf(\S)$. It is an immediate consequence of fppf descent for $\DDAPerf$ that for a morphism $\pi : \X \to \S$ satisfying \CP, $\X_{\S'}$ satisfies \CP[\S'] for any noetherian algebraic derived stack with a map $\S' \to \S$.

\begin{thm}\label{T:cp_fully_faithful}
Let $\pi : \X \to \S$ be a relatively algebraic morphism of derived prestacks. Then the condition that $\pi$ satisfies \CP is local over $\S$ with respect to the fppf topology. Furthermore, if $\pi$ is almost finitely presented, then \CP is equivalent to the ``fully faithful'' part of \Cref{def:cohomologically_proper} (formal properness), namely:
\begin{itemize}[leftmargin=*]
\item[] For any noetherian algebraic derived stack $\S'$ which is complete along $Z \subset |\S'|$, and any morphism $\S' \to \S$, if $\X' := \X \times_\S \S'$ and $Z' \subset |\X'|$ is the preimage of $Z$, then the restriction map $\APerf(\X') \to \APerf(\oh{\X'})$ is fully faithful.
\end{itemize}
In particular, every formally proper morphism satisfies \CP.
\end{thm}

\begin{lem} \label{lem:coh_prop_implies_CP}
Let $\pi : \X \to \Spec(R)$ be a morphism between noetherian algebraic derived stacks such that \[\APerf(\X_{R\ps{t}}) \to \APerf(\oh{\X}_{R\ps{t}})\] is fully faithful, where $\oh{\X}_{R\ps{t}}$ denotes the formal completion of $\X_{R\ps{t}}$ along the preimage of the ideal $(t) \subset \pi_0(R\ps{t})$, then $\pi$ satisfies \CP[R].

\end{lem}

\begin{proof}
Let us denote $S = \Spec(R\ps{t})$ and $R_n = R\ps{t}/(t^n)$. Let $F \in \DCoh(\X)$, and consider the $R\ps{t}$-module $E:= (\pi_S)_\ast (F|_{\X_S}) \in \QC(S)_{\leq 0}$. Applying the base change lemma \Cref{lem:push-bdd-above} to the finite Tor amplitude inclusions $\Spec(R_n) \hookrightarrow S$ and the fully-faithfulness hypothesis of the lemma, we have
$$
E \simeq \RHom_{\X_S}(\cO_{\X_S},F|_{\X_S}) \xrightarrow{\simeq} \RHom_{\oh{\X}_S}(\cO_{\oh{\X}_S},F|_{\oh{\X}_S}) \simeq \ilim R_n \otimes_{R\ps{t}} E,
$$
so $E$ is complete as an $R\ps{t}$-module, i.e. $E \xrightarrow{\simeq} \oh{E}$ is an equivalence. Applying base change to the flat map $S \to \Spec(R)$ we have $E \simeq R\ps{t} \otimes_R \pi_\ast(F)$.

It therefore suffices to show that a complex $M \in R\mod_{\leq 0}$ such that $R\ps{t} \otimes_R M$ is complete must have coherent homology. By \cite{DAG-XII}*{Thm.~4.2.13} it suffices to prove this for $M \in \QC(R)^\heart$, for which we may assume that $R$ is classical. Express $M$ as a filtered union of coherent submodules $M = \bigcup M_\alpha$. If $M$ is not coherent, then this union does not stabilize, so we can find a sequence $m_0,m_1,\ldots \in M$ which is not contained in any $M_\alpha$. Then $R[[t]] \otimes_R M = \bigcup R[[t]] \otimes_R M_\alpha$, but the element $\sum t^i m_i \in H_0( (R[[t]] \otimes_R M)^\wedge)$ does not lie in the image of any of the $R[[t]]\otimes_R M_\alpha$, and so $R[[t]] \otimes_R M$ does not surject onto its completion. This contradicts the hypothesis that $R\ps{t} \otimes M$ is complete.
\end{proof}

\begin{proof}[Proof of \Cref{T:cp_fully_faithful}]
In order to show locality, consider an algebraic fppf morphism $\S' \to \S$, and let $\pi^\prime : \X^\prime := \X \times_\S \S^\prime \to \S^\prime$. We must show that $\pi$ satisfies \CP if $\pi^\prime$ does -- the other direction follows by definition. Note that the condition of being almost of finite presentation is local on the target for the fppf toplogy. For any map $\Spec(R) \to \S$ for a noetherian $R\in \SCR$, there exists an fppf map $\Spec(R^\prime) \to \Spec(R)$ such that $\Spec(R^\prime) \to \Spec(R) \to \S$ factors through $\S^\prime$. Because $\pi^\prime : \X^\prime \to \S^\prime$ satisfies \CP, we have that \CP[R^\prime] holds, and thus so does \CP[R] by flat base change (\Cref{lem:push-bdd-above}) and faithfully flat descent.

The fact that the condition in the statement implies \CP follows from \Cref{lem:coh_prop_implies_CP}, and the converse is \Cref{prop:formal-functions}.
\end{proof}

\begin{prop} \label{P:universally_closed}
Let $\S$ be a noetherian derived stack. Then any morphism $\pi \colon \X \to \S$ satisfying \CP, and in particular and formally proper morphism, is universally closed.
\end{prop}

\begin{lem} \label{L:noetherian_universally_closed}
Let $\pi : \X \to \Spec(R)$ be a finite type map of classical algebraic stacks, with $R$ noetherian. Then $\pi$ is universally closed if and only if $|\X_A| \to |\Spec(A)|$ is closed for any finite type $R$-algebra $A$.
\end{lem}
\begin{proof}
It suffices to show that for any $R$-algebra $A$, the base change $\X_A \to \Spec(A)$ is closed. By \Cref{thm:coh-gen} $\QC(\X)^\heart$ is a locally noetherian $R$-linear abelian category in the sense of \cite{popescu}*{Sect.~5.8}. Note also that $\QC(\X_A)^\heart$ is the category of $A$-module objects in $\QC(\X)^\heart$, and it is a locally finitely presentable (i.e. compactly generated) abelian category. In fact, it is generated by the pullbacks of coherent sheaves on $\X$, which are finitely presented objects of $\QC(\X_A)^\heart$ \cite{artinzhang}*{B3.17, B5.1i}. We write $A = \bigcup_\alpha A_\alpha$ as a filtered union of finitely generated $R$-algebras $A_\alpha$. Then pullback induces an equivalence of categories of finitely presented objects
\[
(\QC(\X_A)^\heart)^{fp} \xrightarrow{\simeq} \colim_\alpha (\QC(\X_{A_\alpha})^\heart)^{fp}.
\]

If $\Z \hookrightarrow \X_A$ is a finitely presented closed immersion, then $\cO_\Z = \op{coker}(M \to \cO_{\X_A})$ for some finitely presented quasi-coherent sheaf $M$. Then the map $M \to \cO_{\X_A}$ is the pullback of a map $M_\alpha \to \cO_{\X_{A_\alpha}}$ in $\QC(\X_{A_\alpha})^\heart$ with $M_\alpha$ finitely presented. In particular $\Z$ is the preimage of a closed substack $\Z_\alpha \hookrightarrow \X_{A_\alpha}$ 

Then we have a diagram of topological spaces, in which the square is cartesian
\[
\xymatrix{ |\Z| \ar@{->>}[r] & |\Z_\alpha| \times_{|\Spec(A_\alpha)|} |\Spec(A)| \ar[r] \ar[d] & |\Z_\alpha| \ar[d]^{\pi_\alpha} \\
& |\Spec(A)| \ar[r]^f & |\Spec(A_\alpha)| }
\]
It follows that $\pi(|\Z|) = f^{-1}(\pi_\alpha(|\Z_\alpha|))$, and hence $\pi(\Z)$ is closed, because by hypothesis $\pi_\alpha$ is a closed map.

So we have shown that the image under $\pi$ of the underlying subset of a finitely presented closed immersion is closed in $|\Spec(A)|$, and we must show that the same is true for the subset underlying \emph{any} closed immersion $\Z \hookrightarrow \X_A$. We know that $\cO_\Z = \op{coker}(M \to \cO_{\X_A})$ for some $M \in \QC(\X_A)^\heart$. Writing $M = \colim_\alpha M_\alpha$ as a colimit of finitely presented quasi-coherent sheaves, we can write $\Z$ as a cofiltered intersection of schemes $\Z = \cap_\alpha \Z_\alpha$, where $\cO_{\Z_\alpha} = \op{coker}(M_\alpha \to \cO_{\X_A})$, hence $\Z_\alpha \to \X_A$ is a finitely presented closed immersion.

We claim that $\pi(|\Z|) = \cap_\alpha \pi(|\Z_\alpha|)$. The inclusion $\pi(|\Z|) \subset \cap_\alpha \pi(|\Z_\alpha|)$ is clear. If $p \in \pi(|\Z|) \subset |\Spec(A)|$ is a point, represented by a map $\Spec(K) \to \Spec(A)$ for some field $K$, then $\Z_K \subset \X_K$ is non-empty. Because $\X_K$ is noetherian, the cofiltered system $(\Z_\alpha)_K$ is eventually constant, i.e. $\Z_K = (\Z_\alpha)_K$ for $\alpha$ sufficiently large. It follows that $p \in \cap_\alpha \pi(|\Z_\alpha|)$, hence $\pi(|\Z|) = \cap_\alpha \pi(|\Z_\alpha|)$. This shows that $\pi(|\Z|)$ is closed, because we have already shown that $\pi(|\Z_\alpha|)$ is closed for all $\alpha$.
\end{proof}

\begin{proof}[Proof of \Cref{P:universally_closed}]
Being universally closed is local for the smooth topology, so it suffices to assume $\S = \Spec(R)$ is an affine Noetherian derived scheme and show that for any $R$-algebra $A$ the map $|\X_A| \to |\Spec(A)|$ is closed. This is a statement about underlying classical stacks, so by \Cref{L:noetherian_universally_closed} it suffices to consider only almost finitely presented discrete $R$-algebras $A$. So it suffices to show that if $\pi : \X \to \Spec(A)$ is an almost finitely presented map of noetherian classical stacks which satisfies \CP, then it is closed.

By Chevalley's theorem the image of $\pi$ is constructible, so $\pi$ factors through a closed subscheme of $\Spec(A)$ in which it has a dense image. We replace $\Spec(A)$ with this closed subscheme, so that we may assume $\pi : \X \to \Spec(A)$ is dense, and we must show that it is surjective. We can choose any point $p \in |\Spec(A)|$, and consider the base change to the local ring $A_p$.

It suffices to show that $p$ lies in the image of $\pi_{A_p} \colon \X_{A_{p}} \to \Spec(A_p)$ assuming the image of this map contains a dense open subscheme $U \subset \Spec(A_p)$. We can choose a prime contained in $U$ which has dimension $1$. Its closure defines a closed subscheme of dimension $1$, $Z \subset \Spec(A_p)$, whose special point is $p$, and whose generic point lies in $U$ and hence the image of $\pi_{A_p}$. It suffices to consider the base change $\pi_Z : \X_Z \to Z$, and to show that $p \in \im(\pi_Z)$.

If the special point $p \in Z$ did not lie in the image of $\pi_Z$, then $\pi_Z$ would factor through the generic point $\X_Z \to Z \setminus \{p\} \to Z$. The latter morphism is affine, hence $t$-exact, and it is immediate that the pushforward of \emph{any} non-zero coherent sheaf is no longer coherent. This contradicts the property \CP for the map $\pi_Z$, because $(\pi_Z)_\ast(\cO_{\X_Z}) \in \DDAPerf(Z)$ is a non-zero object with coherent cohomology which we have argued is pushed forward from $Z \setminus \{p\}$. It follows that we must have $p \in \im(\pi_Z)$.

\end{proof}

\Cref{T:cp_fully_faithful} shows that \CP is ``half'' of the condition of formal properness, so it is desirable to have a more direct criterion for \CP:

\begin{prop}\label{prop:CP_quick_test} Let $\pi : \X \to \S$ be an almost finitely presented morphism of algebraic derived (or spectral) stacks satisfying \CD, and assume $\S$ is quasi-compact with affine diagonal. Then \CP is equivalent to \CP[\S] for the morphism $\pi$.
\end{prop}

\begin{proof}
Let $S \to \S$ be a morphism from a noetherian affine derived scheme, and let $\X^\prime := \X \times_\S S$. By taking shifts it suffices to show that \CP[\S] implies that $\pi^\prime_\ast(\APerf(\X^\prime)_{\geq d}) \subset \APerf(S)$ for any fixed $d$. The base-change formula, \Cref{prop:push-CD}, shows that this holds on the full subcategory of $\APerf(\X')$ spanned by derived pullbacks from $\APerf(\X)$. Since $\X' \to \X$ is affine, any $F \in \APerf(\X^\prime)_{\geq d}$ admits the cobar simplicial resolution by pullbacks of objects in $\APerf(\X)_{\geq d}$. Let $d$ be bigger than the universal cohomological dimension of $\pi$, so that each term in the pushforward simplicial resolution is connective. The result follows by noting that $\APerf(S)_{\geq 0}$ is preserved by geometric realizations, by a straightforward spectral sequence argument.
\end{proof}

%%%%%%%%%%%%%%%%%%%%%%%%%%%%%%%%%%%%%%%%%%
%%%%%% Criterion for verifying (pGE) %%%%%
%%%%%%%%%%%%%%%%%%%%%%%%%%%%%%%%%%%%%%%%%%

\section{\texorpdfstring{$h$}{h}-descent Theorems}
\label{section:h-descent}

In this section we establish a general ``$h$-descent'' theorem, \Cref{thm:descent-aperf} for $\APerf$ and $\Perf$ for later use, and we believe it is of independent interest. It is similar to -- and can be deduced, in characteristic $0$ and in the presence of a dualizing complex, from -- the $h$-descent theorem for $\Ind\DCoh$ in \cite{tsd-mf} and in \cite{GaitsgoryIndCoh}. Nevertheless, the proof is more elementary since it avoids use of the shriek pullback functors.

As an interesting consequence, we show in \Cref{prop:stack-h} that locally noetherian algebraic spectral stacks with quasi-affine diagonal are sheaves for the spectral $h$-topology. This is in contrast to the classical $h$-topology, for which the functor represented by an affine scheme is \emph{not} a sheaf.

We state most of the results in this section for derived stacks, but all of the results hold in the spectral context as well (except for \Cref{prop:stack-h}, which we only prove for spectral stacks). In the proofs, we indicate the modifications needed to prove the claim in both contexts. Throughout this section, we fix a base noetherian simplicial commutative algebra $\kk \in \SCR$, and in the spectral context we use the same symbol to denote a base $E_\infty$-algebra $\kk \in \CAlg^{cn}$.

\subsection{Descent pattern for closed immersions}
The goal for this subsection is to investigate a general descent pattern for what one might call the \emph{(derived) nil-immersion topology}. This will be the Grothendieck topology on $\SCR_\kk$ generated by surjective almost finitely presented closed immersions.

\begin{lem}\label{lem:descent-nil} Suppose that $\sF$ is a presheaf on $\SCR_\kk$ which satisfies
  \begin{enumerate}[label=(\roman*)]
    \item (``nilcomplete'') The natural map $\sF(R) \longrightarrow \ilim_n \sF(\tau_{\leq n} R)$ is an equivalence for all $R \in \SCR_\kk$;
    \item (``infinitesimally cohesive'') For every pullback diagram
      \[ \xymatrix{
     A' \ar[r] \ar[d] & A \ar[d]\\
     B' \ar[r] & B \\
     } \] in $\SCR_\kk$ for which the maps $\pi_0 A \to \pi_0 B$ and $\pi_0 B' \to \pi_0 B$ are surjections whose kernels are nilpotent ideals, the induced map 
     \[ \sF(A') \longrightarrow \sF(A) \times_{\sF(B)} \sF(B') \]
     is an equivalence.  (In fact, it is enough to require this only when the pullback diagram is a square-zero extension in the sense of \cite{HigherAlgebra}*{Section~8.4}.)
  \end{enumerate}
    Let $\pi\colon Z = \Spec(R') \to X = \Spec(R)$ be an almost finitely presented, surjective, closed immersion, and $\Cech(\pi)$ its Cech nerve. Then the natural pullback functor
  \[ \sF(R) \longrightarrow \Tot\{\sF(\Cech(\pi))\} = \Tot\left\{ \sF(R'^{\otimes_R \bullet+1}) \right\} \]
  is an equivalence. The analogous claim is true for spectral stacks as well.
\end{lem}

Before giving the proof of this lemma, let us make some remarks and state a corollary:

\begin{rem} There is a slight variant of the above where, rather than considering a single surjective closed immersion, we consider a finite family of closed immersions which are jointly surjective.  To handle this one can, for instance, replace (ii) by an arbitrary pushout diagram along closed immersions.  This is, in principle, convenient -- but will be subsumed in our application by $h$-descent.
\end{rem}

\begin{rem}
Note that in the derived context there is a difference between \emph{descent}, i.e., for Cech covers, and \emph{hyper-descent}, i.e., for hyper-covers that are not $n$-coskeletal for any $n$.  This difference is \emph{severe} for surjective closed immersions: Indeed the map from the constant simplicial diagram $\left\{\Spec(\pi_0(R)^{red}) \right \} \to \Spec(R)$ is a hyper-cover for the derived nil-immersion topology, and this shows that a presheaf $\sF$ is a hypersheaf for the derived nil-immersion topology if and only if $\sF(R) \to \sF(\pi_0(R)^{red})$ is an equivalence for all $R \in \SCR_{\kk}$.

\end{rem}

\begin{cor}\label{cor:cplt-Z} Let $\sF$ satisfy the conditions of \Cref{lem:descent-nil}, and suppose now only that $\pi \colon Z \to X$ is an almost finitely presented closed immersion, but not necessarily surjective.  Then the natural pullback functor gives an equivalence
  \[ \sF(\Spf(R)) \longrightarrow \Tot\{\sF(\Cech(\pi))\}.\]
\end{cor}
\begin{proof} Recall that $\Spf(R) \isom \dlim \Spec(R_n)$ for $\{R_n\}$ as in \Cref{prop:cplt-affine} (we do not need the finite Tor amplitude assumption here, so this works in the spectral setting as well by \Cref{rem:cplt-affine-spectral}), so that
  \[ \sF(\Spf(R)) \isom \ilim_n \sF(R_n) \] 
  and applying \Cref{lem:descent-nil} to the base change of $\pi$ to each $\Spec(R_n)$ we obtain that
  \[ \sF(R_n) \isom \Tot\{\sF(\Cech(\pi \times_{\Spec(R)} \Spec(R_n)))\}. \]  
  Taking the inverse limit of these equivalences we complete the proof, since
  \[ \ilim_n \sF(R'^{\otimes_R \bullet+1} \otimes_R R_n) \isom \sF(\Spec(R'^{\otimes_R \bullet+1}) \times_{\Spec(R)} \Spf(R)) \isom \sF(R'^{\otimes_R \bullet+1}) \]
  as $\Spec(R'^{\otimes_R \bullet+1}) \to \Spec(R)$ factors through the monomorphism $\Spf(R) \to \Spec(R)$.
\end{proof}

We now give a proof of the lemma:
\begin{proof}[Proof of \Cref{lem:descent-nil}]
  Our plan of proof is as follows: Let $\C$ denote the class of all almost finitely presented nil-thickenings $i$ of affine derived $\kk$-schemes, such that every base change of $i$ satisfies the conclusion of the lemma.  We will show that $\C$ consists of all almost finitely presented nil-thickenings by showing it contains increasingly large classes of morphisms.

  \bigskip

  {\noindent}{\bf Step 0: Case of $\pi$ with a section.} In this case, the augmented cosimplicial diagram
  \[ \sF(\oh{X}) \longrightarrow \left\{ \sF(Z^{\times_X\bullet+1}) \right\} \]
  extends to a \emph{split} augmented cosimplicial diagram.  And any split augmented cosimplicial diagram is a limit diagram.

  \bigskip

  {\noindent}{\bf Step 1: Case of square-zero extensions.} Suppose that $\pi$ fits into a pushout square
  \[ \xymatrix{
  \Spec(R) & \ar[l]_\pi \Spec(R')  \\
  \Spec(R') \ar[u]& \ar[l] \Spec(R') \oplus M[1] \ar[u]\\
  } \]
  for some connective almost perfect $R'$-module $M$ -- i.e., $\pi$ is a square-zero extension. Then by hypothesis 
  \[ \sF(R) \stackrel\sim\longrightarrow \sF(R') \times_{\sF(R' \oplus M[1])} \sF(R'). \]
  Furthermore, square-zero extensions are stable under base change, so there are also equivalences
  \[ \sF(R'^{\otimes_R \bullet}) \stackrel\sim\longrightarrow \sF(R'^{\otimes_R \bullet} \otimes_R R') \times_{\sF(R'^{\otimes_R \bullet} \otimes_R (R' \oplus M[1]))} \sF(R'^{\otimes_R \bullet} \otimes_R R') \]
  Taking totalizations, and commuting totalizations and inverse limits, we conclude that to show the conclusion for $\pi_R$ it suffices to know it for each of $\pi_{R'}$ and $\pi_{R' \oplus M[+1]}$.  But each of these belong to $\C$ by Step 0.  Thus, the conclusion holds for any square-zero extension $\pi_R$.  Thus, square-zero extensions lie in $\C$.

  \bigskip

  {\noindent}{\bf Step 2: Composites, refinements, and locality.}
  Note that $\C$ is closed under composites -- this is a formal argument with computing a totalization of a bi-cosimplicial space along the diagonal.  
  
  It is also ``local'' in the following sense: Suppose that $\pi\colon Z \to \Spec(R)$ is in $\C$, then a map $i \colon Z' \to \Spec(R)$ is in $\C$ if and only if its base-extension $i_Z \colon Z' \times_{\Spec(R)} Z \to Z$ is in $\C$.

  Finally, if a composite $h = f \circ g$ is in $\C$ -- i.e. $f$ is a ``refinement'' of $h \in \C$ -- then $f \in \C$.  A special case is the case of a morphism admitting a section (so that $h = \id$).  The general case reduces to this: to show that $f \in \C$, it is enough to show that its base change along $h$ is so.  But this base change has a section, induced from the diagonal of $f$.

  \bigskip

  {\noindent}{\bf Step 3: Case of $Z, X$ classical}
  Suppose that $\pi \colon Z \to X$ has both $Z$ and $X$ an ordinary classical scheme.  In this case, $R \to R'$ is a nilpotent surjection of ordinary rings.  Filtering by powers of the nilradical shows that it is a composite of (classical) square-zero extensions, which happen to also be derived square-zero extensions by \cite{HigherAlgebra}*{Section~8.4}.  Thus, $\pi \in \C$.  (Note that though such $\pi$ are not preserved by base change, the base change will still be a composite of square-zero extensions.)

  \bigskip

  {\noindent}{\bf Step 4: Reduction to $X$ classical.}
Suppose that $A \in \SCR_\kk$ and that we write $A = \ilim_n A_n$ as an inverse limit of almost perfect $A$-algebras $\{A_n\}$ satisfying the conditions of \Cref{lem:random-t-stuff}. Then we claim that nilcompleteness implies that
\[ \sF(A) \longrightarrow \ilim_n \sF(A_n) \]
is an equivalence.  Indeed, this follows by evaluating the inverse limit
\[ \ilim_{n,m} \sF(\tau_{\leq m} A_n) \]
in two ways -- first $n$ then $m$, and vice-versa -- and applying nilcompleteness in each case.

This hypothesis applies to the inverse systems $\{ \tau_{\leq n} R \}$ and any base change of it,  such as
  \[ R' \isom (\ilim_n \tau_{\leq n} R) \otimes_R R', \qquad R'^{\otimes_R \bullet} = (\ilim_n \tau_{\leq n} R) \otimes_R R'^{\otimes_R \bullet}  \]
  This reduces proving that $\pi \in \C$ to proving that each base change  $\pi_{\tau_{\leq n} R} \in \C$. 

  Furthermore, note that for each $n$ the map $\Spec(\pi_0(R)) \to \Spec(\tau_{\leq n} R)$ is a composite of square-zero extensions -- thus it is in $\C$.  Using locality, we reduce proving $\pi_{\tau_{\leq n} R} \in \C$ to proving $\pi_{\pi_0(R)} \in \C$.

  \bigskip
  
  {\noindent}{\bf Step 5: Completing the proof} Suppose now that $\pi \colon Z = \Spec(R') \to X = \Spec(R)$ is a nil-thickening with $X$ classical.  Let $i_0 \colon Z^{cl} \to Z$ be the inclusion of the classical part of $Z$.  By Step 3, the composite $\pi \circ i_0$ is in $\C$.  By Step 2 (``refinements'') this proves that $\pi \in \C$.
\end{proof}

\subsubsection{Closed descent for $\QC^{cn}$ and variants}
As an application of the above pattern, we have:
\begin{prop}\label{prop:closed-descent-QCcn} Suppose that $\pi \colon \Z \to \X$ is an almost finitely presented, surjective, closed immersion of derived (or spectral) $\kk$-stacks. Then the pullback functor induces an equivalence of $\infty$-categories
  \[ \QC(\X)^{cn} \longrightarrow \Tot\{\QC(\Cech(\pi))^{cn}\} = \Tot\left\{ \QC(\Z^{\times_{\X} \bullet+1})^{cn} \right\}. \]
The same holds with $\QC^{cn}(-)$ replaced by $\QC^{acn}(-)$, $\APerf(-)$ or $\Perf(-)$.
\end{prop}

\begin{proof} Note that by commuting limits we may suppose $\X = \Spec(R)$ is affine.  Now we are in a position to apply \Cref{lem:descent-nil}.  Let us check the the hypotheses hold:
 
\bigskip

  Note that (i) holds for $\QC^{cn}$ because it is left $t$-complete: Let us write $R_n$ for $\tau_{\leq n} R$.  To see that the restriction map is fully faithful we must show that the map
  \[ \Map_R(M,N) \longrightarrow \ilim_n \Map_{R_n}(M \otimes_R R_n, N \otimes_R R_n) \]
  is an equivalence.  Using left $t$-completeness on both sides we may identify this with
  \[ \ilim_k \Map_R(\tau_{\leq k} M, \tau_{\leq k} N) \longrightarrow \ilim_{n,k} \Map_{R_n}(\tau_{\leq k}(M \otimes_R R_n), \tau_{\leq k}(N \otimes_R R_n)) \]
  Next note that the maps
  \[ \tau_{\leq k}(M) \longrightarrow \tau_{\leq k}(M \otimes_R R_n) \quad \text{ for $n \geq k$} \]
  are all equivalences -- from this it follows that the previous map is an equivalence. 

  To see that it is essentially surjective, suppose that
  \[ \{M_n\} \in \ilim_n \{ (\tau_{\leq n} R)\mod^{cn} \} \]
  then note that the pushforwards $\{M_n\} \in R\mod^{cn}$ form an inverse system satisfying the hypotheses of \Cref{lem:random-t-stuff}.  In particular, letting $M = \ilim M_n$ we see that $M$ is connective.

  \bigskip

  A similar argument shows that (i) holds for $\APerf^{cn}$.  Fully-faithfulness follows from the above, and to conclude it is enough to note that if the system $\{M_n\}$ consists of modules which are perfect to order $k$, then the inverse limit $M$ is also perfect to order $k$.  This is because $\tau_{\leq k} M = \tau_{\leq k} M_n$ for $n \geq k$ compatibly with the natural equivalences
  \[ (R\mod)^{cn}_{\leq k} \isom \left(R_n\mod\right)^{cn}_{\leq k} \quad \text{for $n \geq k$}. \]
  Thus if $\tau_{\leq k} M_n$ is compact in $(R_n\mod)^{cn}_{\leq k}$, then the same is true for $\tau_{\leq k} M$.

  \bigskip

  The case of (i) for each of $\QC^{acn}$ and $\APerf$ then follows by shifting the modules as needed until they are connective --  notice that if we shift $M_0$ to be connective, the rest will be as well.   The case of $\Perf$ follows from $\APerf$ upon noting that $\Perf \subset \APerf$ may be recognized as the dualizable objects.

  \bigskip
  
  Finally, notice that (ii) for each of our four categories follows form \cite{DAG-IX}*{Theorem~7.1, Prop.~7.7} (these results apply to simplicial commutative rings as well, because the forgetful functor $\SCR \to \CAlg^{cn}$ preserves fiber products).  Thus the hypotheses of of \Cref{lem:descent-nil} hold and our result is proved.
\end{proof}

\begin{cor} Suppose that $\pi \colon \Z \to \X$ is an almost finitely presented closed immersion of algebraic derived (or spectral) $\kk$-stacks. Let $\oh{\X}$ denote the formal completion of $\X$ along the image of $\Z$.  Then the pullback functor
  \[ \QC(\oh{\X})^{cn} \longrightarrow \Tot\{\QC(\Cech(\pi))^{cn}\} = \Tot\left\{ \QC(\Z^{\times_{\X} \bullet+1})^{cn} \right\} \]
  is an equivalence of $\infty$-categories.  The same holds with $\QC^{cn}(-)$ replaced by $\QC^{acn}(-)$, $\APerf(-)$, or $\Perf(-)$.
\end{cor}
\begin{proof} Combine the argument of the previous Proposition with \Cref{cor:cplt-Z}.
\end{proof}

\subsection{Descent pattern for the \texorpdfstring{$h$}{h}-topology}

By a derived $h$-cover we will mean a morphism $\pi \colon \X' \to \X$ which is representable by algebraic derived stacks almost of finite presentation, and which is a universal topological submersion, and we use the same definition for spectral stacks. Examples include fppf surjections (since they are universally open), proper surjections (since they are universally closed), and more generally formally proper surjections (which are also universally closed by \Cref{P:universally_closed}).

Our descent pattern is based on the idea that the $h$-topology is generated by surjective closed immersions, fppf covers, and ``abstract blowup squares.''  In the derived context, closed immersions are no longer monomorphisms -- this necessitates modifying the abstract blowup square condition by taking suitable formal completions. We say that a cartesian diagram of prestacks
  \begin{equation}\label{eqn:abstract_blowup_square}
  \xymatrix{
  \oh{\Y} \ar[d]_{\oh{\pi}} \ar[r]^{i'} & \Y \ar[d]^{\pi} \\
  \oh{\X} \ar[r]_-i & \X = \Spec(R)
  }
  \end{equation}
is an \emph{abstract blowup square} (with affine base) if $\X = \Spec(R)$ for a noetherian $R \in \SCR_\kk$, $\pi$ is a proper derived algebraic space, $\oh{\X}$ is the completion of $\X$ along a closed subset $|\Z| \subset |\X|$, and  $\pi^{-1}(\X \setminus \Z) \to \X \setminus \Z$ is an isomorphism.

\begin{prop}\label{prop:h-descent-blowup} Suppose that $\sF$ is a presheaf on locally noetherian algebraic derived (or spectral) $\kk$-stacks satisfying:
  \begin{enumerate}[label=(\roman*)]
      \item $\sF$ has descent for surjective almost finitely presented closed immersions;
      \item $\sF$ has fppf descent;
      \item For any abstract blowup square \eqref{eqn:abstract_blowup_square} with affine base, the natural map
\[ \sF(\X) \longrightarrow \sF(\oh{\X}) \times_{\sF(\oh{\Y})} \sF(\Y) \]
is an equivalence.
  \end{enumerate}

  Then $\sF$ has descent for $h$-covers.
\end{prop}
\begin{proof} 

Let $\D$ denote the class of all almost finitely presented surjections $\pi \colon \X' \to \X$ with locally noetherian target.  Let $\C\subset \D$ consist of those $\pi$ satisfying, additionally, the condition that for every almost finitely presented affine map $\Y \to \X$, $\sF$ has descent for the base change $\pi_\Y$, i.e. restriction gives an equivalence $\sF(\Y) \to \Tot\{\sF(\Cech(\pi_\Y))\}$. We wish to show that $\C$ contains all $h$-covers.

Note that if $\pi \in \C$, then $\sF$ has descent for $\pi_\Y$ for \emph{any} almost finitely presented morphism $\Y \to \X$: It suffices to consider $\X$ quasi-compact, i.e. noetherian. We can choose an \'etale surjection $p : \sV \to \Y$ with $\sV$ affine, and by \'etale descent and commuting limits it suffices to show that $\pi_{\sV_n} \in \C$ for each level $\sV_n = \sV^{\times_\Y n}$ in the Cech nerve of $p$. Thus we can reduce to the case where $\Y \to \X$ is representable by derived algebraic spaces. Repeating this trick three more times, we reduce to the case where $\Y \to \X$ is representable by derived schemes, then by separated derived schemes, and finally to where $\Y \to \X$ is affine.

The proof of \Cref{lem:descent-nil}, Step 2, applies verbatim to show that $\C$ is closed under composition and refinements, and to show locality along maps in $\C$.

\bigskip

{\noindent}{\bf Step 0: Reduction to $\X$ an affine scheme.}\\
Since $\sF$ is an fppf sheaf, the assertion is fppf local on $\X$.  Thus we may suppose $\X$ is affine.

{\noindent}{\bf Step 1: Reduction to $\X$ classical and reduced, $\X'$ classical.}\\
Since $\Spec(\pi_0(R)_{red}) \to \Spec(R)$ is a surjective closed immersion, it is in $\C$.  So by ``locality'' to show that $\pi \in \C$ it is enough to show that the base change $\pi_{\pi_0(R)_{red}} \in \C$.  That is, we may suppose that $\X$ is classical and reduced.  Similarly, by ``refinements'' we may suppose that $\X'$ is classical.

\bigskip

{\noindent}{\bf Step 2: A general descent argument for (classical) $h$-covers}
We will prove, by noetherian induction, that for every closed subscheme $Z \subset \X$ it is the case that all $h$-coverings having target $Z$ belong to $\C$.  Thus we may suppose that this holds for all (classical) proper closed subschemes of $\X$, and we must show that every $h$-cover $\pi \colon \X' \to \X$ belongs to $\C$. We can choose a smooth cover $\Spec(A) \to \X'$, and the the composition $\Spec(A) \to \X$ is again an $h$-cover, so it suffices by refinement to consider the case where $\pi$ is representable by derived algebraic spaces.

%We now make the mentioned convenient (but not necessary) reduction: 

Since $\X$ is reduced, we may apply ``generic flatness'' to $\pi$ to deduce the existence of a dense open $U \subset \X$ such that $\pi_U$ is flat. By ``platification par eclatement,'' there exists a closed subscheme $Z \subset \X$ disjoint from $U$ such that the strict transform $\sq{p_Z}(\X')$ is flat over $\Bl_Z(\Y')$.  More precisely, there is a commutative diagram
\[ \xymatrix{
\ar[dr]_{\pi'} \Y' = \sq{p_Z}(\X') \ar@{^{(}->}[r] & \Bl_Z(\X) \times_{\X} \X^\prime \ar[d] \ar[r] & \X' \ar[d]^{\pi} \\ 
& \Y = \Bl_Z(\X) \ar[r]_-{p_Z} & \X 
} \]
such that $\pi' \colon \Y' \to \Y$ is flat.  In order to show that $\pi \in \C$ it is enough, by refinement, to show that $p_Z \circ \pi' \in \C$.  Note that $\pi'$ is surjective by the argument of \cite{Voevodsky-Homology}*{Prop.~3.1.3}, so that it is a flat cover and thus in $\C$.  It thus suffices to show that the blowup morphism $p_Z \in \C$.

We now simplify our notation by renaming $p_Z$ to just $\pi \colon \X' = \Bl_Z(\X) \longrightarrow \X$. We consider an almost finitely presented map $\Y = \Spec(R') \to \X$, and we wish to show that $\sF$ has descent for the base change $\pi_{\Y} : \Y' \to \Y$ of $\pi$.

$\pi_{Z} \in \C$ by our inductive hypothesis, so by Step 1, $\pi_{R/I^n} \in \C$ for all $n \geq 1$, where $I \subset R$ is the ideal defining $Z$. We may summarize this with the slight abuse of notation $\pi_{\Spf(R)} \in \C$. It follows that if $\oh{\Y}$ is the completion of $\Y$ along the preimage of $Z$, then $\sF(\oh{\Y}) \to \Tot\{\sF(\Cech(\pi_{\oh{\Y}}))\}$ is an equivalence. Now consider the cartesian diagram of simplicial derived schemes
\[
\xymatrix{ \Cech(\pi_{\oh{\Y}}) \ar[r] \ar[d] & \Cech(\pi_\Y) \ar[d] \\ \{\oh{\Y}\} \ar[r] & \{\Y\} },
\]
where the bottom row consists of constant simplicial schemes. This diagram is level-wise an abstract blow up square \eqref{eqn:abstract_blowup_square} with affine base, so applying our hypothesis (iii) and commuting limits, we see that
\[
\sF(\Y) \to \Tot\{\sF(\Cech(\pi_\Y))\} \times_{\Tot\{\sF(\Cech(\pi_{\oh{\Y}}))\}} \sF(\oh{\Y})
\]
is an equivalence, and hence $\sF(\Y) \to \Tot\{\sF(\Cech(\pi_\Y))\}$ is an equivalence.

\end{proof}

\subsection{\texorpdfstring{$h$}{h}-descent theorems for \texorpdfstring{$\APerf$}{APerf}, \texorpdfstring{$\Perf$}{Perf}, and quasi-geometric stacks}

By a locally noetherian stack in the following theorem, we mean a stack which is a left Kan extension along the morphism $\SCR^{noeth}_{\kk} \hookrightarrow \SCR_{\kk}$. Examples include locally noetherian algebraic derived $\kk$-stacks, as well as formal completions of such stacks along closed substacks.

\begin{thm}\label{thm:descent-aperf} 
Suppose that $\X$ is a locally noetherian derived (or spectral) $\kk$-stack, and that $\pi \colon \X' \to \X$ is an $h$-covering. Then if $\{\X'_\bullet\} := \Cech(\pi)$, the pullback functors determine equivalences of $\infty$-categories
  \[ (\pi_\bullet)^* \colon \APerf(\X) \to \Tot\{ \APerf(\X'_\bullet) \}.\]
The same is true with $\APerf(-)$ replaced by $\Perf(-)$ or $\APerf(-)^{cn}$.
\end{thm}

We will establish this after some preliminary lemmas.

\begin{lem} \label{lem:formal_equivalence}
Let $p : \X' \to \X$ be a relatively algebraic qc.qs. map satisfying \CD between derived (or spectral) stacks, and let $Z \subset |\X|$ be a cocompact closed subset such that $\oh{\X'}_{p^{-1}(Z)} \to \oh{\X}_Z$ is an isomorphism. Then $p^\ast : \QC_Z(\X) \to \QC_{p^{-1}(Z)}(\X')$ is a equivalence whose inverse is given by $p_\ast$.
\end{lem}
\begin{proof}
This follows from \Cref{thm:qc-spf} in the derived setting, and from \cite{DAG-XII}*{Theorem~5.1.9} in the spectral setting for categories of almost connective modules. For the stronger statement in the spectral setting, note that the claim is local on $\X$, so we may assume $\X = \Spec(R)$. The claim is equivalent to the canonical maps $M \to p_\ast(p^\ast(M))$ and $p^\ast (p_\ast(F)) \to F$ being equivalences for all $M \in \QC_Z(\X)$ and $F \in \QC_{p^{-1}(Z)}(\X')$. This follows from the case of almost connective objects, because the $t$ structures on these categories are right $t$-complete and $p_\ast$ and $p^\ast$ commute with filtered colimits by \Cref{prop:push-CD}.
\end{proof}

\begin{lem}[Excision for $\QC$] \label{lem:qc-pushout}
Consider a cartesian diagram of derived (or spectral) stacks
\[
\xymatrix{\Y' \ar[r]^{p'} \ar[d]^{q'} & \Y \ar[d]^q \\ 
\X' \ar[r]^p & \X}
\]
such that both $p$ and $q$ are relatively algebraic, qc.qs. and satisfy \CD. Assume that there is a cocompact closed subset $Z \subset |\X|$ with complement $\sU = \X \setminus Z$ such that $\oh{\X'}_{p^{-1}(Z)} \to \oh{\X}_Z$ and $\Y \times_\X \sU \to \sU$ are isomorphisms. Then restriction induces an equivalence of categories
\[ \QC(\X) \to \QC(\X') \times_{\QC(\Y')} \QC(\Y). \]
If $p$ is flat at every point in $p^{-1}(Z) \subset |\X'|$, then the same holds with $\QC(-)$ replaced by $\QC(-)^{cn}$, $\QC(-)^{acn}$, $\APerf(-)$, or $\Perf(-)$.
\end{lem}

\begin{proof}
This assertion is local on $\X$, so we are free to assume that $\X = \Spec(R)$ is affine, and hence the other stacks are qc.qs. algebraic.

Let $g = p \circ q' \simeq q \circ p'$. Note that the pullback $\QC(R) \to \QC(\X') \times_{\QC(\Y')} \QC(\Y)$ has a right adjoint
\[ (F_{\X'}, F_{\Y'}, F_{\Y}) \mapsto p_* F_{\X'} \times_{g_* F_{\Y'}} q_* F_{\Y} .\]
For fully faithfulness we must show that the unit of this adjunction is an isomorphism. By the base change formula (\Cref{prop:push-CD}),
\[ g_\ast(F_{\Y'}) \simeq p_\ast(q'_\ast((p')^\ast(F_\Y))) \simeq p_\ast(p^\ast(q_\ast(F_\Y))) \]
So for fully faithfulness it suffices to show that $\forall M \in \QC(R)$,
\[ \xymatrix{
M \ar[d] \ar[r] & p_\ast(p^\ast(M)) \ar[d] \\
q_* q^* M \ar[r] &  p_\ast(p^\ast(q_\ast( q^\ast(M)))) }\]
is a pullback diagram. This follows from the fact that $\cofib(M \to q_\ast q^\ast(M)) \in \QC_Z(\X)$, and hence the unit of adjunction $(-) \to p_\ast(p^\ast(-))$ is an isomorphism when applied to this object, by \Cref{lem:formal_equivalence}.

For essential surjectivity, suppose given $F_\Y$, $F_{\Y'}$, and $F_{\X'}$, with an equivalence $(q')^* F_{\X'} \simeq F_{\Y'} \simeq (p')^*(F_{\Y})$. We can tensor this triple with the fiber sequence $\RGamma_Z(R) \to R \to j_\ast(\cO_\sU)$, where $j : \sU \hookrightarrow \Spec(R)$ is the open immersion. Making use of the projection formula, this shows that $(F_\Y,F_{\Y'},F_{R'})$ is an extension of the two triples:
\[ \left( F_{\X'} \otimes p^\ast(j_\ast(\cO_\sU)), F_{\Y'} \otimes g^\ast(j_\ast(\cO_\sU)), F_{\Y} \otimes q^\ast(j_\ast(\cO_\sU)) \right), \text{ and} \]
\[ \left( F_{\X'} \otimes p^\ast(\RGamma_Z(R)), F_{\Y'} \otimes g^\ast(\RGamma_Z(R)), F_{\Y} \otimes q^\ast(\RGamma_Z(R)) \right).\]
We have already shown fully faithfulness, so it suffices to show that these two triples lie in the essential image of the restriction functor.

For the first triple, using the base change and projection formulas, it suffices to show that restriction induces an equivalence
\[
\QC(\sU) \to \QC(p^{-1}(\sU)) \times_{\QC(g^{-1}(\sU))} \QC(q^{-1}(\sU)).
\]
This follows from the hypothesis that $g^{-1}(\sU) \to p^{-1}(\sU)$ and $q^{-1}(\sU) \to \sU$ are isomorphisms. For the second triple, it suffices to show that restriction induces an equivalence
\[
\QC_Z(R) \to \QC_{p^{-1}(Z)}(\X') \times_{\QC_{g^{-1}(Z)}(\Y')} \QC_{q^{-1}(Z)}(\Y).
\]
This follows from \Cref{lem:formal_equivalence}, which implies that $\QC_{q^{-1}(Z)}(\Y) \to \QC_{g^{-1}(Z)}(\Y')$ and $\QC_Z(R) \to \QC_{p^{-1}(Z)}(\X')$ are equivalences.

For the variants of the lemma, choose a smooth surjection $\Spec(R') \to \X'$, and let $S \subset \pi_0(R')$ be the multiplicative system consisting of all elements which do not vanish along the preimage of $Z$. Then the map $\sU \sqcup \Spec(S^{-1} R') \to \Spec(R)$ is an fpqc cover, so an $R$-module is connective / almost perfect / perfect if and only if its restrictions $p^\ast(M)$ and $q^\ast(M)$ are. This allows us to deduce the claim for $\QC(-)^{cn}$, $\APerf(-)$, and $\Perf(-)$ from the claim for $\QC(-)$.
\end{proof}

\begin{lem} \label{lem:aperf-blowup} Given an abstract blowup square \eqref{eqn:abstract_blowup_square}, the natural restriction map
\[ \APerf(\X) \longrightarrow \APerf(\oh{\X}) \times_{\APerf(\oh{\Y})} \APerf(\Y) \]
is an equivalence. This is also true with $\APerf(-)$ replaced by $\Perf(-)$ or $\APerf(-)^{cn}$.
\end{lem}
\begin{proof}
Note that this assertion is flat local on $\X$, so we are free to assume that $\X = \Spec(R)$ is affine. Let $\oh{R}$ be the completion of $R$ along the ideal defining $Z \hookrightarrow \X$. Restriction induces equivalences $\APerf(\Y') \simeq \APerf(\oh{\Y})$, by the Grothendieck existence theorem, and $\APerf(\oh{R}) \to \APerf(\Spf(R))$. It therefore suffices to show that restriction induces an equivalence $\APerf(R) \to \APerf(\oh{R}) \times_{\APerf(\Y')} \APerf(\Y)$, which follows from \Cref{lem:qc-pushout}.
\end{proof}

\begin{proof}[Proof of \Cref{thm:descent-aperf}]
The claim follows from applying \Cref{prop:h-descent-blowup} to the functors $\APerf(-)$, $\Perf(-)$, and $\APerf(-)^{cn}$. Condition (i) of \Cref{prop:h-descent-blowup} follows from \Cref{prop:closed-descent-QCcn}, condition (ii) is fppf descent, which these functors have, and condition (iii) follows from \Cref{lem:aperf-blowup}.

\end{proof}

\begin{cor}\label{cor:completion} Suppose that $\X$ is a locally noetherian algebraic derived (or spectral) stack, and that $\pi \colon \X' \to \X$ is an almost finitely presented, universally closed morphism. Let $\oh{\X}$ denote the completion of $\X$ along the (closed) image of $\pi$.  Then the pullback functors determine equivalences
  \[ \APerf(\oh{\X}) \stackrel\sim\longrightarrow \Tot\{\APerf(\Cech(\pi))\} = \Tot\left\{ \APerf(\X'^{\times_\X \bullet+1}) \right\} \]
  \[ \Perf(\oh{\X}) \stackrel\sim\longrightarrow \Tot\{\Perf(\Cech(\pi))\} = \Tot\left\{ \Perf(\X'^{\times_\X \bullet+1}) \right\} \]
\end{cor}
\begin{proof} Combine \Cref{thm:descent-aperf} with the argument given for \Cref{cor:cplt-Z}.
\end{proof}

\subsubsection{Consequences for quasi-geometric stacks}

\begin{prop}\label{prop:stack-h}  If $\sF$ is a locally noetherian algebraic spectral stack with quasi-affine diagonal, then $\sF$ has descent for spectral $h$-covers.
\end{prop}
\begin{proof}
It suffices, by Zariski descent, to prove this when $\sF$ is quasi-compact, i.e. noetherian. Given an $h$-cover $\pi \colon \X' \to \X$, let $\X'_\bullet = \Cech(\pi)$. By \Cref{thm:descent-aperf} restriction gives an equivalence $\APerf(\X)^{cn} \simeq \Tot \{ \APerf(\X'_\bullet)^{cn} \}$. Note that $\Tot$ is a limit, and colimits in the totalization are formed level-wise, so we have, for any symmetric monoidal $\infty$-category $\cC$ which admits finite colimits
\[
\Fun_\otimes^c(\cC,\APerf(\X)^{cn}) \simeq \Tot \{ \Fun_\otimes^c(\cC,\APerf(\X'_\bullet)^{cn})\}.
\]
The result follows from the case where $\cC = \APerf(\sF)^{cn}$, by the Tannaka duality theorem \Cref{thm:Tannaka}.

\end{proof}

\begin{cor} \label{cor:stack-pushout}
In the $\infty$-category of noetherian algebraic spectral stacks with quasi-affine diagonal, any cartesian square which satisfies the conditions of \Cref{lem:qc-pushout}, and in which $p$ is flat along $p^{-1}(Z)$, is a also a pushout square.
\end{cor}
\begin{proof}
Combine \Cref{lem:qc-pushout} with Tannaka duality (\Cref{thm:Tannaka}), as in the proof of \Cref{prop:stack-h}.
\end{proof}

\begin{rem}
We expect \Cref{cor:stack-pushout} and \Cref{prop:stack-h} also hold for derived stacks, but one would need to use Tannaka duality in a less direct way (see, for instance, the proof of \Cref{prop:integrability}).
\end{rem}

%%%%%%%%%%%%%%%%%%%%%%%%%%%%%%%%%%%%%%%%%%
%%%%%% Criterion for verifying (cGE) %%%%%
%%%%%%%%%%%%%%%%%%%%%%%%%%%%%%%%%%%%%%%%%%

\section{Formally proper morphisms}
\label{section:CA-CP}

In this section we establish many examples of formally proper morphisms. \Cref{thm:coh_projective_is_proper} states that ``cohomologically projective'' morphisms are formally proper, and we give two examples of large classes of cohomologically projective morphisms in \Cref{prop:projective_over_affine_quotient} and \Cref{prop:good-moduli-projective}. We use also develop a method of establishing formal properness using proper coverings in \Cref{thm:descent-apGE-new}.

\subsection{Cohomologically ample, (CA), systems}

We introduce a structure for a morphism of stacks which generalizes the notion of a relatively ample bundle for a proper morphism of schemes (\Cref{ex:ample_bundle}).

We will say that a set $I$ is \emph{preordered by the nonnegative integers} if we have an assignment of an integer $\val{\alpha} \geq 0$ to each $\alpha \in I$. In this case we say $\alpha \geq \beta$ (respectively $\alpha > \beta$) if $\val{\alpha} \geq \val{\beta}$ (respectively $\val{\alpha} > \val{\beta}$). For $E,F \in \QC(\X)_{\geq 0}$ we say that a map $E \to F$ is surjective if it is $0$-connective, or equivalently $H_0(E) \to H_0(F)$ is surjective in $\QC(\X)^\heart$.

%to be surjective if $\cofib (E \to F) \in \QC(\X)_{\geq 1}$. This is equivalent to the condition that $E \to F$ is 

\begin{defn} \label{defn:CA}
Let $\X$ be a noetherian algebraic derived (or spectral) stack, and let $\{V_\alpha\}_{\alpha \in I}$ be a collection of locally free sheaves on $\X$ indexed by a set preordered by the nonnegative integers. We say that $\{V_\alpha\}_{\alpha \in I}$ is a \emph{cohomologically ample (CA) system} if for all $F \in \Coh(\X)$,
\begin{enumerate} [label=(CA\arabic{enumi})] \label{property:CA}
\item $\forall N \geq 0$, $\Hom(V_\alpha,F) \neq 0$ from some $V_\alpha$ with $\val{\alpha} \geq N$, and
\item $\forall i < 0$, $\exists N \geq 0$ such that $H_i \RHom(V_\alpha, F) = 0$ whenever $\val{\alpha} \geq N$.
\end{enumerate}
If $\pi : \X \to \S$ is a qc.qs. morphism of locally noetherian algebraic derived stacks, we say that $\{V_\alpha\}_{\alpha \in I}$ is a \emph{cohomologically ample system relative to $\pi$} if for any affine derived scheme $T$ which is flat and locally almost finitely presented over $\S$, the system $\{V_\alpha |_{\X \times_\S T} \}_{\alpha \in I}$ satisfies (CA1) and (CA2).
\end{defn}

We will often abbreviate by saying that $\{V_\alpha\}$ is a preordered system.

\begin{ex} \label{ex:ample_bundle}
If $X$ is a projective classical scheme over an affine noetherian base $S$, then $X$ admits a cohomologically ample system. We take $I = \{n \geq 0\}$ with $\val{n}=n$, and we let $V_n := \L^{-n}$, where $\L$ is an ample invertible sheaf on $X$.
\end{ex}

\begin{defn} \label{defn:coh_projective}
A morphism between locally noetherian algebraic derived stacks, $\pi : \X \to \S$, is \emph{cohomologically projective} if it satisfies \CD and \CP, and possesses a relatively \CA system of locally free sheaves $\{V_\alpha\}_{\alpha \in I}$.
\end{defn}

We will collect some useful (and familiar) properties of cohomologically ample systems over several lemmas.

\begin{lem}\label{lem:ca-classical} Let $\X$ be an algebraic derived (or spectral) stack and let $i\colon \X_{cl} \to \X$ be the inclusion of the underlying classical stack. Then a preordered system $\{V_\alpha\}$ is \CA on $\X$ if and only if $\{i^* V_\alpha\}$ is \CA on $\X_{cl}$.
\end{lem}
\begin{proof} This follows immediately from the $(i^*, i_*)$ adjunction and the fact that $i_* \colon \Coh(\X_{cl}) \to \Coh(\X)$ is an equivalence of categories.
\end{proof}

\begin{lem} \label{lem:descent_for_CA}
Consider a cartesian square of locally noetherian algebraic derived (or spectral) stacks
$$\xymatrix{\X^\prime \ar[r]^{f'} \ar[d]^{\pi^\prime} & \X \ar[d]^\pi \\ \S^\prime \ar[r]^f & \S },$$
such that $f$ is flat and locally almost finitely presented. Let $\{V_\alpha\}_{\alpha \in I}$ be a preordered system of locally free sheaves on $\X$, and let $V'_\alpha := (f')^\ast(V_\alpha)$. If $\{V_\alpha\}$ is \CA relative to $\pi$, then $\{V'_\alpha\}$ is \CA relative to $\pi'$. The converse holds if $f$ is surjective.
\end{lem}
\begin{proof}
The first claim follows immediately from the definition. For the converse, observe that for any flat and locally almost finitely presented map $\Spec(R) \to \S$, we can find a flat and locally almost finitely presented map $\Spec(R') \to \S'$ such that $\Spec(R^\prime) \to \S^\prime \to \S$ factors through $\Spec(R) \to \S$ via a faithfully flat map $\Spec(R^\prime) \to \Spec(R)$. For $F \in \Coh(\X_R)$, condition (CA1) and (CA2) can be checked after pullback along the faithfully flat map $\X_{R'} \to \X_{R}$, and thus they hold by the hypothesis that $\{V'_\alpha\}$ is \CA relative to $\pi'$.
\end{proof}

\begin{lem}\label{lem:relative_CA_surjectivity}
Let $\pi : \X \to \S$ be a morphism of noetherian algebraic derived (or spectral) stacks, and let $\{V_\alpha\}_{\alpha \in I}$ be a \CA system relative to $\pi$. Then $\forall F \in \Coh(\X)$ and $\forall N\geq 0$ there is a locally free sheaf $W$ which is a finite direct sum of sheaves in the collection $\{V_\alpha | \val{\alpha} \geq N \}$ such that the canonical map obtained by composing
\begin{equation} \label{eqn:global_section_map}
W \otimes \pi^\ast H_0 \pi_\ast(W^\ast \otimes F) \to W \otimes \pi^\ast \pi_\ast(W^\ast \otimes F) \to W \otimes W^\ast \otimes F \to F
\end{equation}
is surjective. If $\pi$ satisfies \CD, then the map \eqref{eqn:global_section_map} is surjective for all $F \in \APerf(\X)_{\geq 0}$
\end{lem}
Note that the canonical map $H_0 \pi_\ast (W^\ast \otimes F) \to \pi_\ast(W^\ast \otimes F)$ used in this lemma exists because $W^\ast \otimes F \in \QC(\X)^\heart$.
\begin{proof}
The formation of \eqref{eqn:global_section_map} commutes with flat base change by \Cref{lem:push-bdd-above}, so because $\S$ is quasi-compact we may reduce to the case where $\S$ is affine. Note that in this case if $F$ admits a surjection $W \to F$, then the canonical map \eqref{eqn:global_section_map} is surjective, so it suffices to find a $W$ as in the statement of the lemma which admits a surjection $W \to F$.

Let $W \to F$ be a non-zero homomorphism, with $W$ locally free. Let $Q = H_0(\cofib(W \to F))$ be the cokernel of the map $H_0(W) \to F$, and let $K = \fib(F \to Q)$ be the kernel of the resulting quotient map. (CA1) implies that we can find a non-zero homomorphism $V_\alpha \to Q$ with $\val{\alpha} \geq N$, and choosing $\alpha$ sufficiently large we may suppose by (CA2) that $\Hom(V_\alpha,K[1])=0$, so the map $V_\alpha \to Q$ lifts to a map $V_\alpha \to F$. Note that the image of the resulting map $W \oplus V_\alpha \to F$ is strictly larger than $\im(W \to F)$. We can thus replace $W$ with $W \oplus V_\alpha$ and repeat this process as long as $W \to F$ is not surjective, and this process must terminate after finitely many steps because $\X$ is noetherian.

Now assume $\pi$ satisfies \CD, and $F \in \APerf(\X)_{\geq 0}$. If $\pi_\ast$ has cohomological dimension $d$, then $H_0 \pi_\ast(\tau_{> d}(F)) = 0$. (CA2) implies that for $\alpha$ large enough, we can assume that $H_0(\pi_\ast(H_i(F)[-i]))$ for $i=1,\ldots,d$. Because $F$ can be constructed as a finite sequence of extensions of the objects $\tau_{>d}(F)$ and $H_i(F)[-i]$ for $i=0,\ldots,d$, this reduces the claim to showing that $W \otimes \pi^\ast(H_0 \pi_\ast(H_0(F))) \to H_0(F)$ is surjective, which we have already done.
\end{proof}

Under an additional cohomological dimension hypothesis, we can weaken the definition of \CA:
\begin{lem} \label{lem:better_hypotheses_CA}
Let $\X$ be a noetherian stack of finite cohomological dimension, and let $\{V_\alpha\}_{\alpha \in I}$ be a preordered system of locally free sheaves on $\X$ which satisfies (CA1). Then (CA2) is equivalent to either of the following:
\begin{enumerate}%[label=(CA\arabic{enumi}')]
\item[(CA2')] $\forall \beta \in I$, $\exists N \geq 0$ such that $\RGamma(V_\beta \otimes V_\alpha^\ast)$ is connective whenever $\val{\alpha} \geq N$.
\item[(CA2'')] $\forall F \in \APerf(\X)$, $\exists N \geq 0$ such that $\RHom(V_\alpha,F)$ is connective whenever $\val{\alpha} \geq N$.
\end{enumerate}
\end{lem}

\begin{proof}
It is clear that if (CA2'') implies both (CA2) and (CA2'), and an argument similar to the last paragraph in the proof of \Cref{lem:relative_CA_surjectivity} shows that (CA2) implies (CA2'') when $\X$ has finite cohomological dimension. We will thus show that (CA2') implies (CA2'').

We introduce the following condition for $p\geq 0$:
\begin{enumerate}[label=$(CA2)_p$]
\item For all $F \in \APerf(\X)_{\geq 0}$, there exists an integer $N$ such that $\RGamma (F \otimes V_\alpha^*) \in (\bZ -Mod)_{> -p}$ for all $\val{\alpha} \geq N$.
\end{enumerate}
Our goal is to prove that $(CA_2)_1$ holds, and we will show this by descending induction on $p > 0$. Note that \CD implies that this holds for some sufficiently large $p$, providing the base case.
  
Assume $(CA2)_p$, and note that by considering the long exact homology sequence associated to the exact triangle $(\tau_{>0} F) \otimes V_\alpha^\ast \to F \otimes V_\alpha^\ast \to (H_0 F)\otimes V_\alpha^\ast \to$, it suffices to consider $F \in \Coh(\X)$ in order to show $(CA2)_{p-1}$. Let $F \in \Coh(\X)$. By condition (CA1') we have a non-zero homomorphism $\phi : V_\beta \to F$. This leads to two exact triangles
\begin{gather*}
V_\beta \to F \to \cofib(\phi) \to \\
\tau_{>0} \cofib(\phi) \to \cofib(\phi) \to F_1 := \tau_{\leq 0} \cofib(\phi) \to
\end{gather*}  
We want to show that $H_{-p+1} \RGamma (F \otimes V_\alpha^\ast) = 0$ for $\val{\alpha} \gg 0$. For $\val{\alpha} \gg 0$ we have $\RGamma(V_\beta \otimes V_\alpha^\ast) \in (\bZ-Mod)_{\geq 0}$ by (CA2') and $H_{-p+1} \RGamma(V_\alpha^\ast \otimes \tau_{>0} \cofib(\phi)) = 0$ by hypothesis $(CA2)_p$. Thus from the long exact sequence in homology applied to the exact triangles above, we have
\[ H_{-p+1} \RGamma (F \otimes V_\alpha^\ast) \simeq H_{-p+1} \RGamma (\cofib(\phi) \otimes V_\alpha^\ast) \simeq H_{-p+1} \RGamma (F_1 \otimes V_\alpha^\ast) \]
whenever $\val{\alpha} \gg 0$.

Iterating this argument we get a strictly descending sequence $F \twoheadrightarrow F_1 \twoheadrightarrow F_2 \twoheadrightarrow \cdots$ in $\Coh(\X)$ such that $H_{-p+1} \RGamma(F \otimes V_\alpha^\ast) \simeq H_{-p+1} \RGamma(F_i \otimes V_\alpha^\ast)$ for $\val{\alpha} \gg 0$. Because $\X$ is noetherian, we must have $F_n = 0$ for some $n$, hence $H^p \RGamma(F \otimes V_\alpha^\ast) = 0$.

\end{proof}

Using this alternate characterization of \CA systems, we can establish that relatively \CA systems are stable under base change.

\begin{lem} \label{lem:base_change_CA}
Consider a cartesian square of noetherian algebraic derived (or spectral) stacks
$$\xymatrix{\X^\prime \ar[r]^f \ar[d]^{\pi^\prime} & \X \ar[d]^\pi \\ \Spec(R^\prime) \ar[r] & \Spec(R) }$$
such that $\pi$ satisfies \CD. If a preordered system $\{V_\alpha\}$ on $\X$ is \CA, then its restriction $\{V_\alpha|_{\X^\prime}\}$ is \CA.
\end{lem}

\begin{proof}
By \Cref{lem:better_hypotheses_CA} it suffices to check that $\{V_\alpha|_{\X^\prime}\}$ satisfies (CA1) and (CA2'). (CA1) is immediate: for any non-zero $F \in \Coh(\X')$ we can choose a non-zero coherent subsheaf of $F' \subset f_\ast(F)$ and if $\Hom(V_\alpha,F') \neq 0$, then $\Hom(V_\alpha,f_\ast(F)) \simeq \Hom(V_\alpha|_{\X'},F) \neq 0$.

To check property (CA2'): First note that when $\X$ is defined over an affine base, the vanishing of higher global sections in property (CA2') is equivalent to $\pi_\ast(V_\beta \otimes V_\alpha^\ast) \in \QC(\Spec(R))_{\geq 0}$. By the base-change formula (\Cref{prop:push-CD})
\[ \pi'_* \circ f^*\left( V_\beta \otimes V_\alpha^*\right) \simeq \pi_* \left(V_\beta \otimes V_\alpha^* \right)|_{\Spec(R^\prime)} \]
So it suffices to show that for each $\beta \in I$ there is an $N$ such that $\pi_* \left(V_\beta \otimes V_\alpha^* \right) \in \QC(\Spec(R))_{\geq 0}$ for $\val{\alpha} \geq N$, but this holds by the hypothesis that $\{V_\alpha\}$ is \CA by \Cref{lem:better_hypotheses_CA}.
\end{proof}

\begin{cor} \label{cor:CA_relative_vs_absolute}
Let $\pi : \X \to \Spec(R)$ be a morphism of noetherian algebraic derived (or spectral) stacks which satisfies \CD. A preordered system of locally free sheaves on $\X$ is \CA relative to $\pi$ if and only if it is \CA.
\end{cor}
\begin{proof}
The ``only if" direction is tautological, so we must show that if $\{V_\alpha\}$ is \CA, then for any flat finitely presented morphism $\Spec(R^\prime) \to \Spec(R)$, the preordered system $\{V_\alpha|_{\X_{R^\prime}}\}$ is \CA. This follows from \Cref{lem:base_change_CA}.
\end{proof}

\begin{cor} \label{cor:CA_fppf_local}
Let $\pi : \X \to \S$ be a morphism of noetherian algebraic derived (or spectral) stacks which satisfies \CD. Let $S \to \S$ be an fppf morphism from an affine derived scheme. Then a preordered system $\{V_\alpha\}$ on $\X$ is \CA relative to $\pi$ if and only if $\{V_\alpha|_{\X \times_\S S}\}$ is \CA.
\end{cor}
\begin{proof}
Combine \Cref{lem:descent_for_CA} with \Cref{cor:CA_relative_vs_absolute}.
\end{proof}

\begin{cor} \label{cor:CA_change_base}
Consider a cartesian square of locally noetherian algebraic derived (or spectral) stacks
$$\xymatrix{\X^\prime \ar[r] \ar[d]^{\pi^\prime} & \X \ar[d]^\pi \\ \S^\prime \ar[r] & \S }$$
such that $\pi$ is qc.qs. and satisfies \CD. If a preordered system $\{V_\alpha\}$ is \CA relative to $\pi$, then $\{V_\alpha|_{\X^\prime}\}$ is \CA relative to $\pi^\prime$.
\end{cor}

\begin{proof}
By \Cref{lem:descent_for_CA} we can reduce to the case where $\S'$ and $\S$ are affine, in which case the claim follows immediately from \Cref{cor:CA_relative_vs_absolute} and \Cref{lem:base_change_CA}.

\end{proof}

\begin{prop} \label{prop:CA_relativized}
Let $\pi \colon \X \to \S$ be a morphism of noetherian algebraic derived (or spectral) stacks which satisfies \CD, and assume that $\S$ has affine diagonal. Then a preordered system $\{V_\alpha \}_{\alpha \in I}$ on $\X$ is relatively \CA if and only if
\begin{enumerate}
\item $\forall F \in \Coh(\X)$, $\forall N \geq 0$, there is an $\alpha$ with $\val{\alpha} \geq N$ and $H_0 \pi_\ast(F \otimes V_\alpha^\ast) \neq 0$, and
\item for all $\beta \in I$, $\exists N$ such that $\pi_\ast(V_\beta \otimes V_\alpha^\ast) \in \QC(\S)_{\geq 0}$ whenever $\val{\alpha} \geq N$.
\end{enumerate}
\end{prop}
\begin{proof}
Note that conditions (1) and (2) can be checked after base change along an fppf map $\Spec(R) \to \S$, and so they hold if $\{V_\alpha\}$ is relatively \CA.

To prove the converse, fix an fppf morphism from an affine scheme $S \to \S$ and let $\X^\prime = \X \times_\S S$. By \Cref{cor:CA_fppf_local} it suffices to prove that $\{V_\alpha|_{\X^\prime}\}$ is \CA. $\X^\prime \to \X$ is an affine morphism of noetherian stacks, so the proof of (CA1) from \Cref{lem:base_change_CA} applies verbatim. By \Cref{lem:better_hypotheses_CA} it remains to verify (CA2') for $\{V_\alpha|_{\X^\prime}\}$, and this follows from the base change formula.
\end{proof}

\subsection{Cohomologically projective morphisms are formally proper}

In this section we generalize the proof of the Grothendieck existence theorem for projective morphisms of schemes.

\begin{thm}[Strong Grothendieck Existence] \label{thm:coh_projective_is_proper}
Let $\S$ be a noetherian algebraic spectral (or derived) stack which is complete along a closed subset $Z \subset |\S|$. Then for any cohomologically projective morphism $\pi :\X \to \S$, $\X$ is complete along $\pi^{-1}(Z)$.
\end{thm}

The key idea of the proof is contained in the following lemma, which states that the formal completion of a \CA system of locally free sheaves is again \CA in a suitable sense.

\begin{lem} \label{lem:CA_on_formal_completions}
Let $\pi \colon \X \to \S$ be a cohomologically projective morphism of noetherian algebraic derived (or spectral) stacks, let $\{V_\alpha\}$ be a relatively \CA system, and let $Z \subset |\S|$ be a closed subset. Let $i : \oh{\X} \to \X$ be the completion of $\X$ along $\pi^{-1}(Z)$, and $\oh{\pi} \colon \oh{\X} \to \oh{\S}$ the corresponding projection. Then for all $\oh{F} \in \Coh(\oh{\X})$:
\begin{enumerate}
\item $\forall N \geq 0$, one can find $W \in \Coh(\X)$ which is a finite direct sum of locally free sheaves in the collection $\{V_\alpha | \val{\alpha} \geq N \}$ such that $$i^\ast (W) \otimes \oh{\pi}^\ast \oh{\pi}_\ast (i^\ast (W) \otimes \oh{F}) \to \oh{F}$$ is surjective, meaning that its mapping cone is $1$-connective, and 
\item $\exists N \geq 0$ such that $\oh{\pi}_\ast(i^\ast (V_\alpha)^\ast \otimes \oh{F}) \in \APerf(\oh{\S})^{cn}$ for $\val{\alpha} \geq N$.
\end{enumerate}
\end{lem}

\begin{proof}

The claim is fppf local, so by the base change formula \Cref{prop:push-CD}, it suffices to assume $\S = \Spec(R)$ is affine. First observe that for $\oh{F} \in \APerf(\oh{\X})$, $\oh{\pi}_\ast(\oh{F}) \in \APerf(\Spf(R))$: this is equivalent to showing $\oh{\pi}_\ast(\oh{F})|_{\Spec(R')} \in \APerf(R')$ for any map $\Spec(R') \to \S$ whose image lies in $Z$. It thus suffices to show $\pi_{R'}$ maps $\APerf(\X_{R'})$ to $\APerf(\Spec(R'))$, and this follows from \CD and \CP.

Let $\phi : \Z \to \S$ be the classical reduced closed substack whose support is $Z$. Note that $\oh{G} \in \APerf(\oh{\S})$ is connective if and only if $\phi^\ast(\oh{G}) \in \APerf(\Z)$ is connective. Indeed, writing $\Spf(R) = \colim \Spec(R_n)$ as in \Cref{prop:cplt-affine} (we do not need finite Tor amplitude, so this works in both the spectral and derived context), we see that $\oh{G}$ is connective if and only if $G_n \in \APerf(\Spec(R_n))$ is connective for all $n$, and by Nakayama's lemma this is equivalent to the restriction to $\Spec(\pi_0(R_n)^{red}) \simeq \Z$ being connective.

Now consider the base change $\pi' \colon \X' := \X \times_\S \Z \to \Z$. By the base change formula (\Cref{prop:push-CD}) and \Cref{lem:better_hypotheses_CA} applied to system $\{V_\alpha|_{\X'}\}$, which is \CA by \Cref{cor:CA_change_base}, we have that $\phi^\ast(\oh{\pi}_\ast(\oh{F} \otimes i^\ast V_\alpha^\ast)) \simeq \pi'_\ast(\oh{F} \otimes i^\ast V_\alpha^\ast|_{\X'})$ is connective for $\val{\alpha} \gg 0$ and hence we have property (2) above.

In order to prove (1), it suffices to consider only direct sums of $V_\alpha$ with $\val{\alpha}$ sufficiently large so that (2) holds. We must show that the cone of the canonical morphism in (1) lies in $\APerf(\X)_{\geq 1}$, and because both objects are almost perfect, it suffices by Nakayama's lemma to show this after restricting to $\X'$. The base change formula identifies the restriction of the canonical map in (1) with the canonical map
$$W|_{\X'} \otimes (\pi')^\ast (\pi')_\ast (W|_{\X'} \otimes \oh{F}|_{\X'}) \to \oh{F}|_{\X'}.$$
So again because $\{V_\alpha|_{\X'}\}$ is a \CA system and $\pi'$ is \CD, the fact that the cone of this morphism lies in $\APerf(\X')_{\geq 1}$ follows from \Cref{lem:relative_CA_surjectivity}.

\end{proof}

\begin{proof}[Proof of \Cref{thm:coh_projective_is_proper}]
Let $i : \oh{\X} \to \X$ be the inclusion, and $\pi \colon \oh{\X} \to \oh{\S}$ the projection. \Cref{T:cp_fully_faithful} implies that $i^\ast \colon \APerf(\X) \to \APerf(\oh{\X})$ is fully faithful.

\Cref{lem:CA_on_formal_completions} implies that for any $\oh{F} \in \Coh(\oh{\X})$, we can find a locally free sheaf $W$ on $\X$ such that $\oh{\pi}_\ast (i^\ast(W)^\ast \otimes \oh{F}) \in \APerf(\oh{\S})^{cn}$ and $i^\ast(W) \otimes \oh{\pi}^\ast \oh{\pi}_\ast (i^\ast(W)^\ast \otimes \oh{F}) \to \oh{F}$ is surjective. Because $\S$ is complete, there is a unique $G \in \APerf(\S)^{cn}$ such that $\oh{G} \simeq \oh{\pi}_\ast (i^\ast(W)^\ast \otimes \oh{F})$. Note that $i^\ast(W) \otimes (\oh{\pi})^\ast( \oh{G}) \simeq i^\ast (W \otimes \pi^\ast G)$, and it follows that we have a surjection $i^\ast(H_0(W \otimes \pi^\ast(G))) \to \oh{F}$. This verifies the criterion in \Cref{lem:characterize_complete_closed_immersions} for $i^\ast : \APerf(\X) \to \APerf(\oh{\X})$ to be essentially surjective.

\end{proof}

\subsubsection{Examples of formally proper morphisms}

\begin{prop}\label{prop:projective_over_affine_quotient} Let $\pi : X \to Y$ be a finite type projective-over-affine morphism of noetherian schemes over a field, $k$, and let $G$ be a linearly reductive $k$-group acting on $X$ such that $\pi$ is $G$-invariant and $X$ admits a $G$-linearized ample invertible sheaf over $Y$. Suppose furthermore that $H_0 \pi_\ast(X, \sO_X)^G \in \QC(Y)^\heart$ is coherent.  Then, $\X = X/G \to Y$ is cohomologically projective.
\end{prop}

\begin{proof} Property \CD is immediate, since $X/G \to BG \times Y$ is representable, and $G$ is linearly reductive.

  \medskip
 
  {\noindent}{\it Verification of \CA:} We abuse notation and denote $\pi : X/G \to BG \times Y$, and we let $\L = \sO_X(1)$ be the $G$-linearized ample invertible sheaf on $X$, regarded as an invertible sheaf on $X/G$. Let $I = \bZ_{\geq 0} \times J$, where $J$ indexes the set of irreducible representations, $\rho$, of $G$, and let $V_{n,\rho} := \L^{-n}(\rho)$, where we use the notation $F(\rho)$ to denote the tensor product of $F \in \QC(X/G)$ with the pullback of $\rho \in \op{Irrep}(G)$ regarded as locally free sheaf on $BG$.

We will show that the preordered system $\{V_{n,\rho}\}$ satisfies (CA1) and (CA2) after base change along any map $\Spec(R) \to Y$. Equivalently we may assume that $Y = \Spec(R)$, because the formation of $\{V_{n,\rho}\}$ commutes with base change. For $F \in \Coh(X/G)$, $\pi_\ast(F \otimes \L^n) \in \QC(BG \times Y)_{\geq 0}$ for all $n \gg 0$ because $L$ is ample relative to $\pi$. Because $G$ is linearly reductive this implies that $\RGamma(X/G,F \otimes V_{n,\rho}^\ast) \in \QC(\Spec(R))_{\geq 0}$ for all $\rho$ and $n \gg 0$. Furthermore, if $F \neq 0$, then for any $n$ sufficiently large $H_0 \RGamma(X,F\otimes \L^n) \neq 0$. Therefore for any $N$ we can find an $n \geq N$ and a non-zero morphism $\rho \otimes_k R \to \RGamma(X,F \otimes \L^n)$ in $\QC(BG \times Y)$. It follows that $\RGamma(X/G,F \otimes V_{n,\rho}^\ast) \neq 0$.

  \medskip

{\noindent}{\it Verification of \CP:} We may assume $Y = \Spec(R)$ is affine and verify \CP[R]. Let $X' = \Spec_Y(H_0 \pi_\ast(X, \sO_X))$, so that $q : X \to X'$ is a projective $G$-equivariant map by hypothesis. Suppose $F$ is a $G$-equivariant coherent sheaf on $X$.  We must show that 
\[ H^i(X, F)^G = H^0(Y, H_{-i} \circ q_*  F)^G \] is coherent over $R$ for each $i$.  By the usual coherent pushforward theorem for projective morphisms, $H_{-i} \circ q_*  F$ is coherent and by functoriality it is $G$-equivariant. It is thus enough to show that for any $G$-equivariant coherent $A = H^0(X, \sO_X)$-module $M$, $M^G$ is coherent over $A^G$ (and thus, by our hypotheses, over $R$). This is a classical fact: One uses the existence of a Reynolds operator to show that for any $A^G$-submodule $N \subset M^G$, we have $(A \cdot N)^G = N$, and thus $M^G$ must satisfy the ascending chain condition if $M$ does.
\end{proof}

Next we show that any quotient stack which admits a projective good quotient is cohomologically projective. Recall that if $G$ is a smooth group scheme over a scheme $S$, and $X$ is a $G$-scheme over $S$, then a good quotient \cite{seshadri} is a $G$-invariant map to an algebraic space $q : X \to Y$ such that $q_\ast : \QC(X/G) \to \QC(Y)$ is $t$-exact, and $q_\ast (\cO_X)^G = \cO_Y$. This is a special case of the notion of a good moduli space \cite{Alper}.

\begin{ex}
Let $X$ be a projective variety over a field of characteristic 0. Let $G$ be a reductive group acting on $X$ linearized by an equivariant ample invertible sheaf $\L$. Then the morphism from the stacky GIT quotient to the scheme-theoretic GIT quotient, $X^{ss}(\L) / G \to X/\!\!/\!_{\L} G$, is a good quotient.
\end{ex}

We will need the technical hypothesis on a group scheme $G$ over $S$ that $B_S G$ has enough vector bundles (i.e. every coherent $G$-representation is a quotient of a locally free $G$-representation). This holds for $G = (\op{GL}_N)_S$, or any reductive closed subgroup $G \subset (\op{GL}_N)_S$.

\begin{prop}\label{prop:good-moduli-projective}
Let $S$ be a scheme, and let $G$ be a smooth group scheme over $S$ such that $B_S G$ has enough vector bundles. Let $X$ be a finitely presented $G$-scheme over $S$. If $X$ admits a good quotient $q : X \to Y$ such that $Y$ is projective over $S$, then $\pi \colon X/G \to S$ is cohomologically projective.
\end{prop}

\begin{lem} \label{lem:CA_compose}
Let $\X \xrightarrow{f} \Y \xrightarrow{g} \Z$ be morphisms of noetherian algebraic derived (or spectral) stacks which satisfy \CD. Assume that $f$ satisfies \CP. If $\{V_\alpha\}_{\alpha \in I}$ is a \CA system relative to $f$, and $\{W_\beta\}_{\beta \in J}$ is a \CA system relative to $g$, then 
$$\{V_\alpha \otimes f^\ast W_\beta \}_{(\alpha,\beta) \in I \times J}, \text{ with } \val{\alpha,\beta} := \min(\val{\alpha},\val{\beta})$$
is a \CA system relative to $g \circ f$.
\end{lem}

\begin{proof}
By \Cref{lem:descent_for_CA}, it suffices to assume $\Z$ is affine. We apply \Cref{prop:CA_relativized}, which asks us to consider, for each $F \in \Coh(\X)$, the pushforward
$$E_{\alpha,\beta} = (g \circ f)_\ast (F \otimes V_\alpha^\ast \otimes f^\ast W_\beta^\ast) \simeq g_\ast ( f_\ast( F \otimes V_\alpha^\ast) \otimes W_\beta^\ast).$$
First we must show that there is an $N$ such that $\val{\alpha,\beta} \geq N$ implies $E_{\alpha,\beta} \in \QC(\Z)_{\geq 0}$. We can choose an $N$ such that $\val{\alpha} \geq N$ implies that $f_\ast(F \otimes V_\alpha^\ast)$ is connective. Furthermore it lies in $\Coh(\Y)$ because $f$ satisfies \CP. Thus because $\{W_\beta\}$ is \CA for $g$, we can increase our choice of $N$ such that $\val{\beta} \geq N$ implies that $E_{\alpha,\beta}$ is connective, by \Cref{lem:better_hypotheses_CA}. Clearly both inequalities hold if $\val{\alpha,\beta} = \min(\val{\alpha},\val{\beta}) \geq N$.

To complete the verification that $\{V_\alpha \otimes f^\ast V_\beta\}$ is \CA, one must show that for any $N$ there is some $\val{\alpha,\beta} \geq N$ with $H_0 (E_{\alpha,\beta}) \neq 0$. The argument is the same as that of the previous paragraph.
\end{proof}

\begin{proof}[Proof of \Cref{prop:good-moduli-projective}]
The map $\pi$ is the composition $X/G \to Y \to S$, so $\pi$ satisfies \CD because each of these morphisms do. Property \CP follows from the fact that $q_\ast : \QC(X/G) \to \QC(Y)$ preserves coherence \cite{Alper}*{Theorem 4.16}. Thus we focus on property \CA.

Choose a set of vector bundles $\{W_i\}_{i\in J}$ on $B_S G$ which generate $\QC(B_S G)$. If we define $I = J \times \bZ_{\geq 0}$ and let $V_{i,n} = \cO_X \otimes_S W_i$, then one can check that $\{V_{i,n}\}$ is a \CA system relative to $q \colon X/G \to Y$. Indeed, being a good quotient is local over $Y$, so we may consider $Y = \Spec(R)$. In this case (CA2) is immediate from the exactness of $q_\ast$, and (CA1) follows from the fact that $X \to Y$ is affine.

Define $\pi^\prime \colon Y \to S$, and let $L$ be a relatively ample invertible sheaf for $\pi^\prime$. Then $\{L^{-n} | n \geq 0\}$ is a \CA system for $\pi^\prime$. By \Cref{lem:CA_compose} the composition $\pi = \pi^\prime \circ \phi$ admits a \CA system.
\end{proof}

\subsection{Formal properness via proper covers and \texorpdfstring{$h$}{h}-descent}

Our main application for the descent result \Cref{thm:descent-aperf} is the following ``2-out-of-3'' result for formally proper morphisms.

\begin{thm}\label{thm:descent-apGE-new}
Consider relatively algebraic maps of derived (or spectral) stacks
\[ \X' \xrightarrow{f} \X \xrightarrow{g} \S.\]
If $f$ is surjective and satisfies \CP, $g$ is almost finitely presented, and $g\circ f$ is formally proper (resp. satisfies \CP), then $g$ is formally proper (resp. satisfies \CP).
\end{thm}

\begin{proof}
Note that the hypotheses of the proposition are stable under base change over $\S$, so we may assume that $\S$ is algebraic, noetherian, and complete along a closed subset $Z \subset |\S|$, and it suffices to show that if $\X'$ is complete along the preimage of $Z$, then so is $\X$.

Let $\X'_\bullet := \X'^{\times_\X (\bullet+1)}$ be the Cech nerve of $f$, and let $\oh{\X}$ (respectively $\oh{\X'}_\bullet$) denote the completion of $\X$ (respectively $\X'_\bullet$) along the preimage of $Z$. By \Cref{P:universally_closed} $f$ is an $h$-cover of locally noetherian derived stacks, so \Cref{thm:descent-aperf} implies that restriction induces equivalences
\[ \APerf(\X) \to \Tot\{\APerf(\X'_\bullet)\} \text{ and } \APerf(\oh{\X}) \to \Tot \{\APerf(\oh{\X'}_\bullet)\}. \]
Every map $\X_n' \to \X_0'=\X'$ satisfies \CP, and $\X' \to \S$ satisfies \CP by \Cref{T:cp_fully_faithful}, so each $\X_n' \to \S$ satisfies \CP. \Cref{prop:formal-functions} thus implies that $\APerf(\X'_n) \to \APerf(\oh{\X}'_n)$ is fully faithful for all $n$. Furthermore, this map is an equivalence when $n=0$ because $\X'$ is complete along the preimage of $Z$. It follows that $\Tot \{ \APerf(\X'_\bullet) \} \to \Tot \{ \APerf(\oh{\X}'_\bullet) \}$ is an equivalence as well.\footnote{This is a general fact about cosimplicial $\infty$-categories. If $\cA_\bullet \to \cB_\bullet$ is a map of cosimplicial $\infty$-categories which is fully faithful on each level, then $\Tot \cA_\bullet \to \Tot \cB_\bullet$ is fully faithful. If furthermore $\cA_0 \to \cB_0$ is an equivalence, then we can replace $\cB_n$ with the essential image of $\cB_0 \to \cB_n$ under all face maps without effecting the totalization of $\cB_\bullet$, and likewise for $\cA_\bullet$, and the new map $\cA_\bullet \to \cB_\bullet$ is essentially surjective, hence an equivalence, on every level.}

The argument that $g$ satisfies \CP if $g \circ f$ does is similar: \Cref{T:cp_fully_faithful} say that \CP is equivalent to universally satisfying the fully faithful part of the definition of formal properness. This amounts to showing in the set up above that $\Tot\{\APerf(\X'_\bullet)\} \to \Tot\{\APerf(\oh{\X'}_\bullet)\}$ is fully faithful, which follows from the level-wise fully faithfulness.

\end{proof}

As an immediate consequence of this and \Cref{thm:coh_projective_is_proper}, we have:

\begin{cor}[Chow's lemma implies formally proper] \label{cor:chow_set_up}
Let $\S$ be an algebraic derived stack, and let $\X \to \S$ be an almost finitely presented map. If $\X$ admits a map $\X' \to \X$ which is surjective and representable by proper derived schemes, and $\X'$ is a cohomologically projective $\S$-stack, then $\X \to \S$ is formally proper.
\end{cor}

\begin{ex} \label{ex:many_coh_proper}
\Cref{cor:chow_set_up} provides many examples of formally proper maps which are not cohomologically projective:
\begin{enumerate}
\item Olsson proves in \cite{OlssonProper} that if $S$ is a noetherian scheme and $\pi : \X \to S$ is a proper morphism (in the usual sense), then $\X$ admits a proper covering by a projective $S$-scheme, so $\pi$ is formally proper.
\item If $X$ is a non-normal projective scheme over a field $k$, and $G$ is a linearly reductive $k$-group, then $X$ need not admit a $G$-equivariant embedding into some $\bP^n$, but its normalization $X'$ will. Thus $X' / G$ is cohomologically projective, and $X' / G \to X/G$ is finite and surjective, so $X / G$ is formally proper. The simplest example of this is where $G = \Gm$ and $X$ is the nodal curve obtained by identifying the two fixed points if the action on $\bP^1$.
\item Let $G$ be a smooth group scheme over a perfect field $k$ and $1 \to N \to G \to A \to 1$ be the factorization given by Chevalley's theorem, i.e. $N$ is a connected affine group and $A$ is a finite extension of an Abelian variety. The fiber of $BN \to BG$ is $A$, so if $N$ is linearly reductive (or reductive by \Cref{prop:BG_proper} below), then $BG$ is formally proper.
\item If $G$ is an algebraic $k$-group, $S$ is a noetherian $k$-scheme, and $X$ is a $G$-scheme which admits a good quotient $Y$ which is proper over $S$, then $X/G$ is formally proper over $S$ by combining \Cref{cor:chow_set_up} with Chow's lemma for $Y$ and \Cref{prop:good-moduli-projective}.
\end{enumerate}
\end{ex}

\subsubsection{Example: classifying stack of a reductive group scheme}

\begin{prop} \label{prop:BG_proper}
If $S$ is a Nagata scheme and $\cG \to S$ is a reductive group scheme which admits a dominant cocharacter $\lambda \colon (\bG_m)_S \to G$, then $B_S \cG \to S$ is formally proper.
\end{prop}

\begin{ex}
Every reductive group $G$ over a field $k$ is split after a finite extension $k \subset k'$, so \Cref{prop:BG_proper} implies that $(BG)_{k'} \to \Spec(k')$ is formally proper. Applying \Cref{thm:descent-apGE-new} to the finite morphism $(BG)_{k'} \to BG$ implies that $BG \to \Spec(k)$ is formally proper.
\end{ex}

We will prove the proposition after establishing some preliminary results. The dominant cocharacter $\lambda$ defines a maximal torus $T \subset G$, the centralizer of $\lambda$, and a Borel subgroup $B \subset G$, the subgroup attracted to $T$ under conjugation by $\lambda(t)$ as $t \to 0$. If we let $U = \ker(B \to T)$ be the subgroup attracted to the identity section under this conjugation action, then $U = \Spec_S(A)$ for some smooth quasi-coherent sheaf of algebras $A$ which is a $T$-equivariant coalgebra, non-positively graded by the action of $\lambda(\bG_m)$, and with $A_0 = \cO_S$.

\begin{lem} \label{lem:filtrations_on_BB}
Let $\S$ be a noetherian algebraic derived $S$-stack, and let $F \in \QC(BB_\S)^\heart$, then regarding $F$ as a graded (with respect to $\lambda$) element of $\QC(\S)^\heart$, the submodule $F_{\geq w}$ spanned by summands with $\lambda$-weight $\geq w$ is naturally a $B_\S$-equivariant submodule. If $F$ is coherent, then so is $F_{\geq w}$.
\end{lem}
\begin{proof}
Note that it suffices to consider classical $\S$. The projection $B B_\S \to B T_\S$ is a trivial $B U_\S$ gerbe, and we can identify it with the classifying stack for the smooth affine relative group scheme $U_\S/T_\S \to BT_\S$, where $T$ acts on $U$ by conjugation. We identify $\QC(BB_\S)^\heart$ with the category of $A_\S$-comodules in $\QC(BT_\S)^\heart$. Because $A$ is non-positively graded with respect to $\lambda$, and $A_0=\cO_S$, the image of the comultiplication map $F_{\geq w} \to A_\S \otimes_\S F$ must land in the subsheaf $A_\S \otimes_\S F_{\geq w}$.
\end{proof}

\begin{lem} \label{lem:BB_CP}
The map $\pi \colon BB \to S$ satisfies \CP.
\end{lem}
\begin{proof}
It suffices to assume that $S = \Spec(R)$ is affine, and to show that $\pi_\ast (\Coh(BB)) \subset \DCoh(\Spec(R))$. For this, we may also assume that $R$ is discrete, because pushforward induces an equivalence $\Coh((BB)^{\rm cl}) \simeq \Coh(BB)$.

Let $p_\bullet \colon \Z_\bullet \to B_S B$ denote the Cech nerve of the smooth cover $B_S T \to B_S B$, so we have $\Z_n \simeq U^n / T$. By faithfully flat descent $F \simeq \Tot \{(p_\bullet)_\ast F_\bullet \}$, where $F_n \simeq p_n^\ast F$, and thus $\pi_\ast F \simeq \Tot \{(\pi_\bullet)_\ast F_\bullet\}$ where $\pi_n : \Z_n \to \Spec(R)$ is the projection. We may use the Dold-Kan correspondence to write the totalization as a complex
$$\Tot \{ (\pi_n)_\ast F_n \} \simeq M^0 \to M^1 \to \cdots$$ 
where
$$M^n = \op{coker} \left( \bigoplus_{i=1}^n (\pi_{n-1})_\ast F_{n-1}  \xrightarrow{\delta^n_i} (\pi_n)_\ast F_n \right)$$
and the differential $\delta^n : M^{n-1} \to M^n$ is induced by $\delta_0^n$. 

Because $T$ is linearly reductive, we have $(\pi_n)_\ast F_n \simeq (A^{\otimes_R n} \otimes F_0)^T \in \QC(R)^\heart$. If $F_0$ has highest $\lambda$-weight $<h$, then the fact that $A_0 \simeq R$ and $A$ is non-positively graded implies that $(A^{\otimes_R n} \otimes F_0)^T$ is spanned by simple tensors $a_1 \otimes \cdots \otimes a_n \otimes f$ with $a_i = 1$ for all but at most $h-1$ factors. If $n>2h$, then there must be an $i>0$ with $a_i = a_{i+1} = 1$. It follows from this and the fact that the boundary maps are induced by the comultiplication on $A$ that this element is in the image of a boundary map. Hence $M^n = 0$.

We have thus shown that, assuming $F_0$ has a highest $\lambda$-weight, which is always the case if $F_0$ is coherent, $\pi_\ast(F)$ is computed by a complex involving finitely many $M^n$. Furthermore, $(A^{\otimes n} \otimes F_0)^T$ is coherent for every $n$, because $A$ is a finitely generated $R$-algebra and $T$ is linearly reductive. It follows that $M^n \in \Coh(\Spec(R))$ and hence $\pi_\ast(F) \in \DCoh(\Spec(R))$.
\end{proof}

\begin{lem} \label{lem:single_weight_BB}
Let $\S$ be a noetherian algebraic derived $S$-stack, then the pullback functor $\QC(BT_\S)^\heart \to \QC(BB_\S)^\heart$ for the morphism $BB_\S \to BT_\S$ induces an equivalence between the full subcategories of objects concentrated in a single $\lambda$-weight.
\end{lem}
\begin{proof}
The argument for essentially surjectivity is similar to the proof of \Cref{lem:filtrations_on_BB}, except now because $F$ is concentrated in a single $\lambda$-weight and $A$ is non-positively graded, the comultiplication map $F \to A_\S \otimes_\S F$ must agree with the canonical map induced by $\cO_S = A_0 \to A$. Hence the $A$-comodule structure on $F$ is trivial.

The argument for fully faithfulness is similar to the proof of \Cref{lem:BB_CP}. By faithfully flat descent for the smooth cover $BT_\S \to BB_\S$, we have
\[
\RHom_{BB_\S}(F,G) \simeq \Tot \{ \RHom_\S (F_0,A^{\otimes \bullet} \otimes_\S G_0)^T \},
\]
where $F_0$ and $G_0$ denote the restrictions of $F$ and $G$ to $BT_\S$ respectively. If $F_0$ and $G_0$ are both concentrated in $\lambda$-weight $w$, then the restriction map $\RHom_\S(F_0,G_0)^T \to \RHom_\S (F_0,A^{\otimes n} \otimes_\S G_0)^T$ is an equivalence for all $n$, and the result follows.
\end{proof}

\begin{proof}[Proof of \Cref{prop:BG_proper}]

The morphism $B_S B \to B_S \cG$ is surjective, proper, and representable, hence satisfies \CP. By \Cref{thm:descent-apGE-new} it thus suffices to show that $B_S B \to S$ is formally proper. Let $\S$ be a noetherian derived algebraic $S$-stack which is complete along a closed subset $Z \subset |\S|$, and let $i : \oh{BB_\S} \to BB_\S$ be the completion along the preimage of $Z$. We must show that $i^\ast \colon \APerf(BB_\S) \to \APerf(\oh{BB_\S})$ is an equivalence.

\Cref{lem:BB_CP} and \Cref{prop:formal-functions} implies $i^\ast$ is fully faithful, so by \Cref{lem:characterize_complete_closed_immersions} it suffices to show that $\Coh(BB_\S) \to \Coh(\oh{BB_\S})$ is essentially surjective. Note that fully faithfulness of $i^\ast$ combined with the fact that the natural truncation map is an equivalence $\APerf(\oh{BB_\S})^\heart \simeq \Coh(\oh{BB_\S})$ implies that the essential image of $\Coh(BB_\S) \to \Coh(\oh{BB_\S})$ is closed under extensions.

$\Coh(\oh{BB_\S})$ is the limit of the categories $\Coh(BB_R)$ over all maps $\Spec(R) \to \oh{\S}$. Note that for any $F \in \Coh(BB_R)$ the canonical filtration $\cdots \subset F_{\geq w+1} \subset F_{\geq w} \subset \cdots \subset F$ of \Cref{lem:filtrations_on_BB} is compatible with pullback along any map $\phi \colon \Spec(R') \to \Spec(R)$ in the sense that
$$\left( H_0 \circ \phi^\ast F \right)_{\geq w} = H_0 \circ \phi^\ast (F_{\geq w}),$$
because the inclusion $F_{\geq w} \subset F$ is a summand once we forget the $B$-action. It follows that any $\oh{F} \in \Coh(\oh{BB_\S})$ admits a canonical filtration $\cdots \subset \oh{F}_{\geq w+1} \subset \oh{F}_{\geq w} \subset \cdots \subset \oh{F}$ which restricts to the filtration of \Cref{lem:filtrations_on_BB} along any map $\Spec(R) \to \oh{\S}$. Furthermore, by restricting to the reduced classical closed substack of $\S$ whose underlying subset is $Z$, one can see that this filtration is finite.

We have therefore reduced the claim to showing that $\oh{F}_w / \oh{F}_{w+1}$ lies in the essential image of $\Coh(BB_\S) \to \Coh(\oh{BB_\S})$. For any $\eta : \Spec(R) \to \oh{\S}$, the object $(\oh{F}_w / \oh{F}_{w+1})_\eta \in \Coh(BB_R)$ lies in the full subcategory concentrated in $\lambda$-weight $w$, so \Cref{lem:single_weight_BB} implies that $\oh{F}_w / \oh{F}_{w+1}$ lies in the essential image of the pullback functor $\Coh(\oh{BT_\S}) \to \Coh(\oh{BB_\S})$ along the map $BB_\S \to BT_\S$. It therefore suffices to show that $BT \to S$ is formally proper, so that $\Coh(BT_\S) \to \Coh(\oh{BT_\S})$ is essentially surjective.

$T$ becomes isotrivial (i.e. split after a finite \'etale cover) after pulling back to the normalization of $S$ \cite{conrad}*{Cor.~B.3.6}. Thus $T$ is split after pullback to a finite (because $S$ is Nagata) cover $S' \to S$. Then $(BT)_{S'} \to S'$ is cohomologically projective by \Cref{prop:good-moduli-projective} and hence formally proper by \Cref{thm:coh_projective_is_proper}. We then apply \Cref{thm:descent-apGE-new} to the proper cover $BT_{S'} \to BT_S$ to deduce that $BT_S \to S$ is formally proper.

\end{proof}

%%%%%%%%%%%%%%%%%%%%%%%%%%%%%%%%%%%%%%%%%%
%%%%%% From (?GE)+(L) to Maps        %%%%%
%%%%%%%%%%%%%%%%%%%%%%%%%%%%%%%%%%%%%%%%%%
\section{Applications}
\label{sect:mapping_stacks}

\subsection{Algebraicity of mapping stacks}

Recall that for derived $\S$-stacks $\X$ and $\Y$, the mapping prestack is defined by the functor of points
$$\inner{\Map}_\S(\X,\Y) : T \mapsto \Map_\S(\X \times_\S T, \Y).$$
where $T$ is a derived affine scheme over $\S$.\footnote{Note that in order to be a prestack the $\infty$-groupoid $\Map_\S(\X \times_\S T, \Y)$ must be essentially small. This need not be the case in general, but we will see that it is in our situation of interest.}

For our main result of this section, we will fix a base algebraic stack $\S$ which admits a smooth surjection from a disjoint union of derived schemes of the form $\Spec(A)$, where $A$ is a derived $G$-ring which admits a dualizing module (i.e. $A$ is noetherian and $\pi_0(A)$ is a $G$-ring which admits a dualizing complex \cite{DAG-XIV}*{Thm.~4.3.5}). We will say that a map of derived stacks is quasi-affine if it is representable by derived algebraic spaces and its restriction to discrete simplicial commutative rings is quasi-affine.\footnote{In the spectral setting, quasi-affine stacks can be recovered from $\RGamma(\cO_\X,\X)$, which allows for a more intrinsic definition. See \cite{DAG-VIII}*{Prop.~2.4.8}.}

\begin{thm} \label{thm:derived_mapping_stacks}
Let $\S$ be an algebraic derived stack as above, and let $\Y$ be a locally almost finitely presented algebraic derived stack over $\S$ whose diagonal $\Y \to \Y \times_\S \Y$ is quasi-affine. Let $\pi : \X \to \S$ be a formally proper morphism of Tor-amplitude $n$. Then the mapping stack $\inner{\Map}_\S(\X,\Y)$ is a locally almost finitely presented algebraic $n+1$-stack over $\S$, in the sense of \cite{Lurie-Thesis}*{Def.~5.1.3}. Furthermore if $\pi$ is flat, so $\inner{\Map}_\S(\X,\Y)$ is an algebraic stack, then $\inner{\Map}_\S(\X,\Y) \to \S$ has quasi-affine diagonal, and it has affine diagonal if $\Y \to \S$ has affine diagonal.
\end{thm}

\begin{rem}
In fact our proof uses only the two apparently weaker properties: $\pi$ satisfies property \LL below, and for any complete \emph{local} noetherian classical ring $R$ over $\S$, the stack $\X \times_\S \Spec(R)$ is complete along the closed substack defined by the maximal ideal of $R$.
\end{rem}

\begin{rem}[Classical versus derived mapping stacks]\label{rem:classical_vs_derived}
There is a slightly stronger form of Tannaka duality available in the classical setting \cite{hall2014coherent}, which allows one to deduce the algebraicity of $\Map_\S(\X,\Y)$ when $\pi : \X \to \S$ is a flat map of classical stacks satisfying the hypotheses of \Cref{thm:derived_mapping_stacks} and when $\Y$ merely has affine stabilizers over $\S$ as in \cite{hall2014coherent}. In addition, when $\pi$ is flat, the algebraicity of the classical mapping stack is equivalent to the algebraicity of the derived mapping stack by \cite{HAG-II}*{Theorem C.0.9} and \Cref{prop:mapping_cotangent_complex} below, because $\Map_\S(\X,\Y)^{\rm cl} \simeq \Map_{\S^{\rm cl}}(\X^{\rm cl},\Y^{\rm cl})$. So by using classical Artin's criteria instead of the derived criteria, one could analogously strengthen \Cref{thm:derived_mapping_stacks} when $\pi$ is flat.
\end{rem}

\begin{rem}[Noetherian hypotheses]\label{rem:noetherian}
Also in the classical setting, when $\X \to \S$ is proper one can remove the noetherian hypotheses on $\S$ using relative noetherian approximation. We do not carry this out here, because we do not know that a general formally proper map $\X \to \Spec(A)$ admits a formally proper model over a finitely generated $\bZ$-algebra (even in the classical context). However, if $\S$ is a stack satisfying the assumptions of \Cref{thm:derived_mapping_stacks} and $\pi : \X \to \S$ is a flat and formally proper morphism, then for any derived stack $\S'$ with a map $\S' \to \S$ and a map of derived stacks $\Y \to \S'$ which is relatively algebraic and locally almost of finite presentation with quasi-affine diagonal, one can apply Noetherian approximation to the underlying classical stack of $\Y$ to deduce that $\Map_{\S'}(\X\times_\S \S', \Y)$ is algebraic and locally almost of finite presentation with quasi-affine diagonal over $\S'$.
\end{rem}

Let us recall the derived version of Artin's representability criteria introduced by Lurie \cite{Lurie-Thesis} (see also \cite{Pridham_artin}). For a functor $\sF : \SCR_A \to \Sp$ which satisfies \'{e}tale descent, the Artin-Lurie criteria are:
\begin{enumerate}
\item \emph{locally almost of finite presentation}: $\sF$ commutes with filtered colimits of $k$-truncated objects in $\SCR_A$, for any $k$ \\
\item \emph{$n$-truncated}: $\sF(R)$ is $n$-truncated for any discrete $R\in \SCR_A$\\
\item \emph{admits a co-representable deformation theory}:
\begin{enumerate}
\item \emph{admits a $(-n)$-connective cotangent complex}: there is an $L_\sF \in \QC(\X)_{\geq -n}$ such that for any $R \in \SCR_A$ and $\eta : \Spec(R) \to \sF$ corresponding to a point in $\sF(R)$,
$$\sF(R \oplus M) \times_{\sF(R)} \{\eta\} \simeq \Omega^\infty R\Hom_R (\eta^\ast L_\sF, M)$$
The existence of $L_\sF$ is not automatic, but it is defined up to canonical isomorphism when it exists. \\
\item \emph{infinitesimally cohesive} and \emph{nilcomplete}: See \Cref{lem:descent-nil} conditions (i) and (ii) \\
\end{enumerate}
\item \emph{integrable}: for any discrete complete local noetherian $R$ over $A$,  $\sF(R) \to \Map(\Spf(R),\sF)$ is an equivalence. \\
\end{enumerate}

The main result of \cite{Lurie-Thesis} is the following:\footnote{A final version of this theorem in the spectral setting will also appear in \cite{SAG}.}
\begin{thm} \label{thm:derived_artin}
Let $A$ be a derived $G$-ring and let $\sF \in \Fun(\SCR_A,\S)$ be a functor which satisfies \'{e}tale descent. Then $\sF$ is a derived algebraic $n$-stack locally almost of finite presentation over $\Spec(A)$ if and only if the Artin-Lurie criteria (1)-(4) above hold.
\end{thm}

Our proof of \Cref{thm:derived_mapping_stacks} closely follows the proof of the analogous result in the case where $\X$ is a flat and proper algebraic space in \cites{Lurie-Thesis, DAG-XIV}. Namely, we will apply \Cref{thm:derived_artin} after first establishing (3) and (4), whose proofs require some elaboration.

\subsubsection{Left adjoint for the pullback functor}
\label{ssec:LL}

For a flat and proper morphism of schemes $\pi:X \to S$, the pullback functor $\pi^\ast:D_{qc}(S) \to D_{qc}(X)$ admits a left adjoint $\pi_+$ with $\pi_+(F) \simeq (\pi_\ast(F^\dual))^\dual$ for perfect complexes, and defined in general by writing any $F$ as a filtered colimit of perfect complexes. In fact, for a morphism of perfect stacks (in the sense of \cite{BFN}), the existence of $f_+$ is equivalent to the property that $f_\ast$ preserves perfect complexes, and the same formula for $f_+$ applies. More generally we have:

\begin{prop} \label{prop:property_LL_from_CP}
Let $\X$ and $\S$ be locally noetherian algebraic derived stacks, and let $\pi : \X \to \S$ be a qc.qs. morphism of Tor amplitude $d < \infty$ satisfying \CP. Assume that locally $\S$ admits a dualizing complex. Then $\pi^\ast : \QC(\S) \to \QC(\X)$ admits a left adjoint, $\pi_+$, which maps $\QC(\X)_{\geq 0}$ to $\QC(\X)_{\geq -d}$ and preserves almost perfect complexes.
\end{prop}

The key is to reduce the claim to the existence of an adjoint on small categories.

\begin{lem} \label{L:small_category_left_adjoint}
Let $\pi : \X \to \S$ be a morphism of Tor amplitude $d$ between noetherian algebraic derived stacks. Then $\pi^\ast : \QC(\S) \to \QC(\X)$ admits a left adjoint if and only if $\pi^\ast : \APerf(\S) \to \APerf(\X)$ admits a left adjoint.
\end{lem}

\begin{proof}
If $\pi_+$ exists, then for and $F \in \APerf(\X)$, $\RHom_\S(\pi_+(F),-) \simeq \RHom_\X(F,\pi^\ast (-))$ commutes with filtered colimits in $\QC(\S)_{\leq n}$, because $\pi^\ast : \QC(\S)_{\leq n} \to \QC(\X)_{\leq d+n}$ commutes with filtered colimits. It follows that $\pi_+(F) \in \APerf(\S)$, and thus $\pi_+$ is a left adjoint to $\pi^\ast : \APerf(\S) \to \APerf(\X)$.

Conversely, assume that $\pi^\ast : \APerf(\S) \to \APerf(\X)$ admits a left adjoint. Let $\cC \subset \QC(\X)$ be the full subcategory of objects such that $\Map(F,\pi^\ast(-))$ is corepresentable in $\QC(\S)$. Because $\pi$ has finite Tor amplitude, $\Map(F,\pi^\ast(-))$ commutes with the limit $M = \ilim \tau_{\leq n} M$ for any $M\in \QC(\S)$. It follows that $F \in \cC$ if and only if $\exists G \in \QC(\S)$ such that for all $n$,
\begin{equation} \label{E:criterion}
\Map(\tau_{\leq n+d}(F),\pi^\ast(-)) \simeq \Map(\tau_{\leq n}(G),-) \text{ as functors on } \QC(\S)_{\leq n}.
\end{equation}

For any $F \in \APerf(\X)$ and $n \in \bZ$, \Cref{thm:coh-gen} implies that $\Map(F,\pi^\ast(-)) \simeq \Map(\pi_+(F),-)$ as functors on $\QC(\S)_{\leq n}$, because they agree on $\DCoh(\S)_{\leq n}$ and both commute with filtered colimits. So the criterion \eqref{E:criterion} implies that $\APerf(F) \subset \cC$. $\cC$ is closed under colimits and contains $\DCoh(\X)_{\leq n}$ for all $n$, so \Cref{thm:coh-gen} implies that $\QC(\X)_{\leq n} \subset \cC$ for all $n$.

Finally, given a tower $\cdots \to F_2 \to F_1 \to F_0$ in $\cC$ such that $\tau_{\leq k}(F_i)$ is eventually constant for any $k$ as $i \to \infty$, the corresponding tower $\cdots \to G_2 \to G_1 \to G_0$ of corepresenting objects $G_i = \pi_+(F_i)$ has the same property because $\pi^\ast$ has finite Tor amplitude. Then $G:= \lim_n G_n \in \QC(\S)$ satisfies the criterion \eqref{E:criterion} for $F = \lim_n F_n$. Writing any $F\in \QC(\S)$ as $F \simeq \lim_n \tau_{\leq n}(F)$ shows that $\cC = \QC(\S)$.

\end{proof}

\begin{lem}[``Base-change for $p_+$'']\label{lem:base_change_left_adjoint} Suppose we are given a cartesian square of prestacks
  \[ \xymatrix{
  \X' \ar[d]_{p'} \ar[r]^{q'} & \X \ar[d]^p \\
\S' \ar[r]_q & \S
  } \]
such that $q$ is relatively representable by qc.qs. algebraic derived stacks satisfying \CD. Assume that $p^\ast : \QC(\S) \to \QC(\X)$ admits a left adjoint $p_+$ and likewise for $(p^\prime)^\ast$. Then the canonical base change morphism is an isomorphism of functors
\[ (p')_+ (q')^*  \xrightarrow{\simeq} q^* p_+  \]
\end{lem}
\begin{proof} The base-change isomorphism between these functors is induced by the base-change isomorphism of their right adjoints, $p^* q_* \xrightarrow{\simeq}  (q')_* (p')^*$ provided by \Cref{prop:push-CD}.
\end{proof}

\begin{proof}[Proof of \Cref{prop:property_LL_from_CP}]
Let $U_\bullet \to \S$ be an fppf hypercover of $\S$ by derived schemes, and let $\X_\bullet \to \X$ be the base change along $\pi : \X \to \S$. If for each $n$, the functor $\pi_n^\ast : \QC(U_n) \to \QC(\X_n)$ admits a left adjoint $(\pi_n)_+$, then \Cref{lem:base_change_left_adjoint} implies that applying the functors $(\pi_n)_+$ level-wise defines a functor $(\pi_\bullet)_+ : \Tot \{\QC(\X_\bullet)\} \to \Tot\{ \QC(U_\bullet)\}$. It is straightforward to check that $(\pi_\bullet)_+$ is left adjoint to $(\pi_\bullet)^\ast$, which corresponds to $\pi^\ast$ under the equivalences $\QC(\X) \simeq \Tot\{ \QC(\X_\bullet)\}$ and $\QC(\S) \simeq \Tot \{\QC(U_\bullet)\}$. We may assume that $U_n$ is a union of affine schemes, for which the existence of a left adjoint $(\pi_n)_+$ is equivalent to the existence of a left adjoint for $(\pi_n)^\ast$ over each of the affine schemes comprising $U_n$. Furthermore if $\S$ locally admits a dualizing complex, then we can arrange that each affine scheme appearing in $U_n$ admits a dualizing complex by \cite{Lurie-Thesis}*{Thm.~3.6.8}(see also \cite{DAG-XIV}*{Thm.~4.3.14}). Thus we have reduced to the case where $\S = \Spec(R)$ for some noetherian simplicial commutative ring $R$ which admits a dualizing complex.

By \Cref{L:small_category_left_adjoint} it suffices to show that the functor $h(\bullet) = \Omega^\infty \RHom_\X(F, \pi^\ast (-))$ is corepresentable for any $F \in \APerf(\X)$. We use \cite{DAG-XIV}*{Thm.~4.4.2} which gives five conditions which guarantee that $h(\bullet)$ is corepresentable by an almost perfect complex:
\begin{enumerate}
\item That $h(0)$ is contractible, which is immediate, and $h$ maps pushout squares to pullback squares as $\pi^\ast$ is an exact functor of stable $\infty$-categories and $\RHom_\X(F,\bullet)$ takes pushout squares to pullback squares.
\item The canonical map $h(M) \mapsto \ilim h(\tau_{\leq n} M)$ is an equivalence -- because $\pi$ has finite Tor dimension, $\pi^\ast \ilim \tau_{\leq n} M \simeq \ilim \pi^\ast \tau_{\leq n} M$, and $\RHom(F,\bullet)$ preserves limits.
\item $h$ commutes with filtered colimits in $\QC(R)^{cn}_{\leq n}$ -- the functor $\pi^\ast$ commutes with colimits and maps $\QC(R)^{cn}_{\leq n}$ to $\QC(\X)^{cn}_{\leq n+d}$ for some $d$, and $\RHom(F,\bullet)$ commutes with filtered colimits in $\QC(\X)^{cn}_{\leq n+d}$ because $F$ is almost perfect.
\item There exists an integer $n\geq 0$ such that $h(M)$ is $n$-truncated for every discrete $R$-module, $M$ -- $\pi^\ast (\QC(R)^\heart) \subset \QC(\X)^{cn}_{\leq d}$, so $H_i(\RHom(F,\pi^\ast(M))) = 0$ for $i>d-\min \{j | H_j(F) \neq 0\}$.
\end{enumerate}
The fifth condition requires that for any coherent $R$-module, $M$, the module $\pi_0 h(M)$ is finitely generated as a module over $\pi_0 R$. We rewrite $h(M) \simeq \Omega^\infty \pi_\ast H$, where $H:= \RHom^{\otimes_{\QC(\X)}}_\X(F,\pi^\ast M) \in \DDAPerf(\X)$ by \Cref{lem:bounded_inner_hom}. In particular $\pi_\ast H \in \DDAPerf(R)$ by property \CP, and $H_i \pi_\ast H$ is a coherent $\pi_0 R$-module for all $i$.
\end{proof}

\subsubsection{Cotangent complex of the mapping stack}
\label{sect:Artin_deformation_thy}

\begin{defn}\label{property:L}
Let $\X$ be a stack over a noetherian affine derived scheme, $\Spec(R)$. We introduce the property
\smallskip
\begin{itemize}[leftmargin=2cm]
\item[$(L)_R$:] The pullback functor $f^\ast : \QC(\Spec(R)) \to \QC(\X)$ admits a left adjoint $f_+$.
\end{itemize}
\smallskip
We say that a morphism of stacks $f : \X \to \S$ satisfies \LL if for any noetherian affine derived scheme $\Spec(R) \to \S$, the base change $\X \times_\S \Spec(R) \to \Spec(R)$ satisfies \LL[R].
\end{defn}

\begin{prop} \label{prop:mapping_cotangent_complex}
Let $S$ be an affine derived scheme, let $\X$ and $\Y$ be derived $S$-stacks such that $\pi : \X \to S$ satisfies \LL and has finite Tor amplitude. If $\Y$ admits a cotangent complex then so does $\inner{\Map}_S( \X,\Y)$. For any map $\eta : \Spec(A) \to \inner{\Map}_S(\X,\Y)$ classifying a map $f : \X_A \to \Y$ over $S$, we have
\[
\eta^\ast (L_{\inner{\Map}_S( \X,\Y) / S}) \simeq (\pi_A)_+ f^\ast L_{\Y/S}.
\]
\end{prop}

First we observe that the formation of the split square-zero extension $A \oplus M$ is compatible with pullback in the sense that for any map of rings $A \to B$, corresponding to a map $\Spec(B) \to \Spec(A)$, we have $B \otimes_A (A \oplus M) \simeq B \oplus (B \otimes_A M)$. This motivates the definition of the trivial square-zero extension $\X[F]$ for any functor $\X$ and any $F \in \QC(\X)^{cn}$ by the fiber square
\begin{equation} \label{eqn:def_trivial_sq_zero_ext}
\xymatrix{\Spec(A \oplus a^\ast F) \ar[r] \ar[d] & \X[F] \ar[d] \\ \Spec(A) \ar[r]^a & \X}.
\end{equation}
In other words
$$\X[F](A) := \left\{ a \in \X(A) \text{ and a section of } \Spec(A \oplus a^\ast F) \to \Spec(A) \right\},$$
where the space of sections of $\Spec(A \oplus a^\ast F) \to \Spec(A)$ can further be identified with $\Omega^\infty R \Hom_A (L_A, a^\ast F)$.

\begin{lem} \label{lem:def_of_maps}
Let $S = \Spec(A)$, and let $\X,\Y : \SCR_{A} \to \Sp$ be functors such that $\Y$ admits a cotangent complex. Then for any map $f : \X \to \Y$ over $S$ there is a canonical isomorphism
$$\Map_S(\X[F],\Y) \times_{\Map_S(\X,\Y)} \{f\} \simeq \Omega^\infty R\Hom_\X(f^\ast L_{\Y/S}, F)$$
as functors $\QC(\X)^{cn} \to \Sp$
\end{lem}

\begin{proof}
Let $\op{Aff} / \X[F]$ denote the $\infty$-category of affine schemes along with a morphism to $\X[F]$. By the $\infty$-categorical Yoneda lemma, we have
$$\X[F] = \op{colim} \limits_{T \in \op{Aff} / \X[F]} T$$
By the canonical fiber square \eqref{eqn:def_trivial_sq_zero_ext}, we have a functor $\op{Aff} / \X \to \op{Aff} / \X[F]$ mapping $T \mapsto T \times_{\X} \X[F] \simeq T[F|_T]$. This functor is cofinal because any morphism $T \to \X[F]$ factors canonically through $T \times_\X \X[F] \to \X[F]$, and this factorization is initial in the category of factorizations $T \to T^\prime [F|_{T^\prime}] \to \X[F]$ for varying $T^\prime$. Thus we can write $\X[F]$ as a colimit over $\op{Aff} / \X$
$$\X[F] = \op{colim} \limits_{\eta : T \to \X} T[\eta^\ast F]$$
Hence on mapping spaces of presheaves we have
\begin{gather*}
\Map_S(\X[F],\Y) = \lim \limits_{\substack{(\op{Aff}/\X)^{op} \\ \eta : T \to \X }} \Y(T[\eta^\ast F]), \text{ whereas}\\
\Map_S(\X,\Y) = \lim \limits_{\substack{(\op{Aff}/\X)^{op} \\ \eta : T \to \X }} \Y(T)
\end{gather*}
Taking fibers commutes with limits, so 
$$\Map_S(\X[F],\Y) \times_{\Map_S(\X,\Y)} \{f\} \simeq \lim \limits_{\substack{(\op{Aff}/\X)^{op} \\ \eta : T \to \X }} \Omega^\infty R\Hom_T(\eta^\ast f^\ast L_{\Y/S}, \eta^\ast F)$$
Where we have used the defining property of $L_{\Y / S}$ as corepresenting the fiber of the map $\Y(T[-]) \to \Y(T)$ for affine schemes $T$. This last expression is essentially the definition of $\Omega^\infty R\Hom_{\QC(\X)}(f^\ast L_{\Y / S},F)$.
\end{proof}

\begin{proof} [Proof of \Cref{prop:mapping_cotangent_complex}]
Let $\eta\in \fM(A)$ correspond to an affine derived scheme $\Spec(A)$ over $S$, together with a map $f : \X_A \to \Y$ over $S$.  Let $M \in \QC(A)^{cn}$. Then by definition $\fM(A \oplus M) = \Map_S(\X_{A \oplus M} , \Y)$. If $\pi_A : \X_A \to \Spec(A)$ is the structure morphism, then $\X_{A \oplus M} \simeq \X_A[\pi_A^\ast M]$ over $\Spec(A)$, by the construction of the trivial square-zero extension of functors. Hence by Lemma \ref{lem:def_of_maps} we have a canonical isomorphism
$$\fM(A \oplus M) \times_{\fM(A)} \{f\} \simeq \Omega^\infty R\Hom(f^\ast L_{\Y / S}, \pi_A^\ast M)$$
By hypothesis \hyperref[property:L]{(L)}, the functor $\pi_A^\ast$ has a left adjoint $(\pi_A)_+$, hence we can define $L_{\fM/S}|_{\Spec(A)} := (\pi_A)_+ f^\ast L_{\Y/S}$. For any map of rings $\phi : A \to B$ we have the pullback square
$$\xymatrix{\X_B \ar[d] \ar[r]^{\phi^\prime} & \X_A \ar[r]^f \ar[d] & \Y \\ \Spec(B) \ar[r]^\phi & \Spec(A) & }$$
By \Cref{lem:base_change_left_adjoint}, we have a natural isomorphism $\phi^\ast  (\pi_A)_+ f^\ast L_{\Y/S} \simeq (\pi_B)_+ (\phi^\prime \circ f)^\ast L_{\Y/S}$. Hence the assignment $\Spec(A) / \fM \mapsto (\pi_A)_+ f^\ast L_{\Y/S}$ determines an object of $L_{\fM / S} \in QC(\fM)$.

\end{proof}

\subsubsection{Integrability via the Tannakian formalism}
\label{sect:integrability}

We recall the following version of Tannaka duality in the setting of spectral algebraic geometry, which is a refinement of Lurie's Tannaka duality theorem \cite{DAG-VIII}*{Theorem 3.4.2}:
\begin{thm}[Theorem 5.1 and Lemma 3.13 of \cite{BhattHL}] \label{thm:Tannaka}
Let $\Y$ be a noetherian algebraic spectral stack with quasi-affine diagonal. Then for any prestack $\S$ over $\CAlg^{cn}$, the association $f \mapsto f^\ast$ gives an equivalence of $\infty$-categories
\[
\op{Map}(\S,\Y) \to \Fun_{\otimes}^c(\APerf(\Y)^{cn},\APerf(\S)^{cn}),
\]
where the latter denotes symmetric monoidal functors of symmetric monoidal $\infty$-categories which preserve finite colimits.
\end{thm}

Integrability of the mapping stack functor will follow from the slightly more general fact:

\begin{prop} \label{prop:integrability}
Let $\X$ be a noetherian algebraic derived stack over a noetherian affine derived scheme $S$, and let $\oh{\X}$ be the formal completion of $\X$ along a closed substack. If $\APerf(\X) \to \APerf(\oh{\X})$ is an equivalence of $\infty$-categories, then for any locally noetherian algebraic derived $S$-stack $\Y$ with quasi-affine diagonal, the canonical functor
$$\op{Map}_S ( \X, \Y) \to \op{Map}_S ( \oh{\X} , \Y)$$
is an equivalence of $\infty$-groupoids.
\end{prop}

\begin{proof}
If we write $\Y$ as a union of quasi-compact open substacks $\Y = \bigcup_\alpha \Y_\alpha$, then both mapping $\infty$-groupoids are filtered unions of the mapping groupoids into $\Y_\alpha$, so it suffices to replace $\Y$ with $\Y_\alpha$ and assume that $\Y$ is noetherian. Also note that it suffices to prove the claim for both $\Map(\X,\Y)$ and $\Map(\X,S)$, so we may restrict to just the absolute case.

\medskip
\noindent \textit{Step 1: The spectral version:}
\medskip

Consider the variant of the statement of \Cref{prop:integrability} in which the phrase ``derived stack'' is replaced with ``spectral stack.'' Then the restriction functor $\APerf(\X) \to \APerf(\oh{\X})$ is a $t$-exact equivalence and therefore induces an equivalence of symmetric monoidal $\infty$-categories $\APerf(\X)^{cn} \simeq \APerf(\oh{\X})^{cn}$. So \Cref{thm:Tannaka} immediately implies that $\Map(\X,\Y) \to \Map(\oh{\X},\Y)$ is an equivalence of $\infty$-categories.

\medskip
\noindent \textit{Step 2: The classical version:}
\medskip

The $\infty$-category $\op{Shv}(\op{Ring}^{op})$ of \'etale sheaves on the category of classical affine schemes admits two fully faithful left Kan extension functors
\[
\xymatrix{\op{Shv}(\SCR^{op}) & \op{Shv}(\op{Ring}^{op}) \ar[r]^{LK_{E_\infty}} \ar[l]_{LK_{\SCR}} & \op{Shv}(\CAlg^{cn,op})}.
\]
These functors commute with the formation of $\APerf(-)$, the underlying topological space, and formal completions -- indeed, these claims reduce immediately to the case of representable functors (i.e. affine schemes) where they are evident. The left Kan extension functors also preserve Noetherian algebraic stacks. Finally, a derived stack or spectral stack is quasi-affine if and only if its underlying classical stack is quasi-affine, and it follows that a stack has quasi-affine quasi-affine diagonal if and only if its left Kan extension does.

In particular, if $\X$ and $\Y$ are classical stacks, and $\X^{\rm sp}$ and $\Y^{\rm sp}$ the associated spectral stacks, then the condition that $\APerf(\X) \to \APerf(\oh{\X})$ is an equivalence implies that $\APerf(\X^{\rm sp}) \to \APerf(\oh{\X^{\rm sp}})$ is an equivalence, so the fact that $$\Map_{\op{Shv}(\op{Ring}^{op})}(\X,\Y) \to \Map_{\op{Shv}(\op{Ring}^{op})}(\oh{\X},\Y)$$ is an equivalence follows from the fully-faithfulness of the left Kan extension.

\medskip
\noindent \textit{Step 3: The derived version:}
\medskip

Now let $\X^{\circ}$ and $\Y$ be noetherian algebraic derived stacks, and assume $\Y$ has quasi-affine diagonal, $\X^{\circ}$ is the left Kan extension $LK_{\SCR}$ of a classical algebraic stack, and $\APerf(\X^{\circ}) \to \APerf(\oh{\X^\circ})$ is an equivalence. We will show that for any algebraic derived stack $\X$ admitting a surjective closed immersion $\X^\circ \to \X$, $\Map(\X,\Y) \to \Map(\oh{\X},\Y)$ is an equivalence. This will complete the proof, because any $\X$ satisfying the hypotheses of the proposition admits a canonical surjective closed immersion from $\X^\circ = \X^{\rm cl}$, and $\APerf(\X^\circ) \to \APerf(\oh{\X^\circ})$ is an equivalence by \Cref{lem:completeness_nil}.

Let $\cC \subset \op{Shv}(\SCR^{op})_{\X^{\circ}/}$ be the largest full subcategory consisting of algebraic derived stacks $\X$  admitting a surjective closed immersion $\X^{\circ} \to \X$, and let $\cC' \subset \cC$ be the full subcategory of stacks such that $\Map(\X,\Y) \to \Map(\oh{\X},\Y)$ is an equivalence. Note that \Cref{lem:completeness_nil} implies that $\APerf(\X) \to \APerf(\oh{\X})$ is an equivalence for any $\X \in \cC$.

For the initial object $\X^{\circ} \in \cC$, $\oh{\X^{\circ}}$ is the left Kan extension the completion of the underlying classical stack $\X^\circ|_{\op{Ring}^{op}}$, so it suffices by adjunction to replace $\Y$ with $\Y^{\rm cl}$. Note also that we have $\APerf(\X^\circ|_{\op{Ring}^{op}}) \to \APerf(\oh{\X^\circ|_{\op{Ring}^{op}}})$ is an equivalence of $\infty$-categories, so $\X^{\circ} \in \cC'$ by Step (2).

If $\X \in \cC'$ and $F \in \APerf(\X)^{cn}$, then we claim that $\X[F] \in \cC'$ as well. Indeed, it suffices to show that the map $\Map(\X[F],\Y) \to \Map(\oh{\X[F]},\Y)$ induces an equivalence on the fibers over each point $\{f\} \in \Map(\X,\Y) \simeq \Map(\oh{\X},\Y)$. Given a map $f:\X \to \Y$, \Cref{lem:def_of_maps} identifies the map of fibers over $\{f\}$ with the restriction map
\[
\Omega^\infty \RHom_{\X}(f^\ast(\bL_\Y),F) \to \Omega^\infty \RHom_{\oh{\X}}(\oh{f}^\ast(\bL_\Y),F|_{\oh{\X}})
\]
where $\oh{f}$ denotes the composition of $f$ with the inclusion $\oh{\X} \to \X$. This restriction map is an equivalence because $\X \in \cC'$.

Now let $\X \in \cC$ be a $k$-truncated derived stack, and let $\X' = \tau_{\leq k-1}(\cX)$ be the $k-1$-truncation of $\X$. Recall that in the affine case, if $A \in \SCR$ is $k$-truncated, $A' = \tau_{\leq k-1}(A)$, and $M = \pi_k(A)[k+1]$, then we have a canonical identification
\[
A = A' \times_{0,A' \oplus M,\eta} A',
\]
where $\eta : A' \to A' \oplus M$ is induced by the fiber sequence of $A$-modules $\pi_k(A)[k] \to A \to A'$. The formation of this fiber square commutes with smooth extensions of $A$, so applying it to a smooth hypercover of $\X$ by affine derived schemes gives a canonical commutative diagram of derived stacks realizing $\X$ as a square-zero extension of $\X'$:
\begin{equation}
\xymatrix{ \X'[\pi_k(\cO_\X)[k+1]] \ar[r]^-0 \ar[d]^\eta & \X' \ar[d] \\ \X' \ar[r] & \X}.
\end{equation}

We claim that this diagram, as well as its formal completion, is a pushout square in the $\infty$-category of stacks which are hypersheaves for the smooth topology and are infinitesimally cohesive. Indeed, by hyperdescent it suffices to prove this claim after base change along a smooth map $\Spec(A) \to \X$ (or $\Spec(A) \to \oh{\X}$), at which point this follows immediately from the definition of infinitesimally cohesive. Therefore if $\X' \in \cC'$, we have already shown $\X'[\pi_k(\cO_\X)[k+1]] \in \cC'$, and hence $\X \in \cC'$ by the above pushout square.

By induction we now see that any $k$-truncated $\X \in \cC$ also lies in $\cC'$. Finally, an arbitrary $\X \in \cC$ is the colimit of the truncations $\X^\circ \to \tau_{\leq 1}(\X) \to \tau_{\leq 2}(\X) \to \cdots$ in the category of hypersheaves for the smooth topology which are nilcomplete, and the same is true for the formal completions. It follows that any $\X \in \cC$ lies in $\cC'$.

\end{proof}

\subsubsection{Analysis of the Weil restriction}

\begin{prop} \label{P:weil_restrict}
Let $\pi : \X \to \S$ be a flat morphism of algebraic derived stacks satisfying property \LL, with $\S$ noetherian, and let $\sF \to \X$ be a quasi-affine (resp. affine) morphism. Then the Weil restriction
\[
\pi_\ast(\sF/\X) := \inner{\Map}_\S(\X,\sF) \times_{\inner{\Map}_\S(\X,\X)} \{\id\}
\]
is a quasi-affine (resp. affine) stack over $\S$.
\end{prop}

Before establishing this, let us note the following
\begin{prop} \label{P:left_adjoint_algebras}
Let $\pi : \X \to \S$ be a flat morphism satisfying property \LL between qc.qs. algebraic derived stacks. Then the pullback functor $\pi^\ast : \CAlg(\QC(\S)^\heart) \to \CAlg(\QC(\X)^\heart)$ admits a left adjoint $\pi_\dagger$.
\end{prop}

\begin{proof}
The proof of \cite{DAG-XIV}*{Prop.~3.3.2} can be modified to this context: any quasi-coherent sheaf of algebras $\cA \in \CAlg(\QC(\X)^\heart)$ is a colimit (in fact a coequalizer) of free algebras $\op{Sym}_\X(E)$ for some $E \in \QC(\X)^{\heart}$, so it suffices to show that $\pi_\dagger(\op{Sym}_\X(E)$ exists. Note $\pi^\ast$ admits a left adjoint $\pi_+ : \QC(\X) \to \QC(\S)$ which preserves connective complexes because $\pi$ is flat -- this follows from the case where $\S$ is affine using \Cref{lem:base_change_left_adjoint} as in the first paragraph of the proof of \Cref{prop:property_LL_from_CP}. Then the pullback functor $\pi^\ast : \QC(\S)^\heart \to \QC(\X)^\heart$ admits a left adjoint $H_0(\pi_+(-))$, and one can check that $\pi_\dagger(\op{Sym}_\X(E)) \simeq \op{Sym}_\S(H_0(\pi_+(E)))$ using the universal property of a free algebra. 
\end{proof}

\begin{proof}[Proof of \Cref{P:weil_restrict}]
The claim is local for the smooth topology, so we may assume that $\S = \Spec(A)$ is affine. As we will see in the proof of \Cref{thm:derived_mapping_stacks} below, the functor $\pi_\ast(\sF/\X)$ is infinitesimally cohesive, nilcomplete, and admits a cotangent complex, so by \cite{HAG-II}*{Thm.~C.0.9} it suffices to check algebraicity of its restriction to the full subcategory $\op{Ring}_{\pi_0(A)} \subset \SCR_A$ of discrete objects, and restricting to $\op{Ring}_{\pi_0(A)}$ also suffices to check that that it is affine or quasi-affine. For any $R \in \SCR_A$, we have
\[
\pi_\ast(\sF/\X)(R) \simeq \Map_\X(\X \times_S \Spec(R),\sF).
\]
Because $\pi$ is flat, $\X \times_S \Spec(R)$ is classical for any discrete $R$, so the canonical map of functors $\pi_\ast(\sF^{\rm cl}/\cX) \to \pi_\ast(\cF/\cX)$ is an isomorphism after restricting to $\op{Ring}_{\pi_0(A)}$.

So we have reduced to the analogous claim when $\X,S,$ and $\sF$ are classical stacks, and we work in the category of classical stacks for the remainder of the proof. If $\sF \simeq \Spec_\X(\cA)$ for some quasi-coherent sheaf of algebras over $\X$, then one can check directly from the functor of points that $\pi_\ast(\sF/\X) \simeq \Spec_S(\pi_\dagger(\cA))$, where $\pi_\dagger : \CAlg(\QC(\X)^\heart) \to \CAlg(\QC(\S)^\heart)$ is the left adjoint to $\pi^\ast$ provided by \Cref{P:left_adjoint_algebras}.

If $\sF$ is quasi-affine, we have an open immersion $\sF \subset \Y = \Spec_\X(\cA)$, and the resulting map $\pi_\ast(\sF/\X) \to \pi_\ast(\Y/\X)$ is a monomorphism. Given a map $\Spec(R) \to \pi_\ast(\Y/\X)$, classifying a map $f : \X_R \to \Y$ over $\X$, a composition $\Spec(R') \to \Spec(R) \to \pi_\ast(\Y/\X)$ lifts (uniquely) to $\pi_\ast(\sF/\X)$ if and only if the canonical map $\X_{R'} \to \X_{R}$ factors through the open substack $f^{-1}(\sF) \subset \X_R$. Thus $\Spec(R) \times_{\pi_\ast(\Y/\X)} \pi_\ast(\sF/\X)$ is the subfunctor classified by the complement of the image of $\X_R \setminus f^{-1}(\sF)$ under $\X_R \to \Spec(R)$, which is an open subscheme by \Cref{P:universally_closed}.

\end{proof}

\subsubsection{Proof of \Cref{thm:derived_mapping_stacks}}

By \Cref{thm:derived_artin} it suffices to check that $\sM := \inner{\Map}_\S(\X, \Y)$ in the slice category $\Fun(\SCR,\widehat{\Sp})/\S$ is a derived stack which satisfies the Artin-Lurie criteria. The proof of the second claim in \cite{DAG-XIV}*{Proposition 3.3.5} applies verbatim in the context of simplicial commutative algebras to show that $\sM$ is a sheaf for the smooth topology. So to verify that $\sM \to \S$ is representable by derived algebraic stacks, we may restrict to the situation where $\S = S = \Spec(A)$ for a derived $G$-ring $A$ which admits a dualizing complex.

Realize $\X$ as a colimit of a simplicial diagram $X_\bullet$, where each $X_i = \Spec(R_i)$ is an affine derived scheme which is smooth over $\X$. Because $\Y$ satisfies smooth descent, $\inner{\Map}_S(\X,\Y) = \Tot \{\inner{\Map}_S(X_\bullet,\Y) \}$ in the $\infty$-category $\Fun(\SCR_A, \Sp)$. The subcategory of infinitesimally cohesive and nilcomplete functors is closed under small limits (\cite{DAG-XIV}*{Remark 2.1.11}), so it suffices to prove the claim for $X_i$ which is \cite{DAG-XIV}*{Proposition 3.3.6(2,3)}, whose proof applies verbatim for stacks over $\SCR$.

For a cosimplicial space $X^\bullet$ which is level-wise $k$-truncated, the map $\Tot \{X^\bullet\} \to \Tot_{k+1} \{X^\bullet\}$ is $(-1)$-truncated, which implies that $\Tot \{ - \}$ commutes with filtered colimits of cosimplicial spaces which are level-wise $k$-truncated for some fixed $k$. For any $A' \in \SCR_A$, we have
\[
\sM(A') = \Tot \{ \Y(A' \otimes_A R_\bullet)\}.
\]
If $A'$ is $k$-truncated, then $A' \otimes_A R_i$ is level-wise $k+d$-truncated, where $d$ is the Tor amplitude of the maps $A \to R_i$, and $\Y(A' \otimes_A R_i)$ is $n+d+k$-truncated by \cite{Lurie-Thesis}*{Cor.~5.3.8}. It follows that if $\Y(-)$ commutes with filtered colimits of $k$-truncated objects for any $k$, then so does $\sM(A')$. These observations also show that $\sM(A')$ is $n+d$-truncated when $A'$ is discrete.

We have verified the Artin-Lurie criteria (1),(2),(3b), so it remains to show (3a) and (4), but we have verified these above: \Cref{prop:mapping_cotangent_complex} and \Cref{prop:property_LL_from_CP} imply that $\sM$ admits a $(-1-n)$-connective almost perfect cotangent complex. For integrability, consider a discrete complete local noetherian ring $R$ over $A$, and let $\oh{\X}_R$ be the formal completion of $\X_R$ along the closed subset defined by the maximal ideal of $R$.   , \Cref{prop:integrability} applied to the completion of $\X \times_S \Spec(R)$ along the closed substack defined by the maximal ideal $\fm \subset R$ implies the integrability of $\sM$.

\subsection{Algebraicity of the stack of coherent sheaves}

In \cite{Pridham_moduli}, Pridham uses a modified version of the Artin-Lurie criteria above \cite{Pridham_artin} to give a simplified algebraicity criterion for substacks of the stack of quasi-coherent complexes on a derived algebraic stack. We apply this criterion below.

Given an $\infty$-category $\cC$, we let $\cC^{\cong}$ denote the largest $\infty$-subcategory whose homomorphisms are invertible, i.e., the largest Kan subcomplex of the quasi-category $\cC$. This is referred to as the ``core'' and denoted $\cW(-)$ in \cite{Pridham_moduli}.

\begin{defn}\label{def:stack_sheaves}
For a flat and almost finitely presented morphism of algebraic derived stacks $\X \to \Spec(R)$, let $\underline{\Coh}_{\X/R} \colon \SCR_R \to \Sp$ be the functor which maps $A \in \SCR_R$ to the full $\infty$-subgroupoid of $(\APerf(\X_A)^{cn})^{\cong}$ whose objects are flat over $A$.
\end{defn}

\begin{thm}\label{thm:moduli}
Let $R$ be a derived $G$-ring which admits a dualizing complex, and let $\pi \colon \X \to \Spec(R)$ be a flat and formally proper morphism of derived algebraic stacks. Then $\underline{\Coh}_{\X/R}$ is a derived algebraic stack with affine diagonal.
\end{thm}

We first note the following lemma, which identifies $\underline{\Coh}_{\X/R}(A)$ with the full $\infty$-subgroupoid of $(\QC(\X_A)^{cn})^{\cong}$ consisting of complexes whose restriction to $\X_{\pi_0(A)}$ is flat and finitely presented (in the classical sense).

\begin{lem} \label{lem:aperf_thickening}
Let $\pi \colon \X \to \Spec(A)$ be a flat morphism of algebraic spectral stacks which is finitely presented on underlying classical stacks, and let $\pi_0(A) \to A'$ be a surjection of (discrete) rings with nilpotent kernel. Then $E \in \QC(\X)^{cn}$ is flat over $A$ and almost perfect if and only if $E \otimes_{A} A' \in \QC(\X\times_{\Spec(A)} \Spec(A'))^{cn}$ is flat over $A'$ (hence discrete) and classically finitely presented.
\end{lem}
\begin{proof}
The claim is local, so we may assume $\X = \Spec(B)$ for some $E_\infty$-algebra $B$ over $A$, and let $B' := B \otimes_A A'$. By \cite{HigherAlgebra}*{Thm.~7.2.2.15}, $E$ is flat as an $A$-module if and only if $E \otimes_A N$ is discrete whenever $N \in (A\mod)^\heart$. The hypotheses imply $(A\mod)^\heart$ is generated under extensions by objects of the form $i_\ast(M)$ where $i : \Spec(A') \to \Spec(A)$, and the projection formula implies $M \otimes_A i_\ast(-) \simeq i_\ast(i^\ast(M) \otimes_{A'} (-))$, so $M$ is flat if and only if $i^\ast(M) = M \otimes_B B'$ is flat.

Now assume $E \in B\mod$ is flat over $A$, the map of (discrete) rings $A' \to B'$ is finitely presented, and $E \otimes_B B'$ is a flat and finitely presented discrete $B'$-module. By classical noetherian approximation \cite{stacks-project}*{\href{https://stacks.math.columbia.edu/tag/02JO}{Tag 02JO}}, both $B'$ and $E \otimes_B B'$ arise via base change from the same data over a finitely generated subring $A'' \subset A'$, so $E \otimes_B B'$ is almost perfect.

We have already shown in the proof of \Cref{prop:closed-descent-QCcn} that the functor $\APerf(-)$ is nilcomplete (without noetherian hypotheses), so it suffices to assume that $\pi_i(A)=0$ and hence $\pi_i(B)=0$ for $i\gg 0$. Then $B$ can be obtained from $B'$ from a finite sequence of square-zero extensions, so it suffices to assume $B \to B'$ is a square-zero extension. In this case $\Spec(B)$ is obtained as a pushout along two maps $\Spec(B' \oplus M) \to \Spec(B')$ (see \cite{HigherAlgebra}*{Rem.~7.4.1.7}). It follows from \cite{DAG-IX}*{Prop.~7.7} that $E$ is almost perfect if and only if $E \otimes_B B'$ is almost perfect, which completes the proof.

\end{proof}

\begin{proof}

Consider the functor $\sF \colon \SCR_R \to \widehat{\Sp}$ which assigns $A \mapsto \QC(\X_A)^{\cong}$, and the functor $\cM : \op{Ring}_{\pi_0(R)} \to \op{Cat}_\infty$ which assigns
\[
A \mapsto \left\{ \begin{array}{c} \text{finitely presented and } A\text{-flat objects} \\ \text{in }\QC(\X \times_{\Spec(\pi_0(R))} \Spec(A))^\heart \end{array} \right\}.
\]
We note that \Cref{lem:aperf_thickening} implies that $\cM$ is open in the functor $\sF$ in the sense of \cite{Pridham_moduli}*{Def.~3.8}. Also by \Cref{lem:aperf_thickening} the functor $\underline{\Coh}_{\X/R}$ defined above is the full subfunctor of $\sF$ consisting of objects in $\sF(A)$ whose restriction to $\sF(\pi_0(A))$ are weakly equivalent in $\QC(\X_{\pi_0(A)})$ to an object of $\cM(\pi_0(A))$. \cite{Pridham_moduli}*{Thm.~4.12} gives criteria for the resulting functor to be an algebraic derived $1$-stack, which we now verify:

Condition $(0)$ is the fact one can check if a complex in $\QC(\X_A)^{cn}$ is flat and almost perfect \'etale locally over $A$. Condition (1), that for any $A \in \op{Ring}_{\pi_0(R)}$ and $E \in \cM(A)$ the functor $H_i \RHom_{\X_A}(E,E\otimes_A (-))$ commutes with filtered colimits in $(A\mod)^\heart$, is immediate from the fact that $E$ is $A$-flat and almost perfect. Condition (2), that for all finitely generated $A \in \op{Ring}_{\pi_0(R)}$ and all $E \in \cM(A)$, $H_i \RHom_{\X_A}(E,E)$ are finitely generated $A$-modules for all $i$, follows from the fact that $\RHom_{\X_A}^{\otimes_{\QC(\X_A)}}(E,E) \in \DDAPerf(\X_A)$ by \Cref{lem:bounded_inner_hom} and from property \CP[A], which holds by \Cref{T:cp_fully_faithful}.

Condition (3) states that the functor of components of $\cM$ preserves filtered colimits in $\op{Ring}_{\pi_0(R)}$. This holds because by classical noetherian approximation if $A = \colim_\alpha A_\alpha$ is a filtered colimit, then any flat and finitely presented object in $\QC(\X_A)^\heart$ is the pullback of a finitely presented object in some $\QC(\X_{A_\alpha})^\heart$, and by \cite{stacks-project}*{\href{https://stacks.math.columbia.edu/tag/02JO}{Tag 02JO}} one can increase $\alpha$ so that the object in $\QC(\X_{A_\alpha})^\heart$ is flat and finitely presented.

Finally, because $\cM(A)$ is a $1$-category for any $A \in \op{Ring}_{\pi_0(R)}$, i.e., flat and almost perfect objects in $\QC(\X_A)$ have no negative Exts, condition (4) amounts to the claim that for any complete local discrete noetherian $\pi_0(R)$-algebra $A$ with maximal ideal $\mathfrak{m}$, $\cM(A)$ is the homotopy inverse limit of the categories $\cM(A/\mathfrak{m}^k)$. This follows from the definition of a formally proper morphism (\Cref{def:cohomologically_proper}) and the fact that if $E \in \QC(\X_A)^\heart$ is such that $E \otimes_A (A/\mathfrak{m}^k)$ is $(A/\mathfrak{m}^k)$-flat for all $k\geq 1$, then $E$ is flat over $A$.

This completes the proof that $\underline{\Coh}_{\X/R}$ is an algebraic derived $1$-stack locally almost of finite presentation over $R$. To show that $\underline{\Coh}_{\X/R}$ has affine diagonal, it suffices to show this for the underlying classical stack on $\op{Ring}_{\pi_0(R)}$. This stack agrees with the moduli functor $\cM_\sA$ of \cite{existence-moduli}*{Def.~7.8} associated to the locally noetherian $\pi_0(R)$-linear abelian category $\sA = \QC(\X_{\pi_0(R)})^\heart$. By definition $\cM_\sA(A)$ is the groupoid of $A$-flat and finitely presented $A$-module objects in $\sA$. By \cite{existence-moduli}*{Lem.~7.19}, if $\cM_\sA$ is algebraic and locally of finite presentation over $\pi_0(R)$, then it must have affine diagonal.

%\medskip
%\noindent \textbf{Proof that $\underline{\Coh}_{\X/R}$ has affine diagonal:}
%\medskip
%
%It suffices by a standard argument (see the proof of \cite[\href{https://stacks.math.columbia.edu/tag/08K9}{Tag 08K9}]{stacks-project}) to show that for any finitely generated discrete $H_0(R)$-algebra $A$, and any flat $E,F \in \APerf(\X_A)^{cn}$, the functor $\op{Ring}_A \to \op{Set}$ mapping
%\[
%B \mapsto \Hom_{\X_B}(E\otimes_A B, F \otimes_A B) \simeq \Hom_{\X_A}(E, F\otimes_A B)
%\]
%is representable by a finite type affine $A$-scheme.
\end{proof}

\appendix

%%%%%%%%%%%%%%%%%%%%%%%%%%%%%%%%%%%%%%%%%%%%%%%%%%%%%%%%%%%%%%%%%%
%%%%%%%%%%%%%%%%%%%%%%%%%%%%%%%%%%%%%%%%%%%%%%%%%%%%%%%%%%%%%%%%%%
%%%%%%%%%%%%%%%%%%%%%%%%%%%%%%%%%%%%%%%%%%%%%%%%%%%%%%%%%%%%%%%%%%
%%%%%%%%%%%%%%%%%%%%%%%%%%%%%%%%%%%%%%%%%%%%%%%%%%%%%%%%%%%%%%%%%%

\addtocontents{toc}{\protect\setcounter{tocdepth}{1}}

\section{Quasi-coherent complexes on prestacks}
\label{section:qc}

Suppose that $R_\bullet$ is a simplicial commutative ring. The normalized chain complex $N(R_\bullet)$ obtains the structure of a differential graded algebra via the Eilenberg-Zilber product, and the category of $R_\bullet$-modules, $\C = R_\bullet\mod$, is the stable symmetric monoidal $\infty$-category of left $N(R_\bullet)$-modules in chain complexes (see \cite{HAG-II}*{Sect.~2.2.1}). Furthermore
\begin{enumerate}
    \item $\C_{\geq 0}$ is the unstable $\infty$-category of $N(R_\bullet)$-modules in homologically nonnegative degrees -- i.e., Dold-Kan provides an equivalence of $\C_{\geq 0}$ with the $\infty$-category of simplicial $R$-modules;
    \item $\C$ is equivalent the the $\infty$-category of modules over the $E_\infty$-algebra $\op{End}_{\C}(R_\bullet)$ associated to $R_\bullet$;
    \item If $E,F \in R_\bullet\mod$, then in our notation: $\Map_R(E,F)$ is the simplicial set of maps from a (cofibrant replacement of) $E$ to a (fibrant replacement) of $F$, $\RHom_R(E,F)$ is the $N(R_\bullet)$-module with $(\RHom_R(E,F))_i$ the degree $i$ morphisms from (a replacement of) $N(E)$ to (a replacement of) $N(F)$, and $\Hom_R(E,F) = \Ext^0_R(E,F)$ denotes maps in the derived category of $N(R_\bullet)$-modules.
\end{enumerate}

For any functor $\cF : \SCR \to \CAlg(\widehat{\op{Cat}}_\infty)$, where the latter denotes the $\infty$-category of (not necessarily small) symmetric monoidal $\infty$-categories, and any prestack $\X \in \Fun(\SCR,\widehat{\Sp})$, one can define $\cF(\X)$ as the right Kan extension of $\cF$ along the Yoneda embedding $\SCR \hookrightarrow \Fun(\SCR,\widehat{\Sp})^{op}$ (see \cite{DAG-VIII}*{Sect.~2.7} for a detailed discussion). This means that
\[ \cF(\X) = \ilim_{\eta \in \X(R)} \cF(R), \]
i.e., an object $E \in \cF(\X)$ is the coherent assignment to each pair of an $R \in \SCR$ and an $R$-point $\eta \colon \Spec(R) \to \X$ of an object $E_\eta \in \cF(R)$. Because the forgetful functor $\CAlg(\widehat{\op{Cat}}_\infty) \to \widehat{\op{Cat}}_\infty$ preserves limits, the $\infty$-category underlying $\cF(\X)$ is the limit of the $\infty$-categories underlying $\cF(R)$ above.

Applying this to the functor $\cF \colon R_\bullet \mapsto R_\bullet\mod$ defines the symmetric monoidal stable $\infty$-category $\QC(\X)$ of quasi-coherent complexes on $\X$ \cite{DAG-VIII}*{Def.~2.7.8}, i.e., a quasi-coherent complex $F \in \QC(\X)$ is the coherent assignment to each pair $(R, \eta \in \X(R))$ of an $R$-module $F_\eta \in R\mod$. We also define several other categories associated to functors $\SCR \to \CAlg(\widehat{\op{Cat}}_\infty)$ or $\SCR \to \widehat{\op{Cat}}_\infty$:
\begin{itemize}
\item $\Coh[n](\X)$: associated to the $\infty$-category of compact objects in the $\infty$-category of connective, $n$-truncated $R$-modules.
\item $\Coh(\X)$: special notation for $\Coh[0](\X)$, associated to the compact objects of the abelian category $(R\mod)^{\heart} = (\pi_0(R)\mod)^\heart$, i.e. the ordinary category of finitely presented $\pi_0(R)$-modules.
\item $\APerf(\X) \subset \QC(\X)$: associated to the symmetric monoidal stable $\infty$-category $\APerf(R) \subset R\mod$ of \emph{almost perfect} complexes. $\APerf(R)$ consists precisely of those $R$-modules $M$ such that $\tau_{<\ell} M \in (R\mod)_{<\ell}$ is compact for each $\ell \in \ZZ$.
\item $\Perf(\X) \subset \QC(\X)$: associated to the symmetric monoidal stable $\infty$-subcategory $\Perf(R) \subset R\mod$ of \emph{perfect complexes}. $\Perf(R)$ is the smallest subcategory of $R\mod$ closed under cones, shifts, and retracts and containing $R$.  By \cite{DAG-VIII}*{Prop.~2.7.28}, $\Perf(\X) \subset \QC(\X)$ consists precisely of the \emph{dualizable objects} with respect to the symmetric monoidal structure on $\QC(\X)$.
\end{itemize}

The functor $\SCR \to \CAlg(\widehat{\op{Cat}}_\infty)$ which maps $R \mapsto R\mod$ factors through the canonical functor $\SCR \to \CAlg^{cn}$ which maps $R \mapsto \op{End}_{R\mod}(R)$. It follows that if $\X : \SCR \to \widehat{\Sp}$ is a derived prestack, and $\X^{\rm sp}$ is the spectral algebraic stack obtained by left Kan extension along the functor $\SCR \to \CAlg^{cn}$, then $\QC(\X^{\rm sp}) \simeq \QC(\X)$, where the former denotes the construction of $\QC(-)$ for spectral prestacks. The same is true for all of the categories discussed above.

\begin{defn} \label{defn:locally_noetherian_categories} We say that $\X \in \Fun(\SCR, \widehat{\Sp})$ is a \emph{locally noetherian prestack} if it is left Kan extended, up to sheafification, from a functor on noetherian \cite{Lurie-Thesis}*{Def.~2.5.9} simplicial commutative algebras $\X^N \in \Fun(\SCR^{noeth}, \Sp)$. In this case $\QC(\X) \isom \QC(\X^N)$, where the latter denotes the right Kan extension of $R\mod$ along the Yoneda embedding $\SCR^{noeth} \hookrightarrow \Fun(\SCR^{noeth},\widehat{\Sp})^{op}$, and the same holds for all of the other categories considered above.
\end{defn}

In particular, for a noetherian ring $R \in \SCR^{noeth}$, $E \in R\mod$ is almost perfect if and only if $H_i(E) \in \QC(R\mod)^{\heart}$ is coherent for all $i$ and $H_i(E) = 0$ for $i \ll 0$ \cite{HigherAlgebra}*{Prop.~7.2.4.17}. So we could alternatively define $\APerf(\X) = \APerf(\X^N)$ as the right Kan extension of this functor $\SCR^{noeth} \to \CAlg(\op{Cat}_\infty)$ for a locally noetherian $\X$. This motivates the definition of the full stable subcategories of $\QC(\X^N)$:
\begin{itemize}
 \item $\DDAPerf(\X) \subset \QC(\X^N)$: consisting of objects such that for any pair $(R,\eta \in \X^N(R))$, $H_i(F_\eta) \in \QC(R\mod)^{\heart}$ is coherent for all $i$ and $H_i(F) = 0$ for $i \gg 0$.
 \item $\DCoh(\X) \subset \QC(\X^N)$: consisting of objects such that for any pair $(R,\eta \in \X^N(R))$, $H_i(F_\eta) \in \QC(R\mod)^{\heart}$ is coherent for all $i$ and $H_i(F) = 0$ for $i \gg 0$ and $i \ll 0$.
    \end{itemize}
The notation $\DDAPerf$ is motivated by the fact that when $\X$ admits a Grothendieck dualizing complex, then Grothendieck duality provides an anti-equivalence $\DDAPerf(\X) \isom \APerf(\X)^{op}$. $\DCoh(\X)$ is a version of the ``bounded derived category of coherent sheaves'' for derived stacks. Note that in contrast to $\APerf$, $\DCoh$ and $\DDAPerf$ do not inherit pullback functors and symmetric monoidal structures from $\QC$.

\subsubsection*{\texorpdfstring{$t$}{t}-structures}

When $\QC(\X)$ is presentable, which is always the case for algebraic derived stacks and their formal completions, because presentable categories are closed under small limits, then $\QC(\X)$ carries a $t$-structure whose subcategory of connective objects $\QC(\X)_{\geq 0} = \QC(\X)^{cn}$ consists of those $F$ such that $F_{\eta} \in (R\mod)_{\geq 0}$ is connective for all pairs $(R, \eta \in \X(R))$. Although this definition is formally convenient, using the $t$-structure is usually only practical when $\X$ is an algebraic stack. If  $\pi \colon U = \Spec(R)\to \X$ is an fppf atlas then $\pi^*$ is $t$-exact -- in particular, $F \in \QC(\X)$ is connective (resp., co-connective) if and only if $\pi^* F$ is so. Note also that when $\X$ is a locally noetherian prestack, $\APerf(\X)$ and $\DCoh(\X)$ carry unique $t$-structures for which the inclusion into $\QC(\X)$ is $t$-exact.

One can also consider the Kan extension of $R \mapsto (R\mod)^\heart$, the usual abelian category of quasi-coherent sheaves, which we denote $\QC^\heart(\X)$. In contrast, let $\QC(\X)^{\heart}$ denote the symmetric monoidal abelian category given by the heart of the $t$-structure. Since pullback is right $t$-exact, there is a truncation functor
\[ \left(\ilim_{\eta \in \X(R)} R\mod\right)^\heart \longrightarrow \ilim_{\eta \in \X(R)}\left((R\mod)^\heart\right) \qquad F_{\eta} \mapsto H_0(F_{\eta}) \]
and one can check that this is a fully faithful embedding. When $\X$ is algebraic, it in fact gives an equivalence $\QC(\X)^\heart \to \QC^\heart(\X)$. Note that because $(R\mod)^\heart \simeq (\pi_0(R)\mod)^\heart$, one can further identify $\QC^\heart(\X)$ with the category of quasi-coherent sheaves on the classical stack obtained by restricting $\X$ to the category of discrete simplicial commutative rings $\op{Ring} \subset \SCR$. Likewise in the locally noetherian case, $\DCoh(\X)^{\heart} \isom \APerf(\X)^{\heart}$ coincides with the ordinary abelian category of coherent sheaves on underlying classical prestack.

\subsection{Quasi-coherent pushforwards and base change} We recall the definition of the pushforward of quasi-coherent sheaves in our context.

\begin{defn}Suppose that $f \colon \X \to \Y$ is an arbitrary map of prestacks such that $\QC(\X)$ and $\QC(\Y)$ are presentable. Then $f^* \colon \QC(\Y) \to \QC(\X)$ is a colimit-preserving functor between presentable $\infty$-categories and thus admits a right adjoint \cite{HigherTopos}*{Cor.~5.5.2.9}, which we denote by $f_*$.
\end{defn}

\begin{rem} It is a priori non-obvious that $\QC(\X)$ has anything to do with sheaves of modules in some $\infty$-topos, or that the pushforward defined above has anything to do with a pushforward of sheaves.  Nevertheless, this is true if $\X$ (resp., $\X \to \Y$) is nice enough.  We do not dwell on this point, but the interested reader may consult e.g., \cite{DAG-VIII}*{Prop.~2.7.18} for the case of $\X$ a Deligne-Mumford stack and the \'etale $\infty$-topos.
\end{rem}

One might expect that $f_\ast$ as defined above is well-behaved for stacks which arise ``in nature'', but in general this is quite false: Take
\[ f \colon \X = B\ZZ/p \to \Spec(\bF_p) \]
Then, the functor $f_*$ will not preserve filtered colimits; will not be compatible with arbitrary base-change; and will generally not be pleasant.  Nevertheless, one has the following positive result:

\begin{lem}\label{lem:push-bdd-above} Suppose that $f\colon \X \to \Y$ is a morphism of spectral (or derived) prestacks which is a relative qc.qs. algebraic stack. Then,
\begin{itemize}
      \item Let $i : \Y' \to \Y$ be a relatively algebraic morphism of finite Tor amplitude (e.g. a flat morphism), and let $f' \colon \X' \to \Y'$ and $i' : \X' \to \X$ be the base change of $f$ and $i$ respectively. For $F \in \QC(\X)$, the canonical base change map
      \[   i^*(f_*(F)) \to (f')_*( (i')^*(F)) \]
      is an isomorphism if either $F$ is homologically bounded above (i.e. $F \in \QC(\X)_{<\infty}$) or if $i$ is finite and $F$ is arbitrary.
      \item $f_*$ preserves filtered colimits (equivalently, infinite sums) in $\QC(\X)_{<n}$ for each $n$ (i.e., for uniformly bounded above colimits).
    \end{itemize}
\end{lem}
\begin{proof}[Sketch]
  It is enough, by the definition of $\QC$ as extended from affines and smooth hyperdescent for $\QC$, to verify this in case where $\Y = \Spec(R)$ and $\Y' = \Spec(R')$ are affine. Let $p_\bullet \colon U_\bullet = \Spec(A_\bullet) \to \X$ be a presentation of $\X$ as the geometric realization, in smooth sheaves, of a simplicial diagram of affine schemes along smooth morphisms -- such a diagram exists because $\X$ is $\infty$-quasi-compact.
 
 By smooth hyper-descent for $\QC$, we thus have an equivalence
  \[ (p_\bullet)^*\colon \QC(\X) \isom \Tot\{ \QC(U_\bullet) \} \simeq \Tot\{A_\bullet\mod\}. \]
 Under this equivalence $f^*$ identifies with the cosimplicial diagram of pullbacks $(p_\bullet \circ f)^*$ mapping $M \mapsto A_\bullet \otimes_R M$.  Thus, we can compute the right adjoint $f_*$ in terms of this Cech diagram -- namely
 \[ f_*(\sF) = \Tot\left\{ (p_\bullet \circ f)_* p_\bullet^* \sF \right\}.\]
So if $p_\bullet^\ast(\sF) \simeq M_\bullet \in \Tot\{A_\bullet\mod\}$, then $f_\ast(\sF) \simeq \Tot\{M_\bullet\} \in R\mod$, where in the latter $M_\bullet$ is regarded as a cosimplicial $R$-module.

If $i$ is finite and has finite Tor amplitude, we conclude that $R'$ is perfect as $R$-module (because it is almost perfect of finite Tor amplitude).  Thus, $i^*$ commutes with arbitrary homotopy limits and in particular with the formation of the totalization, so the base change statement follows from the case where $\X$ is affine.
  
More generally, note that $\sF \in \QC(\X)_{<0}$ if and only if $M_\bullet \in (R\mod)_{<0}$ for all $\bullet$.  For such objects, the resulting spectral sequence of a totalization is a (convergent) third quadrant spectral sequence.  The formation of this spectral sequence is evidently compatible with filtered colimits and flat base change; a slight elaboration gives the case of finite Tor dimension base change.
\end{proof}

In order to obtain stronger base-change results, one must consider a restricted class of morphisms:

\begin{defn}\label{property:CD} We say that a morphism $f \colon \X \to \Y$ of spectral (or derived) prestacks is of 
  \emph{cohomological dimension at most $d$}  if for any $F \in \QC(\X)^\heart$ we have $f_* F \in \QC(\X)_{\geq -d}$.  (Note that if $\X$ is algebraic this depends only on the induced morphism of underlying classical stacks.)

  We say that a morphism $f \colon \X \to \Y$ of spectral (or derived) prestacks is \emph{universally of finite cohomological dimension} (or, \emph{satisfies \CD} for short) if there is some $d$ for which this condition is satisfied for the base-change of $f$ along any morphism from an affine $\Spec(A) \to \Y$.
\end{defn}

\begin{prop}\label{prop:push-CD}
Suppose that $f \colon \X \to \Y$ is a morphism of prestacks which is relatively representable by qc.qs. algebraic spectral (or derived) stacks satisfying \CD. Then:
  \begin{enumerate}
      \item $f_*\colon \QC(\X) \to \QC(\Y)$ preserves filtered colimits;
      \item $f_*$ and $f^*$ satisfy the projection formula, i.e., the natural morphism $f_*(F) \otimes G \to f_*(F \otimes f^* G)$ is an equivalence for all $F, G$;
      \item the formation of $f_*$ is compatible with arbitrary base-change.
    \end{enumerate}
\end{prop}

First we note an alternative definition of cohomological dimension which is often equivalent to the one above.
\begin{lem}\label{lem:push-connective}
Let $f : \X \to \Y$ be a morphism of qc.qs. spectral (or derived) algebraic stacks. Then $f$ is of cohomological dimension at most $d$ if and only if $f_\ast (\QC(\X)_{\geq 0}) \subset \QC(\Y)_{\geq -d}$.
\end{lem}
\begin{proof}
This follows from part (2) of \Cref{lem:random-t-stuff} and the fact that $\QC(\X)$ and $\QC(\Y)$ are $t$-complete: for any $F \in \QC(\X)_{\geq 0}$, $F = \varprojlim \tau_{\leq n} F$ and $f_\ast(F) \simeq \varprojlim f_\ast(\tau_{\leq n}(F))$, and finite cohomological dimension implies that $\tau_{\leq k}(f_\ast(\tau_{\leq n}(F)))$ is eventually constant in $n$, for any $k$.
\end{proof}

\begin{proof} 

First let us prove the proposition when $\Y = \Spec(A)$ is affine. The claims that $f_\ast$ preserves filtered colimits and is compatible with flat base change on $\Spec(A)$ can be checked at the level of homology groups, and the fact that $f_\ast (QC(\X)_{\geq 0}) \subset A\mod_{\geq -d}$ by \Cref{lem:push-connective} implies that 
\[ H_i \circ f_* = H_i \circ f_* \circ \tau_{\leq i+d}.\]
This reduces us to showing these claims for $f_\ast$ applied to $\QC(\X)_{<n}$ for any fixed $n$, which is \Cref{lem:push-bdd-above}.

Next let us verify (2): Pick a hypercover $U_\bullet = \Spec(B_\bullet) \to \X$, let $M_\bullet \in \Tot\{B_\bullet\mod\} \simeq \QC(\X)$, and let $N \in \QC(\Y) = A\mod$. We must verify that the natural map
  \[    \Tot\{ M_\bullet \} \otimes_A N \longrightarrow  \Tot\left\{ M_\bullet \otimes_{B_\bullet} (B_\bullet \otimes_A N) \right\} \]
  is a quasi-isomorphism.

  Let $\C \subset A\mod$ denote the full subcategory consisting of those $N \in A\mod$ for which the preceding map is a quasi-isomorphism for all $M_\bullet$. Note that $\C$ is closed under cones, shifts, and retracts since both $f_*$, $f^*$, and $\otimes$ preserve these operations up to quasi-isomorphism.  Next, note that $\C$ is closed under filtered colimits, because all three operations preserve filtered colimits by (1). Finally, observe that $A \in \C$.  But the smallest subcategory $A\mod$ containing $A$ and closed under cones, shifts, and filtered colimits is all of $A\mod$.

For (3), it is enough to consider the case of an affine base-change (See \cite{DrinfeldGaitsgory}*{Proposition 1.3.6} for an argument in the setting of dg-algebras which applies verbatim to $E_\infty$-algebras). Suppose that $S' = \Spec(A') \to \Spec(A)$ is arbitrary, and let $f' \colon \X' = \X \times_{\Spec(A)} S' \to S'$ be the base-change of $S$.  We must show that the natural map
\[ A' \otimes_A \RGamma(\X, F) \longrightarrow \RGamma(\X', \res{F}{\X'}) = \RGamma(\X, A' \otimes_A F) \] is an equivalence of $A'$-modules. Note, however, that regarding this as a morphism of $A$-modules, this is precisely the equivalence of the projection formula.

\smallskip
{\noindent \it Extending to arbitrary bases:}
\smallskip

We have shown that after base change to an affine $\Spec(A) \times_\Y \X$, the base change formula holds for $f$, and this implies the base change formula holds for $f$ itself and arbitrary prestacks $\Y' \to \Y$ (again by \cite{DrinfeldGaitsgory}*{Proposition 1.3.6}). Once we have base change, (1) follows, because colimits in $QC(\Y)$ can be identified by their restriction to affine derived schemes over $\Y$. Likewise the base change formula for $f$ can be checked after base change to an arbitrary affine.

\end{proof}

\begin{ex}
We have occasion to apply \Cref{prop:push-CD} to stacks which are not algebraic when we prove a strong version of the Grothendieck existence theorem, \Cref{thm:coh_projective_is_proper}. There we consider a fiber square of stacks which is the formal completion of a fiber square of algebraic stacks along a cocompact closed subset of the base.
\end{ex}
 
In characteristic zero, it turns out that \CD is very often satisfied:
\begin{prop} \label{prop:CD_characteristic_0}
  Suppose that $S$ is a noetherian characteristic zero derived scheme, and that $\X$ is a finite type algebraic derived $S$-stack such that the automorphisms of its geometric points are affine. Then, $\X$ satisfies \CD over $S$.
\end{prop}
\begin{proof} This is proven in slightly more generality in \cite{DrinfeldGaitsgory}*{Theorem 1.4.2}: One can show that $\X$ has a finite stratification by global quotient stacks, and a straightforward argument shows that having finite cohomological dimension is stable under open-closed decompositions.  It thus suffices to prove the result for global quotient stacks.  This follows by noting that quasi-compact and quasi-separated algebraic spaces have finite cohomological dimension, and that reductive groups (e.g., $GL_n$) are linearly reductive in characteristic zero.
\end{proof}

In addition, for many morphisms between algebraic derived stacks, satisfying \CD is equivalent to having finite cohomological dimension.

\begin{prop}\label{prop:CDd} If $f\colon \X \to \Y$ is a morphism of qc.qs. algebraic derived (or spectral) stacks and $\Y$ has affine diagonal, then the following are equivalent:
  \begin{enumerate}
    \item $f$ is universally of cohomological dimension at most $d$;
    \item for any flat morphism $S = \Spec(R) \to \Y$, the base-change $f_S \colon \X_S \to S$ is of cohomological dimension at most $d$;
    \item $f$ is of cohomological dimension at most $d$;
    \item $f_*$ takes $\QC(\X)_{>0}$ into $\QC(\Y)_{>-d}$.
  \end{enumerate}
\end{prop}

\begin{proof}
\Cref{lem:push-connective} implies that (3) and (4) are equivalent. (1) implies (2) by definition, and (2) implies (3) by flat base change and faithfully flat descent. Finally, the property of being of cohomological dimension at most $d$ is stable under affine base change, because affine maps are conservative and $t$-exact. Thus because $\Y$ is has affine diagonal, (3) implies that $f$ is of cohomological dimension at most $d$ after base change to any affine test scheme, which is (1).
\end{proof}

Finally we note:

\begin{prop}
Given a morphism $f : \X \to \Y$ of derived (or spectral) stacks, property \CD is fppf-local over $\Y$.
\end{prop}
\begin{proof}
\CD is clearly stable under base change. Conversely, if $\Y' \to \Y$ is an fppf algebraic morphism such that $\X \times_\Y \Y' \to \Y$ satisfies \CD, then for any morphism $\Spec(R) \to \Y$, one can find an fppf morphism $\Spec(R') \to \Spec(R)$ such that the composition $\Spec(R') \to \Y$ lifts to $\Y'$. It follows that $\X_{R'} \to \Spec(R')$ has cohomological dimension $d$ and thus so does $\X_R \to \Spec(R)$ by fppf descent and flat base change.
\end{proof}

\subsection{Enough coherent complexes} We will often use the following compact generation statement:

\begin{thm}\label{thm:coh-gen} If $\X$ is a noetherian algebraic derived (or spectral) stack, then:
  \begin{enumerate}
      \item (\cite{LMB00}) $\QC(\X)^{\heart}$ is compactly-generated by $\DCoh(\X)^{\heart}$.  
      \item (\cite{DrinfeldGaitsgory}) The subcategory $\QC(\X)_{<0} \subset \QC(\X)$ is compactly-generated, with compact objects precisely $\DCoh(\X)_{<0}$.
      \item For each $d \geq 0$, the subcategory $\QC(\X)_{\geq 0, <d} \subset \QC(\X)$ is compactly-generated, with compact objects precisely $\DCoh(\X)_{\geq 0, <d}$.
    \end{enumerate}
\end{thm}
\begin{proof} See \cite{LMB00}*{Prop.~15.4} for (1). Since the $t$-structure on $\QC(\X)$ is right $t$-complete and compatible with filtered colimits -- these properties being flat-local and true for affines -- both (2) and (3) reduce to (1) along with the assertion that, for any $K \in \DCoh(\X)$, the functor $\Map(K,\bullet)$ preserves uniformly left $t$-bounded filtered colimits. Because any such $K$ is bounded, we can reduce to showing that $\forall K \in \DCoh(\X)^\heart$ and all $i \geq 0$, $\Ext^i(K,\bullet)$ preserves filtered colimits in $\QC(\X)^\heart$. This can be verified in the affine case by approximating $K$ by perfect complexes, and the general case reduces to the affine case via the argument in the proof of \Cref{lem:push-bdd-above}, where we only need to use the first $i$ levels of the hypercover of $\X$.

\end{proof}

\bibliographystyle{plain}
\bibliography{references}

\begin{bibdiv}
\begin{biblist}

\bib{Alper}{article}{
      author={Alper, Jarod},
       title={Good moduli spaces for {A}rtin stacks},
        date={2013},
        ISSN={0373-0956},
     journal={Ann. Inst. Fourier (Grenoble)},
      volume={63},
      number={6},
       pages={2349\ndash 2402},
         url={http://aif.cedram.org/item?id=AIF_2013__63_6_2349_0},
      review={\MR{3237451}},
}

\bib{Ao06b}{article}{
      author={Aoki, Masao},
       title={Erratum: ``{H}om stacks'' [{M}anuscripta {M}ath. {\bf 119}
  (2006), no. 1, 37--56; mr2194377]},
        date={2006},
        ISSN={0025-2611},
     journal={Manuscripta Math.},
      volume={121},
      number={1},
       pages={135},
         url={http://dx.doi.org/10.1007/s00229-006-0029-3},
      review={\MR{2258535 (2007d:14002)}},
}

\bib{Ao06a}{article}{
      author={Aoki, Masao},
       title={Hom stacks},
        date={2006},
        ISSN={0025-2611},
     journal={Manuscripta Math.},
      volume={119},
      number={1},
       pages={37\ndash 56},
         url={http://dx.doi.org/10.1007/s00229-005-0602-1},
      review={\MR{2194377 (2007a:14003)}},
}

\bib{AlexeevBrion}{article}{
      author={{Alexeev}, Valery},
      author={{Brion}, Michel},
       title={Moduli of affine schemes with reductive group action.},
        date={2005},
     journal={J. Algebraic Geom.},
}

\bib{existence-moduli}{article}{
      author={{Alper}, Jarod},
      author={{Halpern-Leistner}, Daniel},
      author={{Heinloth}, Jochen},
       title={{Existence of moduli spaces for algebraic stacks}},
        date={2018Dec},
     journal={arXiv e-prints},
       pages={arXiv:1812.01128},
      eprint={1812.01128},
}

\bib{alper2015luna}{article}{
      author={Alper, Jarod},
      author={Hall, Jack},
      author={Rydh, David},
       title={A luna$\backslash$'etale slice theorem for algebraic stacks},
        date={2015},
     journal={arXiv preprint arXiv:1504.06467},
}

\bib{alper2019etale}{article}{
      author={Alper, Jarod},
      author={Hall, Jack},
      author={Rydh, David},
       title={The \'etale local structure of algebraic stacks},
        date={2019},
     journal={In preparation},
}

\bib{artinzhang}{article}{
      author={Artin, M.},
      author={Zhang, J.~J.},
       title={Abstract {H}ilbert schemes},
        date={2001},
        ISSN={1386-923X},
     journal={Algebr. Represent. Theory},
      volume={4},
      number={4},
       pages={305\ndash 394},
         url={https://doi.org/10.1023/A:1012006112261},
      review={\MR{1863391}},
}

\bib{Bialynicki-Birula}{article}{
      author={Bialynicki-Birula, A.},
       title={Some theorems on actions of algebraic groups},
        date={1973},
        ISSN={0003486X},
     journal={Annals of Mathematics},
      volume={98},
      number={3},
       pages={480\ndash 497},
         url={http://www.jstor.org/stable/1970915},
}

\bib{BhattHL}{article}{
      author={Bhatt, Bhargav},
      author={Halpern-Leistner, Daniel},
       title={Tannaka duality revisited},
        date={2017},
        ISSN={0001-8708},
     journal={Adv. Math.},
      volume={316},
       pages={576\ndash 612},
         url={https://doi.org/10.1016/j.aim.2016.08.040},
      review={\MR{3672914}},
}

\bib{BFN}{article}{
      author={Ben-Zvi, David},
      author={Francis, John},
      author={Nadler, David},
       title={Integral transforms and {D}rinfeld centers in derived algebraic
  geometry},
        date={2010},
        ISSN={0894-0347},
     journal={J. Amer. Math. Soc.},
      volume={23},
      number={4},
       pages={909\ndash 966},
         url={https://doi.org/10.1090/S0894-0347-10-00669-7},
      review={\MR{2669705}},
}

\bib{Conrad-FormalGAGA}{misc}{
      author={Conard, Brian},
       title={Formal {GAGA} on {A}rtin stacks},
         how={\url{http://math.stanford.edu/~conrad/papers/formalgaga.pdf}},
}

\bib{conrad}{incollection}{
      author={Conrad, Brian},
       title={Reductive group schemes},
        date={2014},
   booktitle={Autour des sch\'{e}mas en groupes. {V}ol. {I}},
      series={Panor. Synth\`eses},
      volume={42/43},
   publisher={Soc. Math. France, Paris},
       pages={93\ndash 444},
      review={\MR{3362641}},
}

\bib{drinfeld2013algebraic}{article}{
      author={Drinfeld, Vladimir},
       title={On algebraic spaces with an action of g\_m},
        date={2013},
     journal={arXiv preprint arXiv:1308.2604},
}

\bib{DrinfeldGaitsgory}{article}{
      author={Drinfeld, Vladimir},
      author={Gaitsgory, Dennis},
       title={On some finiteness questions for algebraic stacks},
        date={2013},
        ISSN={1016-443X},
     journal={Geom. Funct. Anal.},
      volume={23},
      number={1},
       pages={149\ndash 294},
         url={http://dx.doi.org/10.1007/s00039-012-0204-5},
      review={\MR{3037900}},
}

\bib{GaitsgoryIndCoh}{article}{
      author={Gaitsgory, Dennis},
       title={ind-coherent sheaves},
        date={2013},
        ISSN={1609-3321},
     journal={Mosc. Math. J.},
      volume={13},
      number={3},
       pages={399\ndash 528, 553},
         url={https://doi.org/10.17323/1609-4514-2013-13-3-399-528},
      review={\MR{3136100}},
}

\bib{GZB}{article}{
      author={{Geraschenko}, A.},
      author={{Zureick-Brown}, D.},
       title={{Formal GAGA for good moduli spaces}},
        date={2012-08},
     journal={ArXiv e-prints},
      eprint={1208.2882},
}

\bib{HLInstability}{article}{
      author={{Halpern-Leistner}, Daniel},
       title={{On the structure of instability in moduli theory}},
        date={2014Nov},
     journal={arXiv e-prints},
       pages={arXiv:1411.0627},
      eprint={1411.0627},
}

\bib{hall2014coherent}{article}{
      author={Hall, Jack},
      author={Rydh, David},
       title={Coherent tannaka duality and algebraicity of hom-stacks},
        date={2014},
     journal={arXiv preprint arXiv:1405.7680},
}

\bib{HaimanSturmfels}{article}{
      author={{Haiman}, Mark},
      author={{Sturmfels}, Bernd},
       title={Multigraded hilbert schemes.},
        date={2004},
     journal={J. Algebraic Geom.},
}

\bib{Knutson}{book}{
      author={Knutson, Donald},
       title={Algebraic spaces},
      series={Lecture Notes in Mathematics, Vol. 203},
   publisher={Springer-Verlag, Berlin-New York},
        date={1971},
      review={\MR{0302647}},
}

\bib{Lieblich}{article}{
      author={Lieblich, Max},
       title={Remarks on the stack of coherent algebras},
        date={2006},
        ISSN={1073-7928},
     journal={Int. Math. Res. Not.},
       pages={Art. ID 75273, 12},
         url={http://dx.doi.org/10.1155/IMRN/2006/75273},
      review={\MR{2233719 (2008c:14022)}},
}

\bib{DAG-IX}{article}{
      author={Lurie, Jacob},
       title={Derived algebraic geometry {IX}: Closed immersions},
         url={http://www.math.harvard.edu/~lurie/papers/DAG-IX.pdf},
}

\bib{DAG-VIII}{article}{
      author={Lurie, Jacob},
       title={Derived algebraic geometry {VIII}: Quasi-coherent sheaves and
  {T}annaka duality theorems.},
         url={http://www.math.harvard.edu/~lurie/papers/DAG-VIII.pdf},
}

\bib{DAG-XII}{article}{
      author={Lurie, Jacob},
       title={Derived algebraic geometry {XII}: Proper morphisms, completions,
  and the {G}rothendieck existence theorem},
         url={http://www.math.harvard.edu/~lurie/papers/DAG-XII.pdf},
}

\bib{DAG-XIV}{article}{
      author={Lurie, Jacob},
       title={Derived algebraic geometry {XIV}: Representability theorems},
         url={http://www.math.harvard.edu/~lurie/papers/DAG-XIV.pdf},
}

\bib{HigherAlgebra}{article}{
      author={Lurie, Jacob},
       title={Higher algebra},
         url={http://www.math.harvard.edu/~lurie/papers/HigherAlgebra.pdf},
}

\bib{SAG}{article}{
      author={Lurie, Jacob},
       title={Spectral algebraic geometry},
         url={http://www.math.harvard.edu/~lurie/papers/SAG-rootfile.pdf},
}

\bib{Lurie-Thesis}{book}{
      author={Lurie, Jacob},
       title={Derived algebraic geometry},
   publisher={ProQuest LLC, Ann Arbor, MI},
        date={2004},
  url={http://gateway.proquest.com.proxy.library.cornell.edu/openurl?url_ver=Z39.88-2004&rft_val_fmt=info:ofi/fmt:kev:mtx:dissertation&res_dat=xri:pqdiss&rft_dat=xri:pqdiss:0806251},
        note={Thesis (Ph.D.)--Massachusetts Institute of Technology},
      review={\MR{2717174}},
}

\bib{HigherTopos}{book}{
      author={Lurie, Jacob},
       title={Higher topos theory (am-170)},
   publisher={Princeton University Press},
        date={2009},
      volume={189},
}

\bib{LMB00}{book}{
      author={Laumon, G{\'e}rard},
      author={Moret-Bailly, Laurent},
       title={Champs alg\'ebriques},
      series={Ergebnisse der Mathematik und ihrer Grenzgebiete. 3. Folge. A
  Series of Modern Surveys in Mathematics [Results in Mathematics and Related
  Areas. 3rd Series. A Series of Modern Surveys in Mathematics]},
   publisher={Springer-Verlag},
     address={Berlin},
        date={2000},
      volume={39},
        ISBN={3-540-65761-4},
      review={\MR{1771927 (2001f:14006)}},
}

\bib{OlssonProper}{article}{
      author={Olsson, Martin~C.},
       title={On proper coverings of {A}rtin stacks},
        date={2005},
        ISSN={0001-8708},
     journal={Adv. Math.},
      volume={198},
      number={1},
       pages={93\ndash 106},
         url={https://doi.org/10.1016/j.aim.2004.08.017},
      review={\MR{2183251}},
}

\bib{Ol06}{article}{
      author={Olsson, Martin~C.},
       title={{$\underline {\rm Hom}$}-stacks and restriction of scalars},
        date={2006},
        ISSN={0012-7094},
     journal={Duke Math. J.},
      volume={134},
      number={1},
       pages={139\ndash 164},
         url={http://dx.doi.org/10.1215/S0012-7094-06-13414-2},
      review={\MR{2239345 (2007f:14002)}},
}

\bib{popescu}{book}{
      author={Popescu, N.},
       title={Abelian categories with applications to rings and modules},
   publisher={Academic Press, London-New York},
        date={1973},
        note={London Mathematical Society Monographs, No. 3},
      review={\MR{0340375}},
}

\bib{tsd-mf}{article}{
      author={Preygel, Anatoly},
       title={{T}hom-{S}ebastiani and duality for matrix factorzations},
}

\bib{Pridham_moduli}{article}{
      author={Pridham, J.~P.},
       title={Derived moduli of schemes and sheaves},
        date={2012},
        ISSN={1865-2433},
     journal={J. K-Theory},
      volume={10},
      number={1},
       pages={41\ndash 85},
         url={https://doi.org/10.1017/is011011012jkt175},
      review={\MR{2990562}},
}

\bib{Pridham_artin}{article}{
      author={Pridham, J.~P.},
       title={Representability of derived stacks},
        date={2012},
        ISSN={1865-2433},
     journal={J. K-Theory},
      volume={10},
      number={2},
       pages={413\ndash 453},
  url={https://doi-org.proxy.library.cornell.edu/10.1017/is012001005jkt179},
      review={\MR{3004173}},
}

\bib{seshadri}{article}{
      author={Seshadri, C.~S.},
       title={Quotient spaces modulo reductive algebraic groups},
        date={1972},
        ISSN={0003-486X},
     journal={Ann. of Math. (2)},
      volume={95},
       pages={511\ndash 556; errata, ibid. (2) 96 (1972), 599},
         url={https://doi-org.proxy.library.cornell.edu/10.2307/1970870},
      review={\MR{0309940}},
}

\bib{stacks-project}{misc}{
      author={{Stacks project authors}, The},
       title={The stacks project},
         how={\url{https://stacks.math.columbia.edu}},
        date={2019},
}

\bib{HAG-II}{article}{
      author={To\"{e}n, Bertrand},
      author={Vezzosi, Gabriele},
       title={Homotopical algebraic geometry. {II}. {G}eometric stacks and
  applications},
        date={2008},
        ISSN={0065-9266},
     journal={Mem. Amer. Math. Soc.},
      volume={193},
      number={902},
       pages={x+224},
         url={https://doi.org/10.1090/memo/0902},
      review={\MR{2394633}},
}

\bib{Voevodsky-Homology}{article}{
      author={Voevodsky, Vladimir},
       title={Homology of schemes},
        date={1996},
     journal={Selecta Math. (N.S.)},
      volume={2},
      number={1},
       pages={111\ndash 153},
}

\end{biblist}
\end{bibdiv}

\end{document}